\documentclass[reqno,11pt]{amsart}
\usepackage{amsmath,amssymb,latexsym,soul,cite,mathrsfs}
\numberwithin{equation}{section}  
\usepackage{color,enumitem,graphicx}
\usepackage{tikz}
\usetikzlibrary{arrows, patterns}
\usepackage[colorlinks=true,urlcolor=blue,
citecolor=red,linkcolor=blue,linktocpage,pdfpagelabels,
bookmarksnumbered,bookmarksopen]{hyperref}
\usepackage[english]{babel}
\usepackage[left=2.6cm,right=2.6cm,top=2.9cm,bottom=2.9cm]{geometry}
\newtheorem{theorem}{Theorem}[section]
\newtheorem{lemma}[theorem]{Lemma}
\newtheorem{corollary}[theorem]{Corollary}
\newtheorem{proposition}[theorem]{Proposition}

\newtheorem{theoremletter}{Theorem}

\newenvironment{acknowledgement}{\noindent\textbf{Acknowledgments: }}{}
\newtheoremstyle{tttheorem}
{}                
{}                
{\slshape}        
{}                
{\bfseries}       
{'}               
{ }               
{}                
\theoremstyle{tttheorem}

\def\XXint#1#2#3{{\setbox0=\hbox{$#1{#2#3}{\int}$ }
		\vcenter{\hbox{$#2#3$ }}\kern-.6\wd0}}

\title[Multiple Blow-Up Phenomena for $Q$-Curvature in High Dimensions]{Multiple Blow-Up Phenomena for $Q$-Curvature in High Dimensions}

\author[R. Caju]{Rayssa Caju}

\address[R. Caju]{Departamento de Ingenier\'ia Matem\'atica and Centro
de Modelamiento Matem\'atico (CNRS IRL 2807), Universidad de Chile, Beauchef 851, Santiago, Chile}
\email{\href{mailto: rcaju@dim.uchile.cl}{rcaju@dim.uchile.cl}}

\author[A. Silva Santos]{Almir Silva Santos}

\address[A. Silva Santos]{Department of Mathematics, 
	Federal University of Sergipe
	\newline\indent 
	49107-230, S\~ao Cristov\~ao-SE, Brazil}
\email{\href{mailto: almir@mat.ufs.br}{almir@mat.ufs.br}}

\address{Current address: Department of Mathematics, Princeton University, Princeton, NJ, USA, 08544.}
\email{\href{mailto: as1980@princeton.edu}{as1980@princeton.edu}}


\begin{document}

\begin{abstract}
Let $(M,g_0)$ be a closed Riemannian manifold of dimension $n \geq 25$ with positive Yamabe invariant $Y(M,g_0)>0$ and positive fourth-order invariant $Y_4(M,g_0)>0$. We show that, arbitrarily $C^1$-close to $g_0$, there exists a Riemannian metric such that, within its conformal class, one can find infinitely many smooth metrics with the same constant $Q$-curvature and arbitrarily large energy. Moreover, within this conformal class, there exists a sequence of smooth metrics with constant $Q$-curvature equal to $n(n^2-4)/8$ and unbounded volume. This extends to the $Q$-curvature setting the result previously obtained for the scalar curvature in \cite{Marques2015} (see also \cite{MR4937973}). The proof is based on constructing small perturbations of multiple standard bubbles that are glued together.
\end{abstract}

\maketitle

 \tableofcontents
     
\section{Introduction}

On a Riemannian manifold of dimension greater than $2$, a central notion in conformal geometry is the $Q$-curvature, introduced by Branson \cite{MR832360} in the late 1970s and subsequently developed by many authors over the past decades. Together with the conformally covariant fourth-order operator introduced by Paneitz \cite{MR2393291}, now known as the Paneitz operator, the $Q$-curvature plays a role somewhat analogous to that of the Gauss curvature in dimension two and to that of the scalar curvature, as well as the conformal Laplacian, in higher dimensions. It should be noted that in dimension $4$, the $Q$-curvature appears explicitly in the Chern--Gauss--Bonnet formula; see, for instance, \cite{MR2407525,MR3618119}. For a comprehensive overview of the importance of this quantity in differential geometry and its applications, we refer the reader to
\cite{MR2407525,MR1710786}.

Over the past decades, many authors have devoted considerable effort to understanding the geometric and analytic influence of the $Q$-curvature on the underlying manifold. The works \cite{MR3509928,MR3420504,MR3518237,MR3899029,MR1611691,MR2232210,MR3016505,MR2456884,MR1991145,MR4028770,MR4251294,MR4475679,li2025existencecompleteconformalmetrics,MR4923458,li2025conformalmetricsfinitetotal,MR2078747,MR3465087,MR3669775,MazumdarPremoselli2025,gong2025compactness,MR3939772}, among others, provide valuable references on these developments.

In this paper, our goal is to investigate the set of solutions to the constant $Q$-curvature problem in higher dimensions. Let $(M^{n},g_{0})$ be a closed Riemannian manifold of dimension $n \geq 3$. The $Q$-curvature of
$(M^{n},g_{0})$ is defined by
$$
    Q_{g_0}
    = -\frac{1}{2(n-1)}\, \Delta_{g_0} R_{g_0}
      + \frac{n^3 - 4n^2 + 16n - 16}{8(n-1)^2 (n-2)^2}\, R_{g_0}^2
      - \frac{2}{(n-2)^2}\, |{\rm Ric}_{g_0}|^2,
$$
and the Paneitz operator is given by
$$
    P_{g_0} u
    = \Delta_{g_0}^2 u
      + \operatorname{div}_{g_0}\!\left(
          a(n)\, {\rm Ric}_{g_0}(\nabla u, \cdot)
          - b(n)\, R_{g_0}\, du
      \right)
      + c(n)\, Q_{g_0}\, u.
$$
Here, $R_{g_0}$ denotes the scalar curvature, ${\rm Ric}_{g_0}$ the Ricci curvature, and
$\Delta_{g_0}$ the Laplace--Beltrami operator associated with the metric $g_0$. Throughout this
work, we shall use the following dimension-dependent constants:
\[
a(n) = \frac{4}{n - 2},\qquad
b(n) = \frac{(n - 2)^2 + 4}{2 (n - 1)(n - 2)},\qquad
c(n) = \frac{n - 4}{2},\qquad
d(n) = \frac{n (n - 4)(n^{2} - 4)}{16}.
\]

One of the main advantages of the $Q$-curvature is its natural behavior under conformal changes of the metric. More precisely, if $n\neq 4$ and $g$ is a smooth metric conformal to $g_{0}$, written as $g = u^{\frac{4}{n-4}} g_{0}$ with $u \in C^\infty(M)$ and $u>0$, then the $Q$-curvature and the Paneitz operator of $g$ satisfy the following transformation laws
\begin{equation*}\label{eq002} 
   Q_{g}
   = \frac{2}{n-4}\, u^{-\frac{n+4}{n-4}}\, P_{g_0} u\qquad\mbox{ and }\qquad P_{u^{\frac{4}{n-4}}g_0}(v)=u^{-\frac{n+4}{n-4}}P_{g_0}(uv).
\end{equation*}

In the particular case where $n=4$, writing $g = e^{2w} g_0$, the conformal transformation laws for $Q_g$ and $P_g$ are given by
\[
Q_{g}
    = e^{-4w}\left(\tfrac{1}{2} P_{g_0} w + Q_{g_0} \right)\qquad\mbox{ and }\qquad P_{e^{2w}{g_0}}(v)=e^{-4w}P_{g_0}(v).
\]

Inspired by the classical Yamabe problem, the constant $Q$-curvature problem asks whether it is possible to find a metric in the conformal class of $g_0$ with constant $Q$-curvature. By the transformation law for the $Q$-curvature stated above, solving this problem is equivalent to finding a function satisfying a fourth-order partial differential equation.  In dimensions $n \neq 4$, this is equivalent to finding a positive smooth solution $u$ of
\begin{equation}\label{eq105}
P_{g_0} u = \lambda\, u^{\frac{n+4}{n-4}},    
\end{equation}
where $\lambda$ is a constant. In dimension $n = 4$, the problem is equivalent to finding a smooth function $w$ satisfying
\begin{equation}\label{eq106}
    P_{g_0} w + 2 Q_{g_0} = \lambda\, e^{4w}.
\end{equation}

Although one can easily verify that the constant $Q$-curvature problem is variational, and that there has been intense development of the techniques required to treat such problems, establishing existence results has proven to be highly challenging due to its fourth-order nature. In particular, the problem remains open in full generality.

One of the reasons why the equation \eqref{eq105} is not fully understood is the lack of a maximum principle. To the best of our knowledge, the first result in this direction was obtained by Qing and Raske~\cite{MR2232210}. On locally conformally flat manifolds with positive scalar curvature, and under suitable assumptions, they showed that any nontrivial nonnegative solution to \eqref{eq105} must be strictly positive. As a byproduct, in the case $\lambda>0$, they also obtained the existence of positive solutions, as well as a compactness result for the corresponding solution set. See also \cite{MR2119034,MR2558328}. A crucial result was later achieved by Gursky and Malchiodi~\cite{MR3420504}, who proved a maximum
principle for the Paneitz operator under the assumptions that the scalar curvature is nonnegative,
$R_g \geq 0$, and the $Q$-curvature is semipositive, that is, $Q_g \geq 0$ with $Q_g > 0$ somewhere. As a consequence, using a non-local flow, they obtained the existence of positive solutions to \eqref{eq105} for a positive constant $\lambda$. These hypotheses were subsequently improved by Hang and Yang~\cite{MR3518237}; see also \cite{MR3509928}. Their approach relies on the positivity of the Yamabe invariant $Y(M,g)$ (see Section \ref{sec007} for the definition) and on the semi-positivity of the $Q$-curvature. When the constant $\lambda$ is negative, the study of equation \eqref{eq105} becomes delicate. Bettiol, Piccione, and Sire~\cite{MR4251294} observed that even nonisometric conformal metrics with the same constant (negative) $Q$-curvature may exist. This contrasts with the Yamabe problem, where metrics with constant negative scalar curvature are unique within their conformal class.

Due to the different nature of equation \eqref{eq106} in the case $n = 4$, the existence theory for the constant $Q$-curvature problem in this dimension is treated separately. For details in this setting, we refer the reader to~\cite{MR2964636,MR1338677,MR2456884,MR2248155}. We also note that the analysis in dimension $3$ differs substantially from that in higher dimensions; see, for instance, \cite{MR3618119} and the references therein.

As mentioned above, the constant $Q$-curvature problem has a variational structure. For dimensions $n \geq 5$, if we denote by $\mathcal M$ the space of all Riemannian metrics on $M$, the problem is associated with the normalized total $Q$-curvature functional $\mathcal E : \mathcal M \to \mathbb R$, defined by
\[
\mathcal E(g)
    = \operatorname{Vol}(M,g)^{-\frac{n-4}{n}}
      \int_{M} Q_{g}\, dv_{g}.
\]
Restricting the functional $\mathcal E$ to the conformal class $[g] := \{u^{\frac{4}{n-4}} g : u \in C^\infty(M),\ u>0\}$, we obtain the energy functional $\mathcal E_g : [g] \to \mathbb R$, associated with the PDE \eqref{eq105}, given by
\begin{equation}\label{eq101}
    \mathcal E_g(u)
    = \mathcal E\!\left(u^{\frac{4}{n-4}}g\right)
    = \frac{2}{n-4}\,
      \|u\|_{L^{\frac{2n}{n-4}}(M,g)}^{-2}
      \langle P_g u, u \rangle_{L^2},
\end{equation}
where 
\begin{equation}\label{eq023}
    \langle P_g u, u \rangle_{L^2}
    := \int_M \left( (\Delta_g u)^2
        - a(n)\, \operatorname{Ric}_g(\nabla_g u, \nabla_g u)
        + b(n)\, R_g\, |\nabla_g u|^2
        + c(n)\, Q_g\, u^2
    \right) dv_g.
\end{equation}

Gursky, Hang, and Lin \cite{MR3509928} introduced the following conformal invariants in the fourth-order context
$$Y_4^+(M,g)=\inf\{\mathcal E_g(u):u\in C^\infty(M),\; u>0\}$$
and 
\begin{equation}\label{eq075}
    Y_4(M,g)=\inf\{\mathcal E_g(u):u\in H^2(M)\backslash\{0\}\}.
\end{equation}

As observed in \cite{MR3509928}, a strict inequality may occur in general, due to the fourth-order nature of the Paneitz operator. By standard elliptic theory, we have $Y_4(M,g) > 0$ if and only if the first eigenvalue $\lambda_1(P_g)$ is positive, that is, if and only if $P_g$ is positive definite. When the Yamabe invariant $Y(M,g)$ is positive, the authors of \cite{MR3509928} introduced another conformal invariant, denoted by $Y_4^*(M,g)$, and defined by
$$Y_4^*(M,g)=\inf\{\mathcal E(\widetilde g):\widetilde g\in[g]\mbox{ and }R_{\widetilde g}>0\}.$$
Clearly, $Y_4(M,g)\leq Y_4^+(M,g)\leq Y_4^*(M,g)$.

The main result in \cite{MR3509928} states that on any closed Riemannian manifold $(M,g)$ of dimension at least $6$, if $Y(M,g) > 0$ and $Y_4^*(M,g) > 0$, then there exists a metric within the conformal class of $g$ whose scalar curvature and $Q$-curvature are both positive. In particular, they obtained a positive solution to \eqref{eq105} with constant $\lambda > 0$, and showed that in this case $Y_4(M,g) = Y_4^{+}(M,g) = Y_4^{*}(M,g)$. The method applied in \cite{MR3509928} is the method of continuity, and the restriction on the dimension appears in both the open and closed parts of the argument. It is expected that this result should hold in dimension $5$, however, to the best of our knowledge, this remains an open question.

At this point, it is important to highlight that the results obtained in \cite{MR3509928,MR3420504} play a crucial role in our argument, particularly due to the fourth-order nature of the problem. We use the former to conformally deform the background metric to one with positive scalar curvature and positive $Q$-curvature, while the maximum principle established by Gursky and Malchiodi~\cite{MR3420504} is employed to guarantee the positivity of the resulting solution.


\subsection{Compactness for \texorpdfstring{$Q$}{Lg}-curvature and the main result}

In light of the significant advances in the existence theory for the $Q$-curvature equation \eqref{eq105}, in parallel with the Yamabe problem, a natural question is to describe the full set of positive solutions to this problem. The first result in this direction was obtained by Hebey and Robert~\cite{MR2119034} in the locally conformally flat setting, assuming the Paneitz operator is of {\it strong positive type}. In the same setting, Qing and Raske~\cite{MR2232210} established compactness under the assumptions that $(M,g)$ is not conformal to the round sphere, that $Y(M,g)>0$ and $Y_4^+(M,g)>0$, and that the Poincaré exponent is below $(n-4)/2$.

 Later, inspired by the ideas developed for the scalar curvature counterpart in \cite{MR2425176,MR2472174}, Wei and Zhao~\cite{MR3016505} showed that compactness for the constant $Q$-curvature problem fails in dimensions $n \geq 25$, the same threshold as in the Yamabe problem. They constructed a metric on $\mathbb{S}^n$ that admits an $L^\infty$-unbounded family of solutions to \eqref{eq105} with $\lambda > 0$. The general idea is to look for positive solutions that are small perturbations of the standard bubble. To overcome the lack of a maximum principle, they introduced a weighted $L^\infty$ norm to ensure that, if the error term is sufficiently small, then the perturbation remains positive everywhere.

Afterwards, Li and Xiong~\cite{MR3899029} investigated the compactness problem for equation \eqref{eq105}. For $\lambda < 0$, they proved that the set of solutions is compact in the $C^{4}$ topology in all dimensions $n \geq 5$, without any extra assumption. For $\lambda > 0$, assuming that the Riemannian manifold is not conformally equivalent to the round sphere, that the kernel of the Paneitz operator is trivial, and that its Green's function is positive, they proved compactness in the $C^{4}$ topology under any of the following additional assumptions:
\begin{itemize}
    \item the first eigenvalue of the conformal Laplacian is positive and $(M,g)$ is locally conformally flat or $n=5,6,7$;
    \item $5 \leq n \leq 9$ and the positive mass theorem holds for the Paneitz operator;
    \item $n \geq 8$ and the Weyl tensor does not vanish anywhere.
\end{itemize}

Following Schoen's outline of the proof of compactness to the Yamabe problem, Li \cite{MR4028770} established $C^{4,\alpha}$-compactness in dimensions $5 \leq n \leq 7$ under the assumptions that $R_g \geq 0$, $Q_g \geq 0$ with $Q_g > 0$ at some point, and that $(M,g)$ is not conformally equivalent to the round sphere. Very recently, Gong, Kim, and Wei~\cite{gong2025compactness} proved compactness for Riemannian manifold not conformally equivalent to the round sphere in dimensions $5 \leq n \leq 24$, under the assumption that $Q_g\geq 0$ and $Q_g>0$ somewhere, and $Y(M,g)>0$. Their work also addresses a sixth-order conformally invariant equation, and they show that the behavior of the dimension threshold is quite different in this setting. In fact, they proved that in the sixth-order case compactness holds in dimensions $7 \leq n \leq 26$, whereas a blow-up example exists for all $n \geq 27$. See also the compactness result in \cite{MR2248155} for equation \eqref{eq106}, and \cite{MazumdarPremoselli2025} for a compactness result for higher-order $Q$-curvature.

It is noteworthy that in \cite{MR3518237} the authors proved a $C^\infty$ compactness result, in all dimensions $n \geq 5$, for the set of minimizers $u$ of \eqref{eq101}, under the assumptions that $Y(M,g) > 0$, $Y_4(M,g) > 0$, $Q_g$ is semipositive, and that $(M,g)$ is not conformally diffeomorphic to the round sphere.


Motivated by these observations, we now turn to our main result. We will normalize the constant $Q$-curvature to be that of the round sphere $\mathbb S^{n}$, which is equal to $n(n^{2}-4)/8$. In this case the constant in \eqref{eq105} is $\lambda=d(n)$. Denote the set $\mathfrak{M}_g=\left\{\widetilde g\in[g]: Q_{\widetilde g}=n(n^2-4)/8\right\}$.

Our main result in this paper extends the results obtained in \cite{Marques2015}, providing valuable additional information about the full set of solutions to the $Q$-curvature problem. It reads as follows:

\begin{theoremletter}\label{teo008}
    Let $(M^n,g_0)$ be a closed Riemannian manifold of dimension $n\geq 25$ satisfying $Y(M,g_0)>0$ and $Y_4(M,g_0)>0$. Then for any $\varepsilon>0$ there exists a smooth metric $g$ with $\|g-g_0\|_{C^1(M,g_0)}<\varepsilon$, such that the set $\{\overline g\in\mathfrak{M}_{g}:\mathcal E(\overline g)\geq \ell\}$ is infinite for all $\ell\in\mathbb N$. Moreover, there exists a sequence of smooth metrics $(g_k)$ conformal to $g$ such that $Q_{g_k}=n(n^2-4)/8$ and $\operatorname{Vol}(M,g_k)\to \infty$ as $k\to\infty$.
\end{theoremletter}

We remark that, as in the scalar curvature setting \cite{Marques2015}, it is not possible to improve Theorem \ref{teo008} to achieve $C^2$-closeness. In fact, if we consider a Riemannian manifold $(M,g_0)$ whose Weyl tensor is nonvanishing everywhere, then any metric $g$ sufficiently close to $g_0$ in the $C^2$ topology also has a Weyl tensor that is nonvanishing everywhere. Using the result in \cite{MR3509928}, we obtain that the metric satisfies the hypotheses of \cite[Theorem 1.1]{MR3899029}, which would imply that $\mathfrak{M}_{g}$ is compact.

\subsection{Background on the Yamabe problem}\label{sec007}

The constant $Q$-curvature problem is a fourth-order analogue of the well-known Yamabe problem. Given a closed Riemannian manifold $(M,g)$ of dimension $n \geq 3$, the Yamabe problem asks whether it is possible to find a conformal metric $\widetilde g$ with constant scalar curvature. If one writes $\widetilde g = u^{\frac{4}{n-2}} g$, the existence of such a metric reduces to finding a positive solution $u$ of the equation
\begin{equation}\label{eq107}
    L_g(u) = R_{\widetilde g}\, u^{\frac{n+2}{n-2}},
\end{equation}
where $L_g = -\frac{4(n-1)}{n-2}\,\Delta_g + R_g$ is the so-called conformal Laplacian. The affirmative answer to the Yamabe problem was established through the combined works of Yamabe \cite{MR125546}, Trudinger \cite{MR240748}, Aubin \cite{MR431287}, and Schoen \cite{MR788292}. For a comprehensive discussion of the problem, we refer the reader to \cite{MR888880}.

This is a variational problem, and the basic idea of the proof is to show that a minimizer for the corresponding functional exists. This is equivalent to showing that the Yamabe invariant, defined by
\[
Y(M,g)
    = \inf_{\widetilde g \in [g]}
      \mathcal{Y}(\widetilde g),
\]
is achieved,
where $\mathcal Y(\widetilde g)$ is the Yamabe energy of $\widetilde g$ given by $\mathcal Y(\widetilde g)=\operatorname{Vol}(M,\widetilde g)^{-\frac{n-2}{n}}      \int_M R_{\widetilde g}\, dv_{\widetilde g}$. Clearly, by definition, $Y(M,g)$ is a conformal invariant. 

For a conformal class with a nonpositive Yamabe invariant, it is well known that the Yamabe problem admits a unique solution among metrics of unit volume. A natural question, then, is how the set of solutions to \eqref{eq107} behaves when $Y(M,g) > 0$. In a topics course at Stanford in 1988, Schoen formulated the {\it Compactness Conjecture}, which asserts that the set of solutions to the Yamabe problem is compact, provided the manifold is not conformally equivalent to the standard sphere; see \cite{MR1144528,MR1173050}. The round sphere is special because its group of conformal transformations is non-compact. In \cite{MR1173050}, Schoen proved the compactness conjecture for every locally conformally flat manifold that is not conformally diffeomorphic to the round sphere. He also suggested a strategy to establish compactness in the non-locally conformally flat setting.

Over the years, several partial but important results were obtained, providing affirmative answers to the compactness conjecture in various settings, either under low-dimensional assumptions or under additional hypotheses in higher dimensions; see \cite{MR2041549,MR2197144,MR2164927,MR2309836}. The compactness conjecture was affirmatively resolved in the general case by Khuri, Marques, and Schoen \cite{MR2477893}, provided the dimension satisfies $n\leq 24$. Their approach was based on the {\it Weyl Vanishing Conjecture}, which they proved to hold in these dimensions. 

Surprisingly, Brendle \cite{MR2425176} has constructed examples of Riemannian metrics on spheres of dimension at least 52 for which the compactness statement fails. In a subsequent paper, Brendle and Marques \cite{MR2472174} extended these examples to the dimensions $25\leq n\leq 51$. In \cite{MR2551136}, Marques extended the method from \cite{MR2425176} to show that the Weyl Vanishing Conjecture fails in all dimensions greater than $24$. 

The metrics constructed in \cite{MR2425176,MR2551136,MR2472174} have constant scalar curvature, and their Yamabe energies are smaller than the Yamabe invariant of the round sphere, $Y(\mathbb{S}^n, g_{\rm sph})$. In 1987, Kobayashi \cite{MR919505} proved the existence of metrics within any conformal class with positive Yamabe invariant, whose Yamabe energies can be arbitrarily large and whose scalar curvatures can be made arbitrarily close to a constant.  Pollack in \cite{MR1266473} constructed metrics with constant scalar curvature and arbitrarily large Yamabe energies. It is important to note that both results hold for any dimension $n \geq 3$, and the metrics constructed by Pollack are not within the conformal class of the background metric.

Later, Berti and Malchiodi \cite{MR1814428} extended the method developed in \cite{MR1719213} to prove the existence of $C^k$ metrics on $\mathbb{S}^n$, arbitrarily $C^k$-close to the round metric, with $n \geq 4k+1$, for which the compactness conjecture fails.

Finally, Marques \cite{Marques2015} extended the method developed in \cite{MR2425176,MR2472174} to construct Riemannian metrics with arbitrarily finitely many blow-up points, provided the Riemannian manifold has positive Yamabe invariant and dimension at least $25$. In particular, there exist metrics with constant scalar curvature and both arbitrarily large Yamabe energy and volume. Very recently, Gong and Li \cite{MR4937973} used a different method to construct a metric on $\mathbb{S}^n$ that contains a sequence of metrics within its conformal class with constant scalar curvature and unbounded volume, provided $n \geq 25$.

\subsection{Strategy of the proof}

As in the previously mentioned works dealing with noncompactness, our strategy follows the outline introduced by \cite{MR2425176,MR2472174}, where counterexamples were constructed in the sphere by perturbing one single standard bubble. 

 Following the ideas developed in \cite{Marques2015}, our construction looks for positive solutions that are small perturbations of \emph{multiple standard bubbles}, meaning that we cut and glue finitely many standard bubbles along disjoint balls. The overall idea is to apply a Lyapunov-Schmidt type argument, under the assumption that the conformal invariants $Y(M,g_0)$ and $Y_{4}(M, g_{0})$ are positive. 
 
The first step of the proof is to find a $C^1$-close metric $g_s$ that preserves the sign of the conformal invariants and is conformally flat in some geodesic ball, where the bubbles are localized. The next step then is to reduce the search of solutions to finding critical points of a certain energy functional $\mathcal{F}_g: \mathbb{R}^n \times (0,\infty)\to\mathbb{R}$, and then locate the critical points of $\mathcal F_g$ by studying the critical points of an auxiliary functional $F$, where $g$ is a certain perturbation of $g_s$ in the geodesic ball. For a suitable choice of parameters, the functional $F$ is sufficiently close to $\mathcal F_g$, allowing the transfer of information between them. At this stage, we managed to guarantee that the multiple bubbles remain non-interacting while we estimate both the energy and the reduced energy functional, thereby obtaining the required bounds in our analysis.

Since we are dealing with a fourth-order problem, several additional computational challenges arise in the setting of the $Q$-curvature. For instance, the associated energy functional $\mathcal{F}_g(\xi, \varepsilon)$ is substantially more difficult to analyze. In this part, we rely on the machinery already developed by Wei and Zhao in \cite{MR3016505}, who constructed a counterexample using a single standard bubble.

Another crucial part of the construction is to prove that the critical point obtained is positive. A novelty of our contribution, compared to Wei and Zhao, is that, thanks to the maximum principle established by Gursky and Malchiodi \cite{MR3420504} and the results of \cite{MR3509928,MR3016505}, we can guarantee the positivity of the solution. This contrasts with \cite{MR3016505}, which employs a weighted $L^{\infty}$-norm to control the perturbative term, ensure its smallness, and therefore ensure positivity of the solution they construct.

\subsection{Organization of the paper}
In Section \ref{sec002}, we provide some preliminaries by constructing a $C^1$-close metric that is conformally flat in a sufficiently small geodesic ball. We also define our approximate solution by gluing together multiple bubbles. In Section \ref{sec003}, we perform a Lyapunov-Schmidt reduction to construct a solution to \eqref{eq105} as a small perturbation of the approximate solution. In Section \ref{sec004}, we establish some results in $\mathbb{R}^n$ that will be necessary in the subsequent section. In Section \ref{sec006}, we introduce the perturbed metric and define the set of parameters, then derive an expansion for the energy functional and define the auxiliary energy functional. Finally, in Section \ref{sec008}, we prove the main theorem of this work.
\\

\acknowledgement{RC is supported by the Centro de Modelamiento Matemático (CMM), BASAL project FB210005 of ANID–Chile, and by Fondecyt grant \# 11230872. ASS is supported by CNPq grants 408834/2023-4, 312027/2023-0, 444531/2024-6, 403770/2024-6, 400078/2025-2 and FAPITEC/SE: 019203.01303/2024-1. This work was completed while ASS was a Visiting Fellow at Princeton University, and he thanks Fernando Codá Marques for his warm hospitality.}

\section{Preliminaries}\label{sec002}

In this preliminary section, we first describe the perturbation introduced in \cite{MR788292} (see also \cite{Marques2015,MR1266473}). The goal is to construct a metric, $C^1$-close to the background metric, which is conformally flat in a neighborhood of a fixed point. We then introduce the approximate solution and establish an $L^{\frac{2n}{n+4}}$ estimate. 

By combining the results from \cite{MR3509928} with the assumptions $Y(M,g_0)>0$ and $Y_4(M,g_0)>0$, we may henceforth assume throughout this work that the background metric $g_0$ has positive scalar curvature and positive $Q$-curvature. Moreover, throughout this work, we will assume that $n \geq 25$, unless explicitly stated otherwise. Different constants will be denoted by the letters $c$ or $C$, possibly even within the same line.

\subsection{Change of the metric in a geodesic ball}\label{sec005}

Fix a point $p \in M$ and consider polar normal coordinates $(r,\theta)$ centered at $p$ on the geodesic ball $B_{2s}(p)$, for some $s>0$.  In these coordinates, the background metric $g_0$ takes the form $g_0 = dr^2 + r^2 h(r,\theta),$ where, for each $r \ge 0$, $h(r,\theta)$ is a Riemannian metric on $\mathbb{S}^{n-1}$.  It is well known that if $h_0$ denotes the standard round metric on $\mathbb{S}^{n-1}$, then $h(0,\theta) = h_0(\theta)$ and $\partial_r h(0,\theta) = 0$. Moreover, the metric $dr^2 + r^2 h_0$ corresponds to the Euclidean metric $g_{\operatorname{euc}}$. 

Let $\eta:\mathbb R\to[0,1]$ be a non-increasing smooth function such that
\begin{equation}\label{eq096}
    \eta(t)=\begin{cases}
    1, & t\leq 1\\
    0, & t\geq 2
\end{cases}
\end{equation}
and \(|\eta^{(i)}(t)| \leq ct^{-i}\) for all \(i \in \{1,\ldots,4\}\), for some constant $c>0$. Given $s>0$ and a point $q \in M$, define the smooth cut-off function $\eta_{(s,q)}(x):=\eta\!\left( d_g(q,x)/s \right)$. 

We now define a perturbed smooth metric $g_s$ on $M$ as follows.  
Set $g_s = g_0$ on $M \setminus B_{2s}(p)$, and inside $B_{2s}(p)$ write $g_s = dr^2 + r^2 \big( \eta_{(s,p)}\, h_0 + (1 - \eta_{(s,p)})\, h \big)$. It follows immediately that
\begin{equation}\label{eq035}
    g_s=\begin{cases}
    g_{\operatorname{euc}} & \mbox{in } B_s(p),\\
    g_0 & \mbox{in }M\backslash B_{2s}(p),
\end{cases}
\end{equation}
and additionally
\begin{equation}\label{eq033}
    \|g_0-g_s\|_{C^i(M,g_0)}\leq Cs^{2-i},\quad i=0,1,2,
\end{equation}
for some positive constant $C$ depending only on $g_0$. Since the Ricci curvature and the scalar curvature depend only up to the second derivatives of the metric, it follows that $|\operatorname{Ric}_{g_s}|$ and $|R_{g_s}|$ remain uniformly bounded. On the other hand, the $Q$-curvature involves derivatives of the metric up to fourth order. By construction, we have $Q_{g_s} \equiv Q_{g_0}$ on $M \setminus B_{2s}(p)$, while $Q_{g_s} \equiv 0$ in $B_s(p)$. Moreover, on the annulus $B_{2s}(p)\setminus B_s(p)$ the only contribution to the variation of $Q_{g_s}$ comes from derivatives of the cut-off function $\eta_{(s,p)}$, which satisfy estimates of the form $|\partial^i \eta_{(s,p)}| \le C s^{-i}$ for $i \in \{1,\ldots,4\}$.  Hence, $|Q_{g_0}(q) - Q_{g_s}(q)| \leq C s^{-2}$ for all $q \in B_{2s}(p)\setminus B_s(p)$. This implies that
\begin{equation}\label{eq034}
    \lim_{s\to 0}\|Q_{g_0}-Q_{g_s}\|_{L^{\frac{n}{4}}(M,g_0)}=0.
\end{equation}

\begin{proposition}\label{propo001}
    As $s \to 0$, the perturbed metrics $g_s$ satisfy the following convergence properties:
\begin{enumerate}
    \item[(a)] $Y(M,g_s) \longrightarrow Y(M,g_0)$;
    \item[(b)] $Y_4(M,g_s) \longrightarrow Y_4(M,g_0)$;
    \item[(c)] $Y_4^{+}(M,g_s) \longrightarrow Y_4^{+}(M,g_0)$.
\end{enumerate}
\end{proposition}
\begin{proof}
The convergence of the Yamabe invariant stated in item (a) is established in \cite[Proposition~2.1]{Marques2015}. We now proceed to prove item~(b).

By \eqref{eq035} and \eqref{eq033}, there exists  $d_s \to 0$ as $s \to 0$ such that, for every smooth function $u$, we have
\begin{align*}
    &\int_M \Big((\Delta_{g_s} u)^2- a(n)\, \operatorname{Ric}_{g_s}(\nabla_{g_s} u,\nabla_{g_s} u)+ b(n)\, R_{g_s}\, |\nabla_{g_s} u|_{g_s}^2\Big) dv_{g_s}
\\
    &= (1 + d_s)\int_M \Big((\Delta_{g_0} u)^2- a(n)\, \operatorname{Ric}_{g_0}(\nabla_{g_0} u,\nabla_{g_0} u)+ b(n)\, R_{g_0}\, |\nabla_{g_0} u|_{g_0}^2\Big) dv_{g_0}.
\end{align*}

Also, there exists a function $h_s$ such that $dv_{g_s} = (1+h_s)\, dv_{g_0}$, where $\overline{h}_s := \sup_M |h_s| \to 0$ as $s \to 0$.  Thus, using H\"older's inequality and the fact that  $Q_{g_s}$ is uniformly bounded in $L^{\frac{n}{4}}(M,g_0)$, we obtain
\begin{align*}
    \langle P_{g_s} u , u \rangle_{L^2} & = (1+d_s)\, \langle P_{g_0} u , u \rangle_{L^2}+ c(n)\!\int_M \! Q_{g_s} u^2 (1+h_s)\, dv_{g_0} - c(n)(1+d_s)\!\int_M \! Q_{g_0} u^2\, dv_{g_0}\\
    &\ge (1+d_s)\, \langle P_{g_0} u , u \rangle_{L^2} + c(n)\!\int_M (Q_{g_s}-Q_{g_0})\,u^2\, dv_{g_0} - c(n)\overline{h}_s\!\int_M Q_{g_s}u^2\, dv_{g_0}\\
    &\qquad - c(n)d_s \int_M Q_{g_0}u^2\, dv_{g_0}\\
    &\ge (1+d_s)\, \langle P_{g_0}u , u\rangle_{L^2} - C(n,g_0)\!\left(\|Q_{g_s}-Q_{g_0}\|_{L^{\frac{n}{4}}(M,g_0)} + \overline{h}_s + d_s\right) \|u\|_{L^{\frac{2n}{n-4}}(M,g_0)}^{2}.
\end{align*}

This implies that
\begin{align*}
    \frac{\|u\|_{L^{\frac{2n}{n-4}}(M,g_s)}^{2}}
         {\|u\|_{L^{\frac{2n}{n-4}}(M,g_0)}^{2}}
    \, \mathcal{E}_{g_s}(u)
    \;\ge\;
    (1+d_s)\, \mathcal{E}_{g_0}(u)
    - C(n,g_0)\!\left(
        \|Q_{g_s}-Q_{g_0}\|_{L^{\frac{n}{4}}(M,g_0)}
        + \overline{h}_s + d_s
      \right).
\end{align*}

On the other hand, since
\[
    \frac{\|u\|_{L^{\frac{2n}{n-4}}(M,g_s)}^{2}}
         {\|u\|_{L^{\frac{2n}{n-4}}(M,g_0)}^{2}}
    \le (1+\overline{h}_s)^{\frac{n-4}{n}},
\]
we obtain
\begin{equation*}\label{eq025}
    (1+\overline{h}_s)^{\frac{n-4}{n}}\, Y_4^+(M,g_s)
    \;\ge\;
    (1+d_s)\, Y_4^+(M,g_0)
    - C\!\left(
        \|Q_{g_s}-Q_{g_0}\|_{L^{\frac{n}{4}}(M,g_0)}
        + \overline{h}_s + d_s
      \right).
\end{equation*}

Therefore,
\[
    \liminf_{s\to 0} Y_4^+(M,g_s) \;\ge\; Y_4^+(M,g_0).
\]

Similarly, one shows that
\[
    a_s\, Y_4^+(M,g_0)
    \;\ge\;
    b_s\, Y_4^+(M,g_s) + c_s,
\]
with $a_s \to 1$, $b_s \to 1$, and $c_s \to 0$ as $s \to 0$.  
Hence,
\[
    \limsup_{s\to 0} Y_4^+(M,g_s)
    \;\le\;
    Y_4^+(M,g_0),
\]
which concludes the proof of item (b). The proof of item (c) is analogous.
\end{proof}

For $s>0$ sufficiently small, Proposition~\ref{propo001} ensures that the metric $g_s$ satisfies $Y(M,g_s)>0$  and $Y_4(M,g_s)>0,$ provided that $Y(M,g_0)>0$ and $Y_4(M,g_0)>0$.

\subsection{Approximate Solution: The \texorpdfstring{$\ell$}{Lg}-bubbles}\label{sec001}

Suppose that $n\geq 5$. By the classical work of Lin \cite{MR1611691}, for every pair $(\xi,\lambda)\in\mathbb{R}^n\times(0,\infty)$, the family of functions
\begin{equation}\label{eq043}
w_{(\xi,\lambda)}(x) := \left(\frac{2\lambda}{\lambda^{2}+|x-\xi|^{2}}\right)^{\frac{n-4}{2}} = \lambda^{\frac{4-n}{2}}\,w_{0}\!\left(\frac{x-\xi}{\lambda}\right),
\end{equation}
where
\begin{equation}\label{eq047}
w_{0}(x):=\left(\frac{2}{1+|x|^{2}}\right)^{\frac{n-4}{2}},
\end{equation}
are solutions of the fourth-order equation
\begin{equation}\label{eq012}
\Delta^{2} w=d(n)\,w^{\frac{n+4}{n-4}}
\quad\text{in }\mathbb{R}^{n}.
\end{equation}

Moreover, it was shown in \cite{MR1611691} that every positive solution \( w \in H^{2}(\mathbb{R}^{n}) \) of \eqref{eq012} is of this form, for some pair \((\xi, \lambda)\). In other words, the family \eqref{eq043} characterizes all positive entire solutions, up to the natural translations and dilations. The function $w_{(\xi,\lambda)}$ is commonly referred to as a \emph{bubble}, and it satisfies  
\begin{equation}\label{eq005}
\int_{\mathbb R^n} w_{(\xi,\lambda)}^{\frac{2n}{n-4}}(x)\, dx  
= \left( \frac{8\, Y_4^+(\mathbb S^n, g_{\mathrm{can}})}{n(n^2-4)} \right)^{\frac{n}{4}}.    
\end{equation}

For each \(t>0\), we define 
\begin{equation}\label{eq009} 
\overline w_{( \xi,\lambda, t)}(x)=\eta_{(t,\xi)}(x) w_{( \xi,\lambda)}(x),
\end{equation}
where $\eta_{(t,\xi)}(x)=\eta\!\left(|x-\xi|/t\right)$ and $\eta$ is the cut-off function defined in \eqref{eq096}.

Let $(M,g_0)$ be a compact Riemannian manifold with positive scalar curvature  and positive $Q$-curvature.  Consider the perturbed metric $g_s$ defined by \eqref{eq035}, where $s>0$ is chosen smaller than the injectivity radius of $(M,g_0)$.

Let $R \in (0, s/4)$ and fix a positive integer $\ell$.  
Consider $\xi = (\xi_1,\ldots,\xi_\ell)$, $\varepsilon = (\varepsilon_1,\ldots,\varepsilon_\ell)$ and $r = (r_1,\ldots,r_\ell)$ where each $\xi_i \in \mathbb{R}^n$, and the parameters
$\varepsilon_i > 0$ and $r_i > 0$ are sufficiently small for 
$i=1,\ldots,\ell$, subject to the conditions $|\xi_i| < R - 2 r_i$ and $|\xi_i - \xi_j| > 2 (r_i + r_j)$ whenever $i\neq j$. This choice ensures that the open balls 
$B_{2 r_i}(\xi_i)$ and $B_{2 r_j}(\xi_j)$ are pairwise disjoint whenever 
$i \neq j$.  
Moreover, since the support of the function 
$\overline{w}_{(\xi_i,\varepsilon_i,r_i)}$ is contained in 
$B_{2 r_i}(\xi_i)$, we also have $\operatorname{supp}\big( \overline{w}_{(\xi_i,\varepsilon_i,r_i)} \big)    \subset B_{2 r_i}(\xi_i)    \subset B_{R}(0)$, for all $i = 1,\ldots,\ell$. Define the $\ell$-bubble $ W_{(\xi,\varepsilon,r)}\in C^\infty(M)$ by
\begin{equation}\label{eq050}
    W_{(\xi,\varepsilon,r)}(x)    =
    \begin{cases}
        \displaystyle \sum_{i=1}^{\ell} 
        \overline{w}_{(\xi_i,\varepsilon_i,r_i)}(x),
        & \text{if } x\in B_{2R}(p),\\[0.4em]
        0, 
        & \text{if } x \in M \setminus B_{2R}(p).
    \end{cases}
\end{equation}
Here, we consider normal coordinates in $B_{2R}(p)$ with respect to the metric $g_0$.

For $\alpha>0$ and $r=(r_1,\ldots,r_\ell)$, it will be convenient to consider 
the open set of parameters $\mathcal{D}_{(\alpha,r)} \subset 
(\mathbb{R}^{n})^\ell\times (0,\infty)^{\ell}$, defined by
\begin{align}\label{eq011}
    \mathcal{D}_{(\alpha,r)}
    :=
    \Big\{
        (\xi,\varepsilon):
        \ & \frac{\varepsilon_i}{r_i} < \alpha,
        \quad 
        |\xi_i - \xi_j| > 2(r_i + r_j)
        \ \text{for } i\neq j,
\\[-0.2em]
        & |\xi_i| < R - 2 r_i,
        \qquad
        \frac{1}{2} < \frac{\varepsilon_i}{\varepsilon_j} < 2
        \ \text{for all } i,j
    \Big\}.\nonumber
\end{align}

Consider now a smooth, trace-free, symmetric two-tensor $h$ on $\mathbb{R}^n$ 
satisfying
\begin{equation}\label{eq014}
    |h(x)| + |\partial h(x)| + |\partial^2 h(x)|
    + |\partial^3 h(x)| + |\partial^4 h(x)|
    \;\leq\; \alpha < 1,
\end{equation}
for all $x \in \mathbb{R}^n$, and such that $h(x)=0$ whenever $|x|\ge R$. We now define a metric $g$ on $M$ by
\begin{equation}\label{eq013}
    g(x)=\begin{cases}
        \exp(h(x)), & x \in B_s(p), \\[0.3em]
        g_s(x),     & x \in M \setminus B_s(p).
    \end{cases}
\end{equation}
It is straightforward to verify that, under the assumptions of Theorem~\ref{teo008}, the conformal invariants $Y(M,g)$, $Y_4^+(M,g)$, and $Y_4^*(M,g)$ remain positive for all sufficiently small $\alpha>0$.

Since \(R \in(0, s/4)\) and the metric \(g_s\) coincides with the Euclidean metric on \(B_s(p)\), it follows that \(g\) is a smooth metric on \(M\).  Moreover, because \(h\) is trace-free, we have \(dv_g = dv_{g_s}\) on \(M\).  In addition, by construction and by \eqref{eq014}, we obtain the estimate $|Q_g - Q_{g_s}|\;\le\;c(n,g_0)\, \alpha$. For this particular choice of metric, the following estimate concerning the \(\ell\)-bubble configuration holds.
\begin{proposition}\label{propo002}
Fix $\ell \in \mathbb{N}$ and $r=(r_1,\ldots,r_\ell)\in \mathbb{R}^\ell$ with $0 < r_i < \min\{1, R/2\}$ for all $i \in \{1,\ldots,\ell\}$. Then there exists a positive constant $c(n,\ell)$ such that, if $(\xi,\varepsilon)\in \mathcal{D}_{(\alpha,r)}$, one has
\[
\left\| P_gW_{(\xi,\varepsilon,r)} \;-\; 
d(n)\, W_{(\xi,\varepsilon,r)}^{\tfrac{n+4}{\,n-4}} \right\|_{L^{\tfrac{2n}{\,n+4}}(M,g)}\leq c(n,\ell)\,\alpha.
\]
\end{proposition}
\begin{proof}
    Fix $(\xi,\varepsilon)\in\mathcal D_{(\alpha,r)}$. Observe first that, since $W_{(\xi,\varepsilon,r)}\equiv 0$ on $M\backslash\bigcup_{i=1}^\ell B_{2r_i}(\xi_i)$, the estimate holds trivially in this region. To estimate in each of the balls \( B_{2r_i}(\xi_i) \), first note that the Paneitz operator can be written as
\begin{equation}\label{eq028}
        P_g u  =  \Delta_g^2 u+a(n)\left\langle\operatorname{Ric}_g, \nabla^2 u\right\rangle-b(n) R_g \Delta_g u  +\frac{6-n}{2(n-1)}\left\langle\nabla R_g, \nabla u\right\rangle+c(n) Q_g u.
    \end{equation}

Observe that in $B_{2r_i}(\xi_i)$ we have $g = \exp(h)$ and $W_{(\xi,\varepsilon,r)} = \overline{w}_{(\xi_i,\varepsilon_i,r_i)}$. Using the expression of the metric, as in \cite[Section~5]{MR4778469} and \cite[Section~4]{MR3016505}, we obtain pointwise estimates. Noting that $W_{(\xi,\varepsilon,r)} = w_{(\xi_i,\varepsilon_i)}$ in $B_{r_i}(\xi_i)$ and that $w_{(\xi_i,\varepsilon_i)}$ is a solution of \eqref{eq012}, we can derive a pointwise estimate of the form
 $$\left| P_gW_{(\xi,\varepsilon,r)}\;-\; 
d(n)\, W_{(\xi,\varepsilon,r)}^{\tfrac{n+4}{\,n-4}} \right|\leq \sum_{j=0}^4f_j|\partial^j w_{(\xi_i,\varepsilon_i,r_i)}|,$$
where $f_j$ are linear combinations of $h$ and its derivatives up to order four. From this we get the estimates in $B_{r_i}(\xi_i)$.

In the remaining annular region $A_i:=B_{2r_i}(\xi_i) \setminus B_{r_i}(\xi_i)$, for all $x\in A_i$ we have $$|\partial^j\overline w_{(\xi_i,\varepsilon_i)}(x)|\leq c(n)\left(\frac{\varepsilon_i}{r_i^2}\right)^{\frac{n-4}{2}}r_i^{-j},\qquad j=0,1,2,3,4.$$ 
This shows that the $L^{\frac{2n}{n+4}}(A_i)$-norm of the terms in 
\eqref{eq028} is bounded, up to a constant, by $\varepsilon_i / r_i$, 
which completes the desired estimate.
\end{proof}

\section{Lyapunov-Schmidt Reduction}\label{sec003}

Fix $\varepsilon>0$ and $\xi\in\mathbb{R}^n$.  
We begin this section by introducing the special family of functions
\begin{equation}\label{eq098}
\varphi_{(\xi,\varepsilon,k)}(x)
    :=
    \frac{2}{n-4}\,\varepsilon\, \partial_k w_{(\xi,\varepsilon)}(x)\,
    w_{(\xi,\varepsilon)}(x)^{\frac{8}{\,n-4\,}},    
\end{equation}
where $\partial_0=\partial_\varepsilon$ and 
$\partial_k=\partial_{\xi_k}$ for $k=1,\ldots,n$.  
Explicitly,
\begin{equation}\label{eq051}
    \varphi_{(\xi,\varepsilon,0)}(x)
    =
    \left(
        \frac{2\varepsilon}{\varepsilon^2 + |x-\xi|^2}
    \right)^{\!\frac{n+4}{2}}
    \frac{|x-\xi|^2 - \varepsilon^2}{|x-\xi|^2 + \varepsilon^2},
\end{equation}
and, for $k=1,\ldots,n$,
\begin{equation}\label{eq052}
    \varphi_{(\xi,\varepsilon,k)}(x)
    =
    \left(
        \frac{2\varepsilon}{\varepsilon^2 + |x-\xi|^2}
    \right)^{\!\frac{n+4}{2}}
    \frac{2\varepsilon\, (x_k - \xi_k)}{|x-\xi|^2 + \varepsilon^2}.
\end{equation}

By property \eqref{eq005} of the standard bubble, the functions $\varphi_{(\xi,\varepsilon,k)}$ are $L^{2}(\mathbb R^n)$-orthogonal to $w_{(\xi,\varepsilon)}$. Moreover, it is straightforward to verify that $\|\varphi_{(\xi,\varepsilon,k)}\|_{L^{\frac{2n}{n+4}}(\mathbb{R}^n)}$ is independent of both $\xi$ and $\varepsilon$.

Given $(\xi,\varepsilon)\in\mathcal D_{(\alpha,r)}$, define smooth functions 
$\overline{\varphi}_{(\xi_j,\varepsilon_j,r_j,k)} \in C^\infty(M)$ by
\begin{equation}\label{eq010}
    \overline{\varphi}_{(\xi_j,\varepsilon_j,r_j,k)}(x)=\begin{cases}
        \eta_{(r_j,\xi_j)}\varphi_{(\xi_j,\varepsilon_j,k)},& x \in B_{R}(p),\\[0.6em]
        0, & x \in M\setminus B_{R}(p),
    \end{cases}
\end{equation}
where $\eta_{(r_j,\xi_j)}$ is defined in Section \ref{sec005}. Note that $\operatorname{supp}\overline{\varphi}_{(\xi_j,\varepsilon_j,r_j,k)}    \subset B_{2r_j}(\xi_j)\subset B_{R}(p)$. For $(\xi,\varepsilon)\in\mathcal D_{(\alpha,r)}$, consider the finite set $\mathcal{F}_{(\xi,\varepsilon,\alpha,r)}:=\left\{        \overline{\varphi}_{(\xi_j,\varepsilon_j,r_j,k)}:j=1,\ldots,\ell,\;k=0,\ldots,n\right\}$, and its orthogonal complement
\begin{equation}\label{eq059}
    \mathcal F^\perp_{(\xi,\varepsilon,\alpha,r)}(M,g_s)
    :=
    \left\{
        \omega\in W^{2,2}(M,g_s)
        :
        \int_M \omega\,\varphi\, dv_{g_s}=0
        \ \text{for all }\varphi\in 
        \mathcal F_{(\xi,\varepsilon,\alpha,r)}
    \right\}.
\end{equation}

Since $dv_g = dv_{g_s}$, it immediately follows that $\mathcal F^\perp_{(\xi,\varepsilon,\alpha,r)}(M,g)
    =
    \mathcal F^\perp_{(\xi,\varepsilon,\alpha,r)}(M,g_s)$.
\begin{lemma}\label{lem002}
    If $\alpha\in(0,1)$ is sufficiently small, the set of functions $\mathcal F_{(\xi,\varepsilon,\alpha,r)}$ is linearly independent.
\end{lemma}
\begin{proof}
    Consider a linear combination such that  $\displaystyle\sum_{i=1}^{\ell} \sum_{k=0}^{n} a_{ik}\,
    \overline{\varphi}_{(\xi_i,\varepsilon_i,r_i,k)} = 0$. Then
\begin{equation}\label{eq008}
    \sum_{i=1}^{\ell} \sum_{k=0}^{n} a_{ik}\, \beta_{ikjm} = 0,
\end{equation}
where
\begin{equation}\label{eq085}
    \beta_{ikjm}
    = 
    \varepsilon_j \int_M 
    \overline{\varphi}_{(\xi_i,\varepsilon_i,r_i,k)}\,
    \partial_m \overline{w}_{(\xi_j,\varepsilon_j,r_j)}\, dv_{g_s},
\end{equation}
and $\overline{w}_{(\xi_j,\varepsilon_j,r_j)}$ is defined in \eqref{eq009}.  
We claim that 
\begin{equation}\label{eq072}
\begin{cases}
    \beta_{ikik} \neq 0, & \text{for all } i \text{ and } k, \\[0.2em]
    \beta_{ikjm} = 0, & \text{whenever } (i,k) \neq (j,m).
\end{cases}
\end{equation}

\medskip

Indeed, if $i \neq j$, then by the definition of $\mathcal{D}_{(\alpha,r)}$ (see \eqref{eq011}) we have $|\xi_i - \xi_j| > 2(r_i+r_j)$. Together with \eqref{eq010}, this implies that the supports of $\overline{\varphi}_{(\xi_i,\varepsilon_i,r_i,k)}$ and $\overline{w}_{(\xi_j,\varepsilon_j,r_j)}$ are disjoint.  Hence $\beta_{ikjm} = 0$ whenever $i\neq j$.

Now note that both 
$\overline{\varphi}_{(\xi_i,\varepsilon_i,r_i,k)}$ and $\overline{w}_{(\xi_i,\varepsilon_i,r_i)}$ have support in the ball $B_{2r_i}(\xi_i)$, where the metric $g_s$ is Euclidean in normal coordinates.  After a change of variables and using the symmetry of the domain, we obtain
\[
\beta_{i0im}
= \varepsilon_i \int_M 
    \overline{\varphi}_{(\xi_i,\varepsilon_i,r_i,0)}\,
    \partial_m \overline{w}_{(\xi_i,\varepsilon_i,r_i)}\, dv_{g_s}
= \int_{B_r(0)} f_1(|z|)\, z_m\, dz = 0, \qquad m\neq 0,
\]
\[
\beta_{iki0}
= \varepsilon_i \int_M 
    \overline{\varphi}_{(\xi_i,\varepsilon_i,r_i,k)}\,
    \partial_{\varepsilon_i} \overline{w}_{(\xi_i,\varepsilon_i,r_i)}\, dv_{g_s}
= \int_{B_r(0)} f_2(|z|)\, z_k\, dz = 0, \qquad k\neq 0,
\]
and
\[
\beta_{ikim}
= \varepsilon_i \int_M 
    \overline{\varphi}_{(\xi_i,\varepsilon_i,r_i,k)}\,
    \partial_m \overline{w}_{(\xi_i,\varepsilon_i,r_i)}\, dv_{g_s}
= \int_{B_r(0)} f_3(|z|)\, z_m z_k \, dz = 0,
\]
for any $k\neq m$ with $km\neq 0$, where $f_j$ are radial functions for $j=1,2,3$.

To conclude, we prove the remaining case, namely that $\beta_{ikik} \neq 0$ for all $i$ and $k$. By definition,
\[
\beta_{i0i0}
= \varepsilon_i \int_M
    \overline{\varphi}_{(\xi_i,\varepsilon_i,r_i,0)}\,
    \partial_{\varepsilon_i} \overline{w}_{(\xi_i,\varepsilon_i,r_i)}\, dv_{g_s}
= c(n,\varepsilon)
  \int_{B_{2r_i}(\xi_i)}
      \eta_{(r_i,\xi_i)}^2      (\partial_{\varepsilon_i} w_{(\xi_i,\varepsilon_i)})^2\,      w_{(\xi_i,\varepsilon_i)}^{\frac{8}{n-4}}\, dx\neq 0.
\]
Now, for $k\in\{1,\ldots,n\}$,
\begin{align*}
\beta_{ikik}
&= \frac{2}{n-4}\varepsilon_i^2
   \int_{B_{2r_i}(\xi_i)}
       \eta_{(r_i,\xi_i)}
       w_{(\xi_i,\varepsilon_i)}^{\frac{8}{n-4}}       \partial_k w_{(\xi_i,\varepsilon_i)}       \left[           \eta_{(r_i,\xi_i)}\partial_k w_{(\xi_i,\varepsilon_i)}           + w_{(\xi_i,\varepsilon_i)}\eta_{(r_i,\xi_i)}'\frac{x_k-\xi_k}{r_i|x-\xi_k|}\right] dx \\
&=: \frac{2}{n-4}(I_1 + I_2).
\end{align*}

To estimate $I_1$, we note that the integrand is nonnegative, which yields
\begin{align*}
I_1
&\ge \varepsilon_i^2
    \int_{B_{r_i}(\xi_i)}
        w_{(\xi_i,\varepsilon_i)}^{\frac{8}{n-4}}
        (\partial_k w_{(\xi_i,\varepsilon_i)})^2\, dx 
= \int_{B_{2r_i}(0)}
    \left(\frac{2\varepsilon_i}{\varepsilon_i^2+|x|^2}\right)^{n+2}
    x_k^2\, dx \\
&\ge \int_{B_{r_i/\varepsilon_i}(0)}
    \left(\frac{2}{1+|y|^2}\right)^{n+2}
    y_k^2\, dy 
\ge \int_{B_1(0)}
    \left(\frac{2}{1+|y|^2}\right)^{n+2}
    y_k^2\, dy 
=: c(n)>0.
\end{align*}

Finally to estimate $I_2$, note that the support of $\eta'_{(r_i,\xi_i)}$ is contained in the annulus 
$A_i := B_{2r_i}(\xi_i)\setminus B_{r_i}(\xi_i)$. Hence,
\[
|I_2|\le c\,\varepsilon_i^2r_i^{-1}
    \int_{A_i}
        w_{(\xi_i,\varepsilon_i)}^{\frac{n+4}{n-4}}
        |\partial_k w_{(\xi_i,\varepsilon_i)}|
        \, dx
\le c\,\varepsilon_i^{n+2} r_i^{-n-1}
\le c\,\alpha^{n+2} r_i.
\]
Thus, for $\alpha>0$ sufficiently small, we conclude that 
$\beta_{ikik} \neq 0$. 

Since the claim holds, then combined with \eqref{eq008}, it follows that 
$a_{ik}=0$ for all $i\in\{1,\ldots,\ell\}$ and $k\in\{0,\ldots,n\}$, which implies that the set of functions is linearly independent.
\end{proof}

\begin{lemma}\label{lem003}
    For all \((\xi, \varepsilon) \in \mathcal{D}_{(\alpha, r)}\), \(k \in \{0, \ldots, n\}\), and \(i, j \in \{1, \ldots, \ell\}\), there exists a constant \(c > 0\) such that, for the functions defined in \eqref{eq009} and \eqref{eq010}, we have
\[
\int_M \overline{\varphi}_{(\xi_j,\varepsilon_j,r_j,k)} \, \overline{w}_{(\xi_i,\varepsilon_i,r_i)} \, dv_{g_s} 
\leq c \left( \frac{\varepsilon_i}{r_i} \right)^n.
\]
\end{lemma}
\begin{proof}
    Since $\overline{\varphi}_{(\xi_j,\varepsilon_j,r_j,k)}$ and 
$\overline{w}_{(\xi_i,\varepsilon_i,r_i)}$ have disjoint supports whenever 
$i \neq j$, we may restrict our attention to the case $i = j$.  

If $k \in \{1,\ldots,n\}$, the result follows directly, since, as in the proof of Lemma~\ref{lem002}, we have $\overline{\varphi}_{(\xi_i,\varepsilon_i,r_i,k)}\perp_{L^2(M,g_s)} \overline{w}_{(\xi_i,\varepsilon_i,r_i)}$. Finally, if $k=0$, we use the fact that $\varphi_{(\xi,\varepsilon,0)} \perp_{L^2(\mathbb{R}^n)} w_{(\xi,\varepsilon)}$. A direct computation gives
\begin{align*}
    \int_M \overline{\varphi}_{(\xi_i,\varepsilon_i,r_i,0)}\,             \overline{w}_{(\xi_i,\varepsilon_i,r_i)} \, dv_{g_s} & = \int_{B_{2r_i}(\xi_i)\setminus B_{r_i}(\xi_i)}\overline{\varphi}_{(\xi_i,\varepsilon_i,r_i,0)}\,\overline{w}_{(\xi_i,\varepsilon_i,r_i)}\, dx+ \int_{B_{r_i}(\xi_i)}\varphi_{(\xi_i,\varepsilon_i,0)}\, w_{(\xi_i,\varepsilon_i)} \, dx\\
    &= \int_{\mathbb{R}^n \setminus B_{r_i}(\xi_i)}\eta_{(\xi_i,r_i)}\varphi_{(\xi_i,\varepsilon_i,0)}\, w_{(\xi_i,\varepsilon_i)} \, dx- \int_{\mathbb{R}^n \setminus B_{r_i}(\xi_i)}\varphi_{(\xi_i,\varepsilon_i,0)}\, w_{(\xi_i,\varepsilon_i)} \, dx\\
    &= \int_{\mathbb{R}^n \setminus B_{\frac{r_i}{\varepsilon_i}}(0)}\left(\eta^2\!\left(\frac{\varepsilon_i |y|}{r_i}\right)-1\right)\left( \frac{2}{1+|y|^2} \right)^n\frac{|y|^2 - 1}{\,|y|^2 + 1\,}\, dy,
\end{align*}
and it follows that
\[
    \left|\int_M \overline{\varphi}_{(\xi_i,\varepsilon_i,r_i,0)}\,\overline{w}_{(\xi_i,\varepsilon_i,r_i)} \, dv_{g_s}    \right|\le\int_{\mathbb{R}^n \setminus B_{\,\frac{r_i}{\varepsilon_i}}(0)}\left( \frac{2}{1+|y|^2} \right)^n dy\le C \left( \frac{\varepsilon_i}{r_i} \right)^{n}.
\]
\end{proof}

Consider the bilinear form  $\mathcal{H}_g : W^{2,2}(M,g) \times W^{2,2}(M,g) \to \mathbb{R}$ defined by
\begin{align}\label{eq024}
    \mathcal{H}_g(u,v)
    = \langle P_g u , v \rangle_{L^2}
      - \frac{n+4}{n-4}\, d(n)
        \int_M W_{(\xi,\varepsilon,r)}^{\frac{8}{n-4}}\, u\, v \, dv_g,
\end{align}
where $\langle P_g u , v \rangle_{L^2}$ is given by \eqref{eq023} and 
$W_{(\xi,\varepsilon,r)}$ is the $\ell$--bubble introduced in \eqref{eq050}.  
Since the functions $\overline{w}_{(\xi_i,\varepsilon_i,r_i)}$ have disjoint 
supports (see Section~\ref{sec001}), it follows that for any $q>0$ one has
\begin{equation}\label{eq066}
    W_{(\xi,\varepsilon,r)}^{\,q}
    = \sum_{i=1}^{\ell} 
        \overline{w}_{(\xi_i,\varepsilon_i,r_i)}^{\,q}.
\end{equation}

\begin{lemma}\label{lem006}
Let $g$ be the Riemannian metric defined in \eqref{eq013}.  There exist constants $C_1 = C_1(n,g_0) > 0$ and $C_2 = C_2(n,\ell,g_0) > 0$ such that, for all $u, v \in W^{2,2}(M,g)$, the following estimates hold:
\begin{equation}\label{eq058}
    \big| \mathcal{H}_g(u,v) - \mathcal{H}_{g_s}(u,v) \big|    \;\le\;    C_1 \alpha \,    \|u\|_{W^{2,2}(M,g_s)} \,    \|v\|_{W^{2,2}(M,g_s)},
\end{equation}
\begin{equation}\label{eq063}
    (1 - C_1 \alpha)\, \|u\|_{W^{2,2}(M,g_s)}^{2}    \;\le\;    \|u\|_{W^{2,2}(M,g)}^{2}    \;\le\;    (1 + C_1 \alpha)\, \|u\|_{W^{2,2}(M,g_s)}^{2},
\end{equation}
\begin{equation}\label{eq006}
    |\mathcal{H}_{g_s}(u,v)|    \;\le\;    C \, \|u\|_{W^{2,2}(M,g_s)} \,         \|v\|_{W^{2,2}(M,g_s)},
\end{equation}
and
\begin{equation}\label{eq016}
    |\mathcal{H}_g(u,v)|    \;\le\;    C_2 \,    \|u\|_{W^{2,2}(M,g)} \,    \|v\|_{W^{2,2}(M,g)} .
\end{equation}
\end{lemma}
\begin{proof}
Since $dv_g = dv_{g_s}$, using the definitions \eqref{eq023} and \eqref{eq024}, we obtain
\begin{align*}
    |\mathcal H_g(u,v) - \mathcal H_{g_s}(u,v)|
    &\le 
      \int_M 
      \big| (\Delta_g u)(\Delta_g v) - (\Delta_{g_s} u)(\Delta_{g_s} v) \big|
      \, dv_{g_s}  
      + C \int_M |(Q_g - Q_{g_s})uv| \, dv_{g_s} \\
    &\quad
      + C \int_M 
          \big| 
              \operatorname{Ric}_g(\nabla_g u, \nabla_g v)
              - \operatorname{Ric}_{g_s}(\nabla_{g_s} u, \nabla_{g_s} v)
          \big|
          \, dv_{g_s} \\
    &\quad 
      + C \int_M 
          \big| 
              R_g \langle \nabla_g u, \nabla_g v \rangle
              - R_{g_s} \langle \nabla_{g_s} u, \nabla_{g_s} v \rangle
          \big|
          \, dv_{g_s}.
\end{align*}

Since the metrics $g$ and $g_s$ coincide on $M \setminus B_s(p)$, and using \eqref{eq033}, \eqref{eq013}, together with the fact that $|R_{g_s}|$ and $|\operatorname{Ric}_{g_s}|$ are uniformly bounded, we may find a constant $c>0$, independent of $s$, such that the terms involving the Laplacian, the Ricci tensor, and the scalar curvature are bounded by $\alpha \,\|u\|_{W^{2,2}(M,g_s)} \,  \|v\|_{W^{2,2}(M,g_s)}$, up to a constant which depends only on $n$.

Finally, since $Q_g$ depends on fourth derivatives of the metric, 
we use \eqref{eq033} and \eqref{eq014} to obtain
\begin{align*}
    \int_M |(Q_g - Q_{g_s})\, u v| dv_{g_s}    &\leq    c\, \alpha    \|u\|_{L^{\frac{2n}{\,n-4\,}}(M,g_s)}    \|v\|_{L^{\frac{2n}{\,n-4\,}}(M,g_s)} .
\end{align*}
Together with the Sobolev inequality, this yields our first estimate \eqref{eq058}.

To prove \eqref{eq063}, we use again that $dv_g = dv_{g_s}$, together with 
\eqref{eq033} and \eqref{eq013}, to show that $\|\nabla_g^2 u - \nabla_{g_s}^2 u\|_{L^2(M,g_s)}$ and $\|\nabla_g u - \nabla_{g_s} u\|_{L^2(M,g_s)}$ are bounded by $\|u\|_{W^{2,2}(M,g_s)}$, up to a constant. The desired inequality \eqref{eq063} follows from the triangle inequality.

We now prove \eqref{eq006}.  
Recall the expression of 
$\langle P_{g_s}u, v \rangle_{L^2}$ in~\eqref{eq023}.  
Using~\eqref{eq033} together with H\"older's inequality, 
we find that there exists a constant $C>0$, independent of $s$, such that 
\[
    |\langle P_{g_s} u, v \rangle_{L^2}|    \;\leq\;    C \, \|u\|_{W^{2,2}(M,g_s)} \, \|v\|_{W^{2,2}(M,g_s)}.
\]

    By \eqref{eq033} and \eqref{eq034}, we obtain that $Q_{g_s}$ is uniformly 
bounded in $L^{\frac{n}{4}}(M,g_s)$.  
Thus, using H\"older's inequality together with the Sobolev inequality, we find
\begin{align*}
    \left|\int_M Q_{g_s}\, u v \, dv_{g_s}\right|    & \leq c \, \|u\|_{L^{\frac{2n}{n-4}}(M,g_s)}              \|v\|_{L^{\frac{2n}{n-4}}(M,g_s)}\leq c \, \|u\|_{W^{2,2}(M,g_s)}\|v\|_{W^{2,2}(M,g_s)} .
\end{align*}

Finally, let us estimate the last term in \eqref{eq024}.  
Using \eqref{eq005} and arguing as above, we have
\begin{align*}
\left|\int_M 
    \overline w_{(\xi_i,\varepsilon_i,r_i)}^{\frac{8}{n-4}} \,
    u v \, dv_{g_s}
\right|
&\le 
    \int_{B_{2r_i}(\xi_i)}
    \left|
        w_{(\xi_i,\varepsilon_i)}^{\frac{8}{n-4}}
        u v
    \right| dv_{g_s}  \le
    \|w_{(\xi_i,\varepsilon_i)}\|_{L^{\frac{2n}{n-4}}(\mathbb{R}^n)}^{\frac{8}{n-4}}
    \, \|uv\|_{L^{\frac{n}{n-4}}(M,g_s)} \\
&\le
    C \, \|u\|_{L^{\frac{2n}{n-4}}(M,g_s)}
         \|v\|_{L^{\frac{2n}{n-4}}(M,g_s)} \le
    c \, \|u\|_{W^{2,2}(M,g_s)}
         \|v\|_{W^{2,2}(M,g_s)} .
\end{align*}

This implies \eqref{eq006}.  
The inequality \eqref{eq016} follows directly from 
\eqref{eq058}, \eqref{eq063}, and \eqref{eq006}.
\end{proof}

\begin{lemma}\label{lem004}
Given $(\xi,\varepsilon)\in\mathcal D_{(\alpha,r)}$, for each 
$i \in \{1,\ldots,\ell\}$ there exist a function 
$v_{(\xi_i,\varepsilon_i,r_i)} \in W^{2,2}(M,g_s)$ and a constant 
$C = C(n,\ell) > 0$ such that
\begin{equation}\label{eq017}
    \widetilde w_{(\xi_i,\varepsilon_i,r_i)}
    :=
    \overline w_{(\xi_i,\varepsilon_i,r_i)}
    - v_{(\xi_i,\varepsilon_i,r_i)}
    \in 
    \mathcal F^\perp_{(\xi,\varepsilon,\alpha,r)}(M,g_s),
\end{equation}
\begin{equation}\label{eq019}
    \big\| v_{(\xi_i,\varepsilon_i,r_i)} \big\|_{L^{\frac{2n}{n-4}}(M,g_s)}
    \le 
    C\left( \frac{\varepsilon_i}{r_i} \right)^{n},
\end{equation}
and
\begin{equation}\label{eq053}
    \big\| v_{(\xi_i,\varepsilon_i,r_i)} \big\|_{W^{2,2}(M,g_s)}
    \le 
    C\left( \frac{\varepsilon_i}{r_i} \right)^{n}.
\end{equation}
\end{lemma}
\begin{proof}
We aim to show the existence of constants $c_{ijk}\in\mathbb{R}$ such that we may take
\begin{equation}\label{eq022}
    v_{(\xi_i,\varepsilon_i,r_i)}
    :=
    \sum_{j=1}^{\ell}\sum_{k=0}^{n}
        c_{ijk}\,
        \eta_{(\xi_j,r_j)}
        w_{(\xi_j,\varepsilon_j)}^{-\frac{8}{n-4}}
        \,\overline{\varphi}_{(\xi_j,\varepsilon_j,r_j,k)}.
\end{equation}
With this choice, condition \eqref{eq017} is equivalent to finding constants 
$c_{ijk}$ such that
\begin{equation}\label{eq018}
    \int_M 
        \overline{w}_{(\xi_i,\varepsilon_i,r_i)}\,
        \overline{\varphi}_{(\xi_t,\varepsilon_t,r_t,l)} 
        \, dv_{g_s}
    =
    \sum_{j=1}^{\ell}\sum_{k=0}^{n} 
        c_{ijk}
        \int_M 
            \eta_{(\xi_j,r_j)}
            w_{(\xi_j,\varepsilon_j)}^{-\frac{8}{n-4}}
            \overline{\varphi}_{(\xi_j,\varepsilon_j,r_j,k)}\,
            \overline{\varphi}_{(\xi_t,\varepsilon_t,r_t,l)}
        \, dv_{g_s},
\end{equation}
for all 
\(    \overline{\varphi}_{(\xi_t,\varepsilon_t,r_t,l)}    \in \mathcal{F}_{(\xi,\varepsilon,\alpha,r)}.
\)
To this end, we will prove the following two claims.

\noindent\textbf{Claim 1.}  
    $\displaystyle\int_M\eta_{(\xi_j,r_j)} 
    w_{(\xi_j,\varepsilon_j)}^{-\frac{8}{n-4}}
    \overline{\varphi}_{(\xi_j,\varepsilon_j,r_j,k)}
    \overline{\varphi}_{(\xi_t,\varepsilon_t,r_t,l)}
    \, dv_{g_s}
    = 0$ whenever  $(j,k)\neq (t,l)$.

If \( j\neq t \), then the functions 
\(\overline{\varphi}_{(\xi_j,\varepsilon_j,r_j,k)}\) and 
\(\overline{\varphi}_{(\xi_t,\varepsilon_t,r_t,l)}\)  
have disjoint supports, so the integral vanishes trivially.  
Thus, assume \( j=t \) and \( k\neq l \).  
In this case, the argument is analogous to the proof that 
\(\beta_{i0im}=0\) and \(\beta_{ikim}=0\) in Lemma~\ref{lem002}.  
The orthogonality follows from symmetry considerations after shifting to 
normal coordinates centered at \(\xi_j\) and observing that the integral of an 
odd function on a symmetric ball vanishes.

\noindent\textbf{Claim 2.}  $\displaystyle\int_M \eta_{(\xi_j,r_j)}\,
w_{(\xi_j,\varepsilon_j)}^{-\frac{8}{n-4}}
\overline{\varphi}_{(\xi_j,\varepsilon_j,r_j,k)}^{\,2}\, dv_{g_s}
\;\geq\; c(n) > 0.$

We first consider the case $k=0$. Using that $\eta_{(\xi_j,r_j)}\equiv1$ on $B_{r_j}(\xi_j)$ and applying the change of variables $x \mapsto \varepsilon_j x$, we obtain
\begin{align*}
    \int_M \eta_{(\xi_j,r_j)}
    & w_{(\xi_j,\varepsilon_j)}^{-\frac{8}{n-4}} 
    \overline{\varphi}_{(\xi_j,\varepsilon_j,r_j,0)}^{\,2} \, dv_{g_s}
    \ge 
    \int_{B_{r_j}(\xi_j)}
        w_{(\xi_j,\varepsilon_j)}^{-\frac{8}{n-4}}
        \varphi_{(\xi_j,\varepsilon_j,0)}^{\,2}\, dx \\
    &= 
    \int_{B_{r_j}(0)}
        \left( \frac{2\varepsilon_j}{\varepsilon_j^2 + |x|^2} \right)^{\!n}
        \left( \frac{\varepsilon_j^2 - |x|^2}{\varepsilon_j^2 + |x|^2} \right)^{\!2}
    dx = 
    \int_{B_{\frac{r_j}{\varepsilon_j}}(0)}
        \left( \frac{2}{1+|x|^2} \right)^{\!n}
        \left( \frac{1 - |x|^2}{1+|x|^2} \right)^{\!2}\, dx \\
    &\ge 
    \int_{B_1(0)}
        \left( \frac{2}{1+|x|^2} \right)^{\!n}
        \left( \frac{1 - |x|^2}{1+|x|^2} \right)^{\!2}\, dx
    =: c(n).
\end{align*}

\medskip

The argument for $k\in\{1,\ldots,n\}$ is analogous. We have
\begin{align}
    \int_M \eta_{(\xi_j,r_j)}w_{(\xi_j,\varepsilon_j)}^{-\frac{8}{n-4}}    \overline{\varphi}_{(\xi_j,\varepsilon_j,r_j,k)}^{\,2}\, dv_{g_s}    \ge     \int_{B_1(0)}        \left( \frac{2}{1+|x|^2} \right)^{\!n}        \left( \frac{2x_k}{1+|x|^2} \right)^{\!2} dx
    =: c(n).
    \label{eq020}
\end{align}

From Claims 1 and 2, it follows that we can choose constants 
\(c_{ijk}\) satisfying \eqref{eq018}. Indeed, set
\[
c_{ijk}=\left(    \int_M         \eta_{(\xi_j,r_j)}        w_{(\xi_j,\varepsilon_j)}^{-\frac{8}{n-4}}        \overline{\varphi}_{(\xi_j,\varepsilon_j,r_j,k)}^{\,2}    \, dv_{g_s}\right)^{-1}\int_M     \overline{w}_{(\xi_i,\varepsilon_i,r_i)}\,    \overline{\varphi}_{(\xi_j,\varepsilon_j,r_j,k)}\, dv_{g_0}.
\]
By Lemma~\ref{lem003} and estimate~\eqref{eq020}, we obtain
\begin{equation}\label{eq021}
    |c_{ijk}|
    \,\le\, C \left( \frac{\varepsilon_i}{r_i} \right)^{\!n}.
\end{equation}

Since 
\(
\eta_{(\xi_j,r_j)}
w_{(\xi_j,\varepsilon_j)}^{-\frac{8}{n-4}}
\overline{\varphi}_{(\xi_j,\varepsilon_j,r_j,k)}
\)
has compact support contained in \(B_{2r_j}(\xi_j)\),  
where the metric \(g_s\) agrees with the Euclidean metric,
we may apply~\eqref{eq051} and~\eqref{eq052} to obtain
\[
\left\|\eta_{(\xi_j,r_j)}
    w_{(\xi_j,\varepsilon_j)}^{-\frac{8}{n-4}}
    \overline{\varphi}_{(\xi_j,\varepsilon_j,r_j,k)}
\right\|_{L^{\frac{2n}{n-4}}(M,g_s)}
\le
C\,\big\| w_{(\xi_j,\varepsilon_j)} \big\|_{L^{\frac{2n}{n-4}}(\mathbb{R}^n)}
= c(n) < \infty.
\]

Combining this bound with~\eqref{eq022} and~\eqref{eq021},  
we obtain~\eqref{eq019}.  
Using again~\eqref{eq021} and an entirely analogous calculation,  
we also deduce~\eqref{eq053}.
\end{proof}

Using the metric \(g\) defined in~\eqref{eq013} and the functions 
\(\widetilde{w}_{(\xi_i,\varepsilon_i,r_i)}\) given in~\eqref{eq017}, 
we define
\begin{equation}\label{eq036}
    \mathcal{V}_{(\xi,\varepsilon)}(g)
    :=
    \left\{
        v \in 
        \mathcal{F}^{\perp}_{(\xi,\varepsilon,\alpha,r)}(M,g)
        :
        \mathcal{H}_g\!\left(
            v, \widetilde{w}_{(\xi_i,\varepsilon_i,r_i)}
        \right) = 0
        \quad\text{for all } i = 1,\ldots,\ell
    \right\}.
\end{equation}

\begin{lemma}
    For every \(w \in \mathcal{V}_{(\xi,\varepsilon)}(g)\),  
every \(j = 1,\ldots,\ell\), and every \(k = 0,\ldots,n\),  
there exists a constant \(c = c(n,\ell) > 0\) such that
\begin{equation}\label{eq029}
    \left|
        \int_M 
            \overline{w}_{(\xi_j,\varepsilon_j,r_j)}^{\frac{n+4}{\,n-4\,}} 
            \, w \, dv_{g_s}
    \right|
    \le 
    c\, \alpha \, \|w\|_{W^{2,2}(M,g_s)},
\end{equation}
\begin{equation}\label{eq030}
    \left|
        \int_{B_{r_j}(\xi_j)} 
            \overline{w}_{(\xi_j,\varepsilon_j,r_j)}^{\frac{n+4}{\,n-4\,}} 
            \, w \, dv_{g_s}
    \right|
    \le 
    c\, \alpha \, \|w\|_{W^{2,2}(M,g_s)},
\end{equation}
and
\begin{equation}\label{eq054}
    \left|
        \int_{B_{r_j}(\xi_j)} 
            \overline{\varphi}_{(\xi_j,\varepsilon_j,r_j,k)} 
            \, w \, dv_{g_s}
    \right|
    \le 
    c\, \alpha \, \|w\|_{W^{2,2}(M,g_s)}.
\end{equation}
\end{lemma}
\begin{proof}
Since the supports of the functions 
\(\overline{w}_{(\xi_i,\varepsilon_i,r_i)}\) and 
\(\overline{w}_{(\xi_j,\varepsilon_j,r_j)}\) 
are disjoint for \(i \neq j\),  
and using that \(w \in \mathcal{V}_{(\xi,\varepsilon)}(g)\) and  
\(\overline{w}_{(\xi_i,\varepsilon_i,r_i)}
= \widetilde{w}_{(\xi_i,\varepsilon_i,r_i)}
+ v_{(\xi_i,\varepsilon_i,r_i)}\),  
together with \(dv_{g_s} = dv_g\),  
and denoting \(c_1(n) = \tfrac{n(n^2-4)}{2}\), we have
\begin{align*}
    c_1(n)
    \left|
        \int_M 
            \overline{w}_{(\xi_j,\varepsilon_j,r_j)}^{\frac{n+4}{n-4}} 
            w \, dv_{g_s}
    \right|
    &\le
    \left| \mathcal{H}_g\big( \overline{w}_{(\xi_j,\varepsilon_j,r_j)}, w \big) \right|
    +
    \left|
        \langle P_g \overline{w}_{(\xi_j,\varepsilon_j,r_j)}, w \rangle_{L^2}
        -
        d(n)
        \int_M 
            \overline{w}_{(\xi_j,\varepsilon_j,r_j)}^{\frac{n+4}{n-4}}
            w \, dv_g
    \right| \\
    &\le
    \left|
        \mathcal{H}_g\big( v_{(\xi_j,\varepsilon_j,r_j)}, w \big)
    \right|
    +
    \left|
        \int_M 
            \big(
                P_g \overline{w}_{(\xi_j,\varepsilon_j,r_j)}
                -
                d(n)\,\overline{w}_{(\xi_j,\varepsilon_j,r_j)}^{\frac{n+4}{n-4}}
            \big)
            w \, dv_g
    \right|.
\end{align*}

Using Lemma~\ref{lem006},  
estimate~\eqref{eq053}, and Hölder’s inequality, we obtain
\begin{align*}
    c_1(n)
    \left|
        \int_M 
            \overline{w}_{(\xi_j,\varepsilon_j,r_j)}^{\frac{n+4}{n-4}} 
            w \, dv_{g_s}
    \right|
    &\le
    C \left( \frac{\varepsilon_j}{r_j} \right)^{\!n}
    \|w\|_{W^{2,2}(M,g_s)}  \\
    &\quad+
    \left\|
        P_g \overline{w}_{(\xi_j,\varepsilon_j,r_j)}
        -
        d(n)\,\overline{w}_{(\xi_j,\varepsilon_j,r_j)}^{\frac{n+4}{n-4}}
    \right\|_{L^{\frac{2n}{n+4}}(M,g)}
    \,
    \|w\|_{L^{\frac{2n}{n-4}}(M,g)}.
\end{align*}

Finally, using the Sobolev inequality,  
together with~\eqref{eq063} and Proposition~\ref{propo002},  
we conclude the estimate~\eqref{eq029}.

To prove~\eqref{eq030}, we write
\begin{equation}\label{eq031}
    \left|
        \int_{B_{r_j}(\xi_j)} 
            \overline{w}_{(\xi_j,\varepsilon_j,r_j)}^{\frac{n+4}{n-4}} 
            w \, dv_{g_s}
    \right|
    \le
    \left|
        \int_M 
            \overline{w}_{(\xi_j,\varepsilon_j,r_j)}^{\frac{n+4}{n-4}} 
            w \, dv_{g_s}
    \right|
    +
    \left|
        \int_{M \setminus B_{r_j}(\xi_j)} 
            \overline{w}_{(\xi_j,\varepsilon_j,r_j)}^{\frac{n+4}{n-4}} 
            w \, dv_{g_s}
    \right|.
\end{equation}

Since the support of  
\(\overline{w}_{(\xi_j,\varepsilon_j,r_j)}\)  
is contained in \(B_{2r_j}(\xi_j)\),  
the second term on the right-hand side of~\eqref{eq031}  
can be estimated directly.  
Combining this with~\eqref{eq029}, we obtain~\eqref{eq030}.

To prove~\eqref{eq054},  
we apply \eqref{eq043}, \eqref{eq051}, and \eqref{eq052} to observe that $\left|
    \overline{\varphi}_{(\xi_j,\varepsilon_j,r_j,k)}
\right|
\le
\overline{w}_{(\xi_j,\varepsilon_j,r_j)}^{\frac{n+4}{n-4}}$ in $B_{r_j}(\xi_j)$. Thus, the inequality~\eqref{eq054} follows directly from~\eqref{eq030}.
\end{proof}

\begin{theorem}\label{teo003}
Consider the Riemannian metric \(g\) defined in~\eqref{eq013}.  
There exist constants \(\beta = \beta(n,\ell) > 0\),  
\(\alpha_0 \in (0,1)\), and \(s_0 > 0\) such that
\[
    \mathcal{H}_{g}(u,u)
    \;\ge\;
    \beta\,\|u\|_{W^{2,2}(M,g)}^{2},
\]
for all 
\(\alpha \in (0,\alpha_0)\),  
\(s \in (0,s_0)\),  
\(r = (r_1,\ldots,r_\ell)\) with  
\(r_i \in \big(0, \min\{1, R/2\}\big)\),  
\(0 < R < s\),  
\((\xi,\varepsilon) \in \mathcal{D}_{(\alpha,r)}\),  
and every  
\(u \in \mathcal{V}_{(\xi,\varepsilon)}(g)\);  
see~\eqref{eq024} and~\eqref{eq036}.
\end{theorem}
\begin{proof}
We first claim that the result holds for the metric \(g_s\). Suppose this is not the case. Then, for every \(m \in \mathbb{N}\) there would exist parameters \(\alpha_m, s_m \in (0,1/m)\),  a pair \((\xi_m,\varepsilon_m) \in \mathcal{D}_{(\alpha_m,r_m)}\),  and a function  \(u_m \in \mathcal{V}_{(\xi_m,\varepsilon_m)}(g)\), such that
\begin{equation}\label{eq037}
    \mathcal{H}_{g_m}(u_m,u_m)
    < 
    \frac{1}{m}\, \|u_m\|_{W^{2,2}(M,g_m)}^{2}.
\end{equation}
where we set \(g_m := g_{s_m}\). In particular, each \(u_m\) is nontrivial, and we may normalize and assume that $\|u_m\|_{W^{2,2}(M,g_m)}^{2} = 1$.

Note that by \cite[p.~2145]{MR3420504}, it is possible to write the Paneitz operator as
\begin{equation*}\label{eq039}
    \langle P_{g} u, u \rangle_{L^{2}}
    =
    \int_M
    \left\{
        \frac{n-6}{n-2} \, (\Delta_g u)^{2}
        \;+\;
        a(n)\, |\nabla_g^{2} u|^{2}
        \;+\;
        b(n)\, R_g \, |\nabla_g u|^{2}
        \;+\;
        c(n)\, Q_g \, u^{2}
    \right\}
    dv_g,
\end{equation*}
and, using this identity, we obtain
    \begin{align*}
        \langle P_{g_m}u_m,u_m\rangle_{L^2} & =\int_M\left\{\frac{n-6}{n-2}(\Delta_{g_m} u_m)^2+a(n)\left|\nabla_{g_m}^2 u_m\right|^2+b(n)R_{g_0}|\nabla_{g_m} u_m|^2+c(n) Q_{g_0} u_m^2\right\} d v_{g_m}\\
        & +\int_M\left(b(n)(R_{g_m}-R_{g_0})|\nabla_{g_m}u_m|^2+c(n)(Q_{g_m}-Q_{g_0})u_m^2\right)dv_{g_m}.
    \end{align*}
    
The construction of \(g_m\) (see Section~\ref{sec005}) implies that  
\(R_{g_m} - R_{g_0}\) is uniformly bounded and supported in \(B_{2 s_m}(p)\).  
Hence, $\int_M 
        (R_{g_m} - R_{g_0})\, |\nabla_{g_m} u_m|^{2} 
    \, dv_{g_m}
    \;\to\; 0$.
    
By the Sobolev embedding  
\(W^{2,2}(M,g_m) \hookrightarrow L^{\frac{2n}{n-4}}(M,g_m)\)  
and estimate~\eqref{eq033},  
the sequence \(\{u_m\}\) is uniformly bounded in  
\(L^{\frac{2n}{n-4}}(M,g_m)\).  
Thus, using~\eqref{eq034}, we obtain
\[
    \left|
        \int_M 
            (Q_{g_m} - Q_{g_0})\, u_m^{2} 
        \, dv_{g_m}
    \right|
    \le 
    C \,
    \|Q_{g_m} - Q_{g_0}\|_{L^{\frac{n}{4}}(M,g_m)}
    \,
    \|u_m\|_{L^{\frac{2n}{n-4}}(M,g_m)}^{2}
    \;\longrightarrow\; 0 .
\]

Therefore, using that \(R_{g_0} > 0\), \(Q_{g_0} > 0\),  
and that \(\|u_m\|^{2}_{W^{2,2}(M,g_m)} = 1\),  
we may extract a subsequence (still denoted \(u_m\)) such that $\langle P_{g_m} u_m, u_m \rangle_{L^{2}}
    \;\to\;
    c_0 > 0 $.
    
Since the support of $\overline{w}_{(\xi_{jm},\varepsilon_{jm},r_{jm})}$ belongs to $B_{2r_{jm}}(\xi_{jm})$, where the metric $g_m$ coincides with the euclidean metric, by \eqref{eq005}, the H\"older inequality, we get
\begin{equation*}\label{eq042}
    \left|\int_M\overline{w}_{(\xi_{jm},\varepsilon_{jm},r_{jm})}^{\frac{8}{n-4}}u_m^2dv_{g_m}\right|\leq C\|u_m\|^2_{L^{\frac{2n}{n-4}}(M,g_m)}\leq C,
\end{equation*}
for some constant $C>0$ independently of $m$. Therefore, up to a subsequence, for each $j\in\{1,\ldots,\ell\}$ the limit $\displaystyle\lim_{m\to\infty}\int_M\overline{w}_{(\xi_{jm},\varepsilon_{jm},r_{jm})}^{\frac{8}{n-4}}u_m^2dv_{g_m}$ exists. 

By \eqref{eq024} and \eqref{eq037} we obtain
\begin{equation}\label{eq038}
\displaystyle\lim_{m\to\infty}\frac{n+4}{n-4}d(n)\sum_{j=1}^\ell\int_M\overline{w}_{(\xi_{jm},\varepsilon_{jm},r_{jm})}^{\frac{8}{n-4}}u_m^2dv_{g_m}\geq c_0>0,
\end{equation}
and thus, for some $j\in\{1,\ldots,\ell\}$ it holds
    \begin{equation}\label{eq040}
        \lim_{m\to\infty}\int_M\overline w_{(\xi_{jm},\varepsilon_{jm},r_{jm})}^{\frac{8}{n-4}}u_m^2dv_{g_m}> 0.
    \end{equation}

\noindent{\bf Claim 1:} There exists $j\in\{1,\ldots,\ell\}$ satisfying \eqref{eq038} such that, up to a subsequence, it holds
\begin{equation}\label{eq041}
    \frac{n-6}{n-2}\lim_{m\to\infty}\int_{\Omega_{jm}}(\Delta_{g_{m}}u_m)^2dv_{g_m}\leq \lim_{m\to\infty}\frac{n+4}{n-4}d(n)\int_M\overline w_{(\xi_{jm},\varepsilon_{jm},r_{jm})}^{\frac{8}{n-4}}u_m^2dv_{g_m}.
\end{equation}

Consider the nonempty set  $A = \left\{
    j \in \{1, \ldots, \ell\} :
    \displaystyle\lim_{m \to \infty}
    \int_M 
        \overline{w}_{(\xi_{j m}, \varepsilon_{j m}, r_{j m})}^{\frac{8}{n-4}}
        u_m^2 \, dv_{g_m}
\right\}$ and define $\Omega_{jm}=B_{\sqrt{m}\varepsilon_{jm}}(\xi_{jm})$. If the claim were not true, then for all \( j \in A \) we would have
$$\frac{n-6}{n-2}\lim_{m\to\infty}\int_{\Omega_{jm}}(\Delta_{g_{m}}u_m)^2dv_{g_m}> \lim_{m\to\infty}\frac{n+4}{n-4}d(n)\int_M\overline w_{(\xi_{jm},\varepsilon_{jm},r_{jm})}^{\frac{8}{n-4}}u_m^2dv_{g_m}.$$
Since $\Omega_{im}\cap\Omega_{jm}=\emptyset$ if $i\not=j$, as before, we get
\begin{align*}
    c_0 & = \lim_{m\to\infty}\langle P_{g_m}u_m,u_m\rangle_{L^2}\\
    & =\lim_{m\to\infty}\int_M\left\{\frac{n-6}{n-2}(\Delta_{g_m} u_m)^2+a(n)\left|\nabla_{g_m}^2 u_m\right|^2+b(n)R_{g_0}|\nabla_{g_m} u_m|^2+c(n) Q_{g_0} u_m^2\right\} d v_{g_m}\\
        & +\lim_{m\to\infty}\int_M\left(b(n)(R_{g_m}-R_{g_0})|\nabla_{g_m}u_m|^2+c(n)(Q_{g_m}-Q_{g_0})u_m^2)\right)dv_{g_m}\\
        & \geq \frac{n-6}{n-2}\lim_{m\to\infty}\int_M(\Delta_{g_{m}}u_m)^2dv_{g_m} \geq \frac{n-6}{n-2}\lim_{m\to\infty}\sum_{j=1}^\ell\int_{\Omega_{jm}}(\Delta_{g_{m}}u_m)^2dv_{g_m}\\
        & > \lim_{m\to\infty}\frac{n+4}{n-4}d(n)\sum_{j=1}^\ell\int_M\overline w_{(\xi_{jm},\varepsilon_{jm},r_j)}^{\frac{8}{n-4}}u_m^2dv_{g_m},
\end{align*}
which contradicts \eqref{eq038}. This proves Claim 1.

Recall that \(\|u_m\|_{W^{2,2}(M,g_m)} = 1\).  
In particular, the sequence \(\{u_m\}\) is uniformly bounded in  
\(W^{2,2}(\Omega_{jm}, g_m)\).  
Moreover, by the Sobolev embedding theorem and~\eqref{eq033},  
we also have that \(u_m\) is uniformly bounded in  
\(L^{\frac{2n}{n-4}}(\Omega_{jm}, g_m)\). Fix \(j \in \{1,\ldots,\ell\}\) satisfying \eqref{eq040} and~\eqref{eq041}.  
Define the rescaled functions  
\(\overline{u}_m : B_{\frac{r_{jm}}{\varepsilon_{jm}}}(0) \subset \mathbb{R}^n \to \mathbb{R}\) by $ \overline{u}_m(y)
    :=
    \varepsilon_{jm}^{\frac{n-4}{2}}
    u_m(\xi_{jm} + \varepsilon_{jm} y)$. Because \(B_{\sqrt{m}}(0) \subset B_{\frac{r_{jm}}{\varepsilon_{jm}}}(0)\)  
and \(\xi_{jm} + \varepsilon_{jm} y \in B_{r_{jm}}(\xi_{jm})\),  
where the metric \(g_m\) coincides with the Euclidean metric, we obtain
\[
    \lim_{m\to\infty}
    \int_{B_{\sqrt{m}}(0)}
        \overline{u}_m(y)^{\frac{2n}{n-4}}\, dy
    =
    \lim_{m\to\infty}
    \int_{\Omega_{jm}}
        u_m(x)^{\frac{2n}{n-4}}\, dx
    \le c(n),
\]
and
\[
    \lim_{m\to\infty}
    \int_{B_{\sqrt{m}}(0)}
        |\nabla^{2}\overline{u}_m(y)|^{2}\, dy
    =
    \lim_{m\to\infty}
    \int_{\Omega_{jm}}
        |\nabla^{2} u_m(x)|^{2}\, dx
    \le c(n).
\]

\medskip

Therefore, using~\eqref{eq043} together with~\eqref{eq040} and \eqref{eq041},  
we obtain a function  
\(\overline{u} : \mathbb{R}^{n} \to \mathbb{R}\) such that,  
up to a subsequence, $0
    <
   \displaystyle \int_{\mathbb{R}^{n}}
        \left( \frac{2}{1 + |y|^{2}} \right)^{4}
        \overline{u}(y)^{2}\, dy
    <
    \infty,$ and
\begin{equation}\label{eq049}
    \frac{n-6}{n-2}
    \int_{\mathbb{R}^{n}}
        (\Delta \overline{u})^{2} \, dy
    \;\le\;
    \frac{n+4}{n-4}d(n)
    \int_{\mathbb{R}^{n}}
        \left( \frac{2}{1 + |y|^{2}} \right)^{4}
        \overline{u}(y)^{2}\, dy.
\end{equation}

\medskip

Since \(u_m \in \mathcal{V}_{(\xi_m,\varepsilon_m)}(g)\)  
(see~\eqref{eq036}),  
using~\eqref{eq051}, \eqref{eq052}, \eqref{eq029}, and~\eqref{eq054},  
we obtain the orthogonality conditions
\begin{equation}\label{eq046}
    \int_{\mathbb{R}^{n}}
        \left( \frac{2}{1 + |y|^{2}} \right)^{\frac{n+4}{2}}
        \overline{u}(y)\, dy
    = 0,
\end{equation}
\begin{equation*}\label{eq044}
    \int_{\mathbb{R}^{n}}
        \left( \frac{2}{1 + |y|^{2}} \right)^{\frac{n+4}{2}}
        \frac{|y|^{2}-1}{|y|^{2}+1}
        \overline{u}(y)\, dy
    = 0,
\end{equation*}
and, for each \(k = 1,\ldots,n\),
\begin{equation}\label{eq045}
    \int_{\mathbb{R}^{n}}
        \left( \frac{2}{1 + |y|^{2}} \right)^{\frac{n+4}{2}}
        \frac{y_{k}}{|y|^{2}+1}
        \overline{u}(y)\, dy
    = 0.
\end{equation}

\noindent\textbf{Claim 2.}  
The function \(\overline{u}\) satisfies the inequality
\begin{equation}\label{eq048}
    \frac{n+4}{n-4}\, d(n)
    \int_{\mathbb{R}^{n}}
        \left( \frac{2}{1 + |y|^{2}} \right)^{4}
        \overline{u}^{2}\, dy
    \;\le\;\frac{n-2}{n+6}
    \int_{\mathbb{R}^{n}} (\Delta \overline{u})^{2}\, dy .
\end{equation}

Consider the stereographic projection  
\(\sigma : \mathbb{S}^{n} \setminus \{N\} \to \mathbb{R}^{n}\),  
given by \(\sigma(x,t) = \frac{x}{1-t}\),  
where \(N\) is the north pole.  
Let \(\rho = \sigma^{-1}\).  
If \(g_{\mathrm{sph}}\) denotes the round metric on \(\mathbb{S}^{n}\),  then $\rho^{*} g_{\mathrm{sph}}    = w_{0}^{\frac{4}{n-4}} g_{\rm{euc}},$ where \(w_{0}\) is given by~\eqref{eq047} and \(g_{\rm{euc}}\) is the Euclidean metric.

Define a function \(\overline{v}\) on the sphere by $\overline{v}    := (\overline{u}\, w_{0}^{-1}) \circ \sigma $. A direct computation shows that
\[
    \int_{\mathbb{S}^{n}} \overline{v}^{2}\, dv_{g_{\mathrm{sph}}}
    =
    \int_{\mathbb{R}^{n}}
        \left( \frac{2}{1 + |y|^{2}} \right)^{4}
        \overline{u}^{2}\, dy
    < \infty .
\]

Moreover, one readily verifies that conditions  
\eqref{eq046} and \eqref{eq045}  
are equivalent to \(\overline{v}\) being \(L^{2}\)-orthogonal  
to the constant function and to the coordinate functions on the round sphere.

Using either a computation analogous to that in  \cite[Appendix D]{MR1040954}  or the conformal invariance of the Paneitz operator,  we obtain
\begin{equation}\label{eq076}
    \int_{\mathbb{R}^{n}} (\Delta \overline{u})^{2}\, dy
    =
    \int_{\mathbb{S}^{n}}
        \left(
            (\Delta_{g_{\mathrm{sph}}} \overline{v})^{2}
            + \frac{n^{2} - 2n - 4}{2}\, |\nabla_{g_{\mathrm{sph}}} \overline{v}|^{2}
            + \frac{n(n-4)(n^{2}-4)}{16}\, \overline{v}^{2}
        \right)
        dv_{g_{\mathrm{sph}}}.
\end{equation}

Since \(\overline{v}\) is orthogonal to the first two eigenspaces  
of the Laplacian on \(\mathbb{S}^{n}\),  
the variational characterization of the eigenvalues yields
\[
    \int_{\mathbb{R}^{n}} (\Delta \overline{u})^{2}\, dy
    \;\ge\;
    \left(
        \lambda_{2}^{2}
        + \frac{n^{2} - 2n - 4}{2}\, \lambda_{2}
        + \frac{n(n-4)(n^{2}-4)}{16}
    \right)
    \int_{\mathbb{S}^{n}} \overline{v}^{2}\, dv_{g_{\mathrm{sph}}},
\]
where \(\lambda_{2} = 2(n+1)\).  
This is precisely inequality~\eqref{eq048}. Inspired by a similar argument in \cite[Appendix D]{MR1040954}, Claim 2 leads to a contradiction with \eqref{eq049}, thereby completing the proof of the theorem for $g_s$.

To extend the result to the metric \(g\), observe that for \(\alpha > 0\) and \(s_{0} > 0\) sufficiently small,  there exists a constant \(\theta = \theta(n,\ell) > 0\) such that $\mathcal{H}_{g_s}(u,u)\ge    \theta\, \|u\|_{W^{2,2}(M,g_s)}^{2}$ for all $u \in \mathcal{V}_{(\xi,\varepsilon)}(g)$. Using estimates~\eqref{eq058} and~\eqref{eq063}, we obtain
\begin{align*}
    \mathcal{H}_g(u,u)
    &\ge 
    \mathcal{H}_{g_s}(u,u)
    \;-\;
    C_1 \alpha\, \|u\|_{W^{2,2}(M,g_s)}^{2} \ge 
    (\theta - C_1 \alpha)\, \|u\|_{W^{2,2}(M,g_s)}^{2} \ge 
    \frac{\theta - C_1 \alpha}{1 + C_1 \alpha}\,
    \|u\|_{W^{2,2}(M,g)}^{2}.
\end{align*}
which concludes the proof of the theorem.
\end{proof}

Using \eqref{eq016}, Theorem~\ref{teo003}, and the Lax-Milgram theorem, 
we obtain the following result.

\begin{theorem}\label{teo004}
Consider the Riemannian metric \(g\) defined in~\eqref{eq013}.  
Let \(\alpha_{0} \in (0,1)\) and \(s_{0} > 0\) be the constants provided by Theorem~\ref{teo003}.  
Fix \(\ell \in \mathbb{N}\), \(\alpha \in (0,\alpha_{0})\), \(s \in (0,s_{0})\), and  
\(r = (r_{1},\ldots,r_{\ell})\) satisfying  
\(0 < r_{i} < \min\{1, R/2\}\) and \(0 < R < s/4\).  
Let \((\xi,\varepsilon) \in \mathcal{D}_{(\alpha,r)}\) and  
\(f \in L^{\frac{2n}{n+4}}(M,g)\).  
Then there exists a unique function \(w_{f} \in \mathcal{V}_{(\xi,\varepsilon)}(g)\) such that $\mathcal{H}_{g}(w_{f}, \varphi)    =     \langle f, \varphi \rangle_{L^{2}(M,g)}$ for all $\varphi \in \mathcal{V}_{(\xi,\varepsilon)}(g)$. Moreover, there exists a constant \(c > 0\), independent of \(f\), such that
\[
    \|w_{f}\|_{W^{2,2}(M,g)}
    \;\le\;
    c \,\|f\|_{L^{\frac{2n}{n+4}}(M,g)}.
\]
\end{theorem}

Now, we aim to extend Theorem~\ref{teo004} 
to the space \(\mathcal{F}^\perp_{(\xi, \varepsilon, \alpha, r)}(M, g)\). Recall the definition of \(\mathcal{V}_{(\xi, \varepsilon)}(g)\) given in \eqref{eq036}. We first prove the following lemma.

\begin{lemma}\label{lem007}
    For $(\xi, \varepsilon) \in \mathcal{D}_{(\alpha, r)}$, with $\alpha>0$ and $r$ as in Theorem~\ref{teo004}, and for $\widetilde{w}_{(\xi_i, \varepsilon_i, r_i)}$ as given in~\eqref{eq017}, define $\widetilde{H}_{ij}
= \mathcal{H}_g
\big(
\widetilde{w}_{(\xi_i, \varepsilon_i, r_i)}, 
\widetilde{w}_{(\xi_j, \varepsilon_j, r_j)}
\big)$. Then there exists \(\alpha_1 \in (0, \alpha_0)\) such that, 
if \((\xi, \varepsilon) \in \mathcal{D}_{(\alpha_1, r)} \), the matrix \((\widetilde{H}_{ij})\) is invertible, with its norm satisfying $0 < c(n, \ell) < \big|(\widetilde{H}_{ij})\big| \leq C(n, \ell)$.
\end{lemma}
\begin{proof} 
First, we prove that $H_{ij}:=\mathcal{H}_g\!\left(\overline{w}_{(\xi_i,\varepsilon_i,r_i)},\,\overline{w}_{(\xi_j,\varepsilon_j,r_j)}\right)$ satisfies the conclusion of the lemma.

Since the functions  
\(\overline{w}_{(\xi_i,\varepsilon_i,r_i)}\) and  
\(\overline{w}_{(\xi_j,\varepsilon_j,r_j)}\)  
have disjoint supports whenever \(i \neq j\),  
we immediately obtain \(H_{ij} = 0\) for \(i \neq j\).  
Thus, it remains to show that \(H_{ii} \neq 0\).

Using that \(dv_g = dv_{g_s}\) and that the metric \(g_s\) coincides with the Euclidean metric on \(B_{r_i}(\xi_i)\),  
where \(\overline{w}_{(\xi_i,\varepsilon_i,r_i)} = w_{(\xi_i,\varepsilon_i)}\),  
we deduce from~\eqref{eq005} that
\begin{align}
    \left|\int_M\overline{w}_{(\xi_i,\varepsilon_i,r_i)}^{\frac{2n}{n-4}}\, dv_g\right| & \ge\int_{B_{r_i}(\xi_i)}w_{(\xi_i,\varepsilon_i)}^{\frac{2n}{n-4}}\, dx=
    \int_{\mathbb{R}^n}
        w_{(\xi_i,\varepsilon_i)}^{\frac{2n}{n-4}}\, dx
    -
    \int_{\mathbb{R}^n \setminus B_{r_i}(\xi_i)}
        w_{(\xi_i,\varepsilon_i)}^{\frac{2n}{n-4}}\, dx \nonumber\\
    &\ge
    \left(
        \frac{8\, Y_4^{+}(\mathbb{S}^n, g_{\mathrm{can}})}{n(n^2 - 4)}
    \right)^{\!\frac{n}{4}}
    -
    c(n) \left( \frac{\varepsilon_i}{r_i} \right)^{\!n}.\label{eq108}
\end{align}

Furthermore, using~\eqref{eq005}, \eqref{eq024}, and 
Proposition~\ref{propo002}, we obtain
\begin{align*}
    \left|
        H_{ii}
        + \frac{8}{n-4}\, d(n)
        \int_M
            \overline{w}_{(\xi_i,\varepsilon_i,r_i)}^{\frac{2n}{\,n-4\,}}
        \, dv_g
    \right|
    \le
    \left|
        \int_M
            \big(
                P_g \overline{w}_{(\xi_i,\varepsilon_i,r_i)}
                -
                d(n)\,
                \overline{w}_{(\xi_i,\varepsilon_i,r_i)}^{\frac{n+4}{\,n-4\,}}
            \big)
            \overline{w}_{(\xi_i,\varepsilon_i,r_i)}
        \, dv_g
    \right|  \\
    \le
    \left\|
        P_g \overline{w}_{(\xi_i,\varepsilon_i,r_i)}
        -
        d(n)\,
        \overline{w}_{(\xi_i,\varepsilon_i,r_i)}^{\frac{n+4}{\,n-4\,}}
    \right\|_{L^{\frac{2n}{n+4}}(M,g)}\,
    \left\|
        \overline{w}_{(\xi_i,\varepsilon_i,r_i)}
    \right\|_{L^{\frac{2n}{n-4}}(M,g)} 
    \le
    c(n,\ell)\, \alpha .
\end{align*}

Hence,
\[
    |H_{ii}|
    \;\ge\;
    \left|
        \frac{8}{n-4}\, d(n)
        \int_M
            \overline{w}_{(\xi_i,\varepsilon_i,r_i)}^{\frac{2n}{\,n-4\,}}
        \, dv_g
    \right|
    - c(n,\ell)\,\alpha .
\]

Using \eqref{eq108} and choosing \(\alpha > 0\) sufficiently small yields \(|H_{ii}| \geq c> 0\). Combined with the fact that \(H_{ij} = 0\) for \(i \ne j\), this proves the desired property for the matrix \((H_{ij})\).

    \noindent\textbf{Claim.}  
There exists a constant \(C = C(n,\ell) > 0\) such that
\[
    \bigl|\mathcal{H}_{g}\bigl(\overline{w}_{(\xi_i,\varepsilon_i,r_i)},\overline{w}_{(\xi_j,\varepsilon_j,r_j)}\bigr)-\mathcal{H}_{g}\bigl(\widetilde{w}_{(\xi_i,\varepsilon_i,r_i)},\widetilde{w}_{(\xi_j,\varepsilon_j,r_j)}\bigr)\bigr|\le C\,\alpha .
\]

By a direct computation, one verifies that  
\(w_{(\xi_i,\varepsilon_i)} \in W^{2,2}(\mathbb{R}^{n})\),  
with norm bounded independently of the parameters \(\xi_i\) and \(\varepsilon_i\).  
Consequently,
\(\overline{w}_{(\xi_i,\varepsilon_i,r_i)} \in W^{2,2}(M,g)\),  
with norm uniformly bounded in \(\xi_i\), \(\varepsilon_i\), and \(r_i\).   Using~\eqref{eq017}, we may write
\begin{align*}
    &
    \mathcal{H}_{g}\bigl(
        \overline{w}_{(\xi_i,\varepsilon_i,r_i)},
        \overline{w}_{(\xi_j,\varepsilon_j,r_j)}
    \bigr)
    -
    \mathcal{H}_{g}\bigl(
        \widetilde{w}_{(\xi_i,\varepsilon_i,r_i)},
        \widetilde{w}_{(\xi_j,\varepsilon_j,r_j)}
    \bigr)
    \\
    &\qquad=
    \mathcal{H}_{g}\bigl(
        \overline{w}_{(\xi_i,\varepsilon_i,r_i)},
        v_{(\xi_j,\varepsilon_j,r_j)}
    \bigr)
    +
    \mathcal{H}_{g}\bigl(
        v_{(\xi_i,\varepsilon_i,r_i)},
        \overline{w}_{(\xi_j,\varepsilon_j,r_j)}
    \bigr)
    -
    \mathcal{H}_{g}\bigl(
        v_{(\xi_i,\varepsilon_i,r_i)},
        v_{(\xi_j,\varepsilon_j,r_j)}
    \bigr).
\end{align*}

The estimate now follows from~\eqref{eq063},  
\eqref{eq016}, and~\eqref{eq053}, which together imply that each term on the right-hand side is bounded by \(\alpha\), up to a constant. This proves the claim.

\medskip

The lemma then follows by a standard perturbation argument.
\end{proof}

\begin{theorem}\label{teo005}
Fix \(\alpha > 0\) and \(s > 0\) sufficiently small, and consider the Riemannian metric \(g\) defined in~\eqref{eq013}.  Let \(\ell \in \mathbb{N}\) and  \(r = (r_{1},\ldots,r_{\ell})\) satisfy  \(0 < r_{i} < \min\{1, R/2\}\) and \(0 < R < s/4\).  Let \((\xi,\varepsilon) \in \mathcal{D}_{(\alpha,r)}\) and  \(f \in L^{\frac{2n}{n+4}}(M,g)\).  Then there exists a unique function \(w_{f} \in \mathcal{F}^{\perp}_{(\xi,\varepsilon,\alpha,r)}(M,g)\) such that
\begin{equation}\label{eq056}
    \langle P_{g} w_{f}, \varphi \rangle_{L^{2}}    \;-\;\frac{n+4}{n-4}d(n)\int_{M}        W_{(\xi,\varepsilon,r)}^{\frac{8}{n-4}}        w_{f}\, \varphi    \, dv_{g}    \;=\;    \int_{M} f\, \varphi\, dv_{g},    \qquad    \forall\, \varphi    \in \mathcal{F}^{\perp}_{(\xi,\varepsilon,\alpha,r)}(M,g).
\end{equation}
Moreover, there exists a constant \(c > 0\), independent of \(f\), such that
\begin{equation}\label{eq061}
    \|w_{f}\|_{W^{2,2}(M,g)}    \;\le\;    c\, \|f\|_{L^{\frac{2n}{n+4}}(M,g)} .
\end{equation}
\end{theorem}
\begin{proof}
    By Theorem \ref{teo004}, given $f\in L^{\frac{2n}{n+4}}(M,g)$, there exists $\overline w_f\in\mathcal V_{(\xi,\varepsilon)}(g)$ such that
    \[
\mathcal{H}_g(\overline w_f,\varphi) = \langle f,\varphi \rangle_{L^2(M,g)} 
\quad \text{for all } \varphi \in \mathcal{V}_{(\xi,\varepsilon)}(g).
\]
Recall the definition of $\mathcal F^\perp_{(\xi,\varepsilon,\alpha,r)}(M,g)$ in \eqref{eq059} and  $\mathcal V_{(\xi,\varepsilon)}(g)$ in \eqref{eq036}.

Given $v\in \mathcal F^\perp_{(\xi,\varepsilon,\alpha,r)}(M,g)$, there exist constants $a_i$, $i=1,\ldots,\ell$, such that $v-\displaystyle\sum_{i=1}^\ell a_i\widetilde w_{(\xi_i,\varepsilon_i,r_i)}\in \mathcal V_{(\xi,\varepsilon)}(g)$. In fact, since the matrix $(\widetilde H_{ij})$ is invertible, see Lemma \ref{lem007}, $a_i$ is exactly the solution of the system 
$$\mathcal H_g(v,\widetilde w_{(\xi_j,\varepsilon_j,r_j)})=\sum_{i=1}^\ell a_i\widetilde H_{ij},\quad j=1,\ldots,\ell.$$
Thus, $\mathcal F^\perp_{(\xi,\varepsilon,\alpha,r)}(M,g)=\mathcal V_{(\xi,\varepsilon)}(g)+\operatorname{span}\{\widetilde w_{(\xi_i,\varepsilon_i,r_i)}:i=1,\ldots,\ell\}$. In the same way, we can find $\lambda_i\in\mathbb R$, $i=1,\ldots,\ell$, such that
\begin{equation}\label{eq062}
    \sum_{i=1}^\ell\lambda_i\widetilde H_{ij}=\int_Mf\widetilde w_{(\xi_j,\varepsilon_j,r_j)}dv_g,\quad \mbox{ for all }j=1,\ldots,\ell.
\end{equation}
Define $w_f=\overline w_f+\displaystyle\sum_{i=1}^\ell\lambda_i\widetilde w_{(\xi_i,\varepsilon_i,r_i)}\in\mathcal F^\perp_{(\xi,\varepsilon,\alpha,r)}(M,g)$. Given $\varphi\in \mathcal F^\perp_{(\xi,\varepsilon,\alpha,r)}(M,g)$, we can write $\varphi=\overline\varphi+\widetilde\varphi$, with $\overline\varphi\in\mathcal V_{(\xi,\varepsilon)}(g)$ and $\widetilde\varphi\in \operatorname{span}\{\widetilde w_{(\xi_i,\varepsilon_i,r_i)}:i=1,\ldots,\ell\}$. Since $\overline w_f\in\mathcal V_{(\xi,\varepsilon)}(g)$, then
$$\mathcal H_g(w_f,\varphi)=\mathcal H_g(\overline w_f,\overline \varphi)+\sum_{i=1}^\ell\lambda_i\mathcal H_g(\widetilde w_{(\xi_i,\varepsilon_i,r_i)},\widetilde\varphi)=\int_Mf\overline\varphi dv_g+\int_Mf\widetilde\varphi dv_g=\int_Mf\varphi dv_g,$$
that is, $w_f$ satisfies \eqref{eq056}. 

Now we aim to establish the estimate of the norm in \eqref{eq061}. 
From the previous construction of $\overline w_f$, it is therefore sufficient to prove that
$$\|\lambda_i\widetilde w_{(\xi_i,\varepsilon_i,r_i)}\|_{W^{2,2}(M,g)}\leq c(n,\ell)\|f\|_{L^{\frac{2n}{n+4}}(M,g)},\quad\mbox{ for all }i=1,\ldots,\ell.$$

By Lemma \ref{lem007} and \eqref{eq062} we obtain that
\begin{equation}\label{eq064}
    |\lambda_i|\leq c(n)\sum_{i=1}^\ell \int_M\left|f\widetilde w_{(\xi_j,\varepsilon_j,r_j)}\right|dv_g\leq c(n)\|f\|_{L^{\frac{2n}{n+4}}(M,g)}\|\widetilde w_{(\xi_j,\varepsilon_j,r_j)}\|_{L^{\frac{2n}{n-4}}(M,g)}.
\end{equation}
Using~\eqref{eq063} and \eqref{eq017}, we obtain
\begin{align*}
    \|\widetilde w_{(\xi_j,\varepsilon_j,r_j)}\|^2_{W^{2,2}(M,g)} \leq C\left( \|\overline w_{(\xi_j,\varepsilon_j,r_j)}\|^2_{W^{2,2}(M,g)}+\|v_{(\xi_j,\varepsilon_j,r_j)}\|^2_{W^{2,2}(M,g)}\right)
\end{align*}
Using \eqref{eq053}, \eqref{eq064}, the Sobolev inequality and the fact that the $W^{2,2}(M,g)$-norm of $\overline{w}_{(\xi_i, \varepsilon_i, r_i)}$ is uniformly bounded, we obtain the result.
\end{proof}

\begin{theorem}\label{teo006}
Under the assumptions of Theorem~\ref{teo005},  
for all \(\alpha > 0\) sufficiently small and every  
\((\xi,\varepsilon) \in \mathcal{D}_{(\alpha,r)}\),  
there exists a unique function  \(U_{(\xi,\varepsilon,r)} \in W^{2,2}(M,g)\) such that \(U_{(\xi,\varepsilon,r)} -W_{(\xi,\varepsilon,r)}    \in\mathcal{F}^{\perp}_{(\xi,\varepsilon,\alpha,r)}(M,g)\) and 
\begin{equation}\label{eq065}
    \langle P_{g} U_{(\xi,\varepsilon,r)}, \varphi \rangle_{L^{2}}
    \;-\;
    d(n)
    \int_{M}
        |U_{(\xi,\varepsilon,r)}|^{\frac{8}{n-4}}
        U_{(\xi,\varepsilon,r)}\, \varphi
    \, dv_{g}
    = 0,
    \qquad
    \forall\, \varphi \in 
    \mathcal{F}^{\perp}_{(\xi,\varepsilon,\alpha,r)}(M,g).
\end{equation}

Moreover, for some positive constant $c=c(n,\ell)>0$, the following estimate holds:
\begin{equation}\label{eq073}
    \|W_{(\xi,\varepsilon,r)} - U_{(\xi,\varepsilon,r)}\|_{W^{2,2}(M,g)}
    \;\le\;
    c\,
    \left\|
        P_{g} W_{(\xi,\varepsilon,r)}
        -
        d(n)\, W_{(\xi,\varepsilon,r)}^{\frac{n+4}{n-4}}
    \right\|_{L^{\frac{2n}{n+4}}(M,g)}.
\end{equation}
\end{theorem}
\begin{proof}
    By Theorem~\ref{teo005}, we have an operator $G_{(\xi,\varepsilon)} : L^{\frac{2n}{n+4}}(M,g) \to 
\mathcal{F}^\perp_{(\xi,\varepsilon,\alpha,r)}(M,g)$, which assigns to each \( f \in L^{\frac{2n}{n+4}}(M,g) \) the unique function \( w_f \in \mathcal{F}^\perp_{(\xi,\varepsilon,\alpha,r)}(M,g) \) satisfying~\eqref{eq056}. Define the map $\Phi_{(\xi,\varepsilon)} :
\mathcal{F}^\perp_{(\xi,\varepsilon,\alpha,r)}(M,g)
\to
\mathcal{F}^\perp_{(\xi,\varepsilon,\alpha,r)}(M,g)$ as
\begin{align*}
    \Phi_{(\xi,\varepsilon)}(w) & =-G_{(\xi,\varepsilon)}\left(P_gW_{(\xi,\varepsilon,r)}-d(n)W_{(\xi,\varepsilon,r)}^{\frac{n+4}{n-4}}\right)\\
    & + d(n)G_{(\xi,\varepsilon)}\left(|W_{(\xi,\varepsilon,r)}+w|^{\frac{8}{n-4}}(W_{(\xi,\varepsilon,r)}+w)-W_{(\xi,\varepsilon,r)}^{\frac{n+4}{n-4}}-\frac{n+4}{n-4}W_{(\xi,\varepsilon,r)}^{\frac{8}{n-4}}w\right).
\end{align*}
It is a simple computation to see that $V_{(\xi,\varepsilon)}$ is a fixed point of $\Phi_{(\xi,\varepsilon)}$ if and only if $U_{(\xi,\varepsilon)}:=W_{(\xi,\varepsilon,r)}+V_{(\xi,\varepsilon)}$ satisfies \eqref{eq065}. Let us show that $\Phi_{(\xi,\varepsilon)}$ has a unique fixed point by showing that it is a contraction.

First, not that using Proposition \ref{propo002} and \eqref{eq061} we find that $\|\Phi_{(\xi,\varepsilon)}(0)\|_{W^{2,2}(M,g)}\leq c(n,\ell)\alpha$. Now, using \eqref{eq066} and the pointwise inequality
\begin{align*}
    \left||W_{(\xi,\varepsilon,r)}+w_0|^{\frac{8}{n-4}}\left(W_{(\xi,\varepsilon,r)}+w_0\right)-|W_{(\xi,\varepsilon,r)}+w_1|^{\frac{8}{n-4}}\left(W_{(\xi,\varepsilon,r)}+w_1\right)-\frac{n+4}{n-4}W_{(\xi,\varepsilon,r)}^{\frac{8}{n-4}}(w_0-w_1)\right|\\
    \leq C\left(|w_0|^{\frac{8}{n-4}}+|w_1|^{\frac{8}{n-4}}\right)|w_0-w_1|,
\end{align*}
for all $w_0,w_1\in L^{\frac{2n}{n-4}}(M,g)$, we obtain that
\begin{align*}
    \|\Phi_{(\xi,\varepsilon)}(w_0)& -\Phi_{(\xi,\varepsilon)}(w_1)\|_{W^{2,2}(M,g)}\\
    \leq &~ C\left\||W_{(\xi,\varepsilon,r)}+w_0|^{\frac{8}{n-4}}\left(W_{(\xi,\varepsilon,r)}+w_0\right)-|W_{(\xi,\varepsilon,r)}+w_1|^{\frac{8}{n-4}}\left(W_{(\xi,\varepsilon,r)}+w_1\right)\right.\\
    & -\left.\frac{n+4}{n-4}W_{(\xi,\varepsilon,r)}^{\frac{8}{n-4}}(w_0-w_1)\right\|_{L^{\frac{2n}{n+4}}(M,g)}\\
     \leq & ~ C\left(\left\|w_0\right\|^{\frac{8}{n-4}}_{L^{\frac{2n}{n-4}}(M,g)}+\left\|w_1\right\|^{\frac{8}{n-4}}_{L^{\frac{2n}{n-4}}(M,g)}\right)\left\|w_0-w_1\right\|_{L^{\frac{2n}{n-4}}(M,g)}.
\end{align*}

Therefore, for $\alpha>0$ small enough, the contraction mapping principle implies that the mapping 
$\Phi_{(\xi, \varepsilon)}$ has a unique fixed point. 
The inequality~\eqref{eq073} follows immediately.
\end{proof}

For $\alpha>0$ and $U_{(\xi,\varepsilon,r)}\in W^{2,2}(M,g)$ as in Theorem \ref{teo006}, define the functional 
$\mathcal F_g:\mathcal D_{(\alpha,r)}\to\mathbb R$ by
\begin{align}
\mathcal{F}_g(\xi,\varepsilon)
= \langle P_g U_{(\xi,\varepsilon,r)}, U_{(\xi,\varepsilon,r)} \rangle_{L^2}
- \frac{n-4}{n}d(n) \int_M |U_{(\xi,\varepsilon,r)}|^{\frac{2n}{n-4}} dv_g
-\frac{4}{n}d(n)\,\ell \left( \frac{8 Y_4^+(\mathbb{S}^n, g_{\mathrm{can}})}{n(n^2 - 4)} \right)^{\frac{n}{4}} .\label{eq082}
\end{align}

\begin{theorem}\label{teo007}
    The function \(\mathcal{F}_g\) is continuously differentiable.  
Moreover, for \(\alpha > 0\) sufficiently small,  
if \((\xi,\varepsilon) \in \mathcal{D}_{(\alpha,r)}\) is a critical point of  
\(\mathcal{F}_g\), then the corresponding function  
\(U_{(\xi,\varepsilon,r)}\) is a smooth nonnegative solution of the equation
\begin{equation}\label{eq074}
    P_g U_{(\xi,\varepsilon,r)}
    \;=\;
    d(n)\, U_{(\xi,\varepsilon,r)}^{\frac{n+4}{n-4}} .
\end{equation}
\end{theorem}
\begin{proof}
    Since \eqref{eq065} holds for all  
\(\varphi \in \mathcal{F}^{\perp}_{(\xi,\varepsilon,\alpha,r)}(M,g)\),  
we may find constants  
\(a_{ik}(\xi,\varepsilon)\),  
\(i = 1,\ldots,\ell\), \(k = 0,\ldots,n\),  
such that
\begin{equation}\label{eq084}
    \langle P_{g} U_{(\xi,\varepsilon,r)}, \varphi \rangle_{L^{2}}
    -
    d(n)
    \int_{M}
        |U_{(\xi,\varepsilon,r)}|^{\frac{8}{n-4}}
        U_{(\xi,\varepsilon,r)} \varphi \, dv_{g}
    =
    \sum_{i=1}^{\ell} \sum_{k=0}^{n}
        a_{ik}(\xi,\varepsilon)
        \int_{M}
            \overline{\varphi}_{(\xi_i,\varepsilon_i,r_i,k)}
            \varphi \, dv_{g},
\end{equation}
for all \(\varphi \in W^{2,2}(M,g)\).   Since  \(U_{(\xi,\varepsilon,r)} - W_{(\xi,\varepsilon,r)}    \in     \mathcal{F}^{\perp}_{(\xi,\varepsilon,\alpha,r)}(M,g)\), for all \(i = 1,\ldots,\ell\), \(k = 0,\ldots,n\), we have
\[
    \int_{M}
        \overline{\varphi}_{(\xi_i,\varepsilon_i,r_i,k)}
        \bigl(
            U_{(\xi,\varepsilon,r)}
            - W_{(\xi,\varepsilon,r)}
        \bigr)
        \, dv_{g}
    = 0.
\]

Differentiating with respect to \(\varepsilon_{j}\)  
and using~\eqref{eq050}, denoting  \(a_{ik} := a_{ik}(\xi,\varepsilon)\),  \(\overline{\varphi}_{ik} := \overline{\varphi}_{(\xi_i,\varepsilon_i,r_i,k)}\) and \(\overline{w}_{i} := \overline{w}_{(\xi_i,\varepsilon_i,r_i)}\), we obtain
\begin{align}
    &
    \int_{M}
        \frac{\partial}{\partial \varepsilon_{j}}
            \overline{\varphi}_{ik}
        \bigl(
            U_{(\xi,\varepsilon,r)}
            - W_{(\xi,\varepsilon,r)}
        \bigr)
        \, dv_{g}
    +
    \int_{M}
        \overline{\varphi}_{ik}
        \left(
            \frac{\partial}{\partial\varepsilon_{j}}
                U_{(\xi,\varepsilon,r)}
            -
            \frac{\partial}{\partial\varepsilon_{j}}
                \overline{w}_{j}
        \right)
        dv_{g}
    = 0 ,
    \label{eq067}
\end{align}
and differentiating with respect to \(\xi_{jt}\), \(t = 1,\ldots,n\), yields
\begin{align}
    &
    \int_{M}
        \frac{\partial}{\partial \xi_{jt}}
            \overline{\varphi}_{ik}
        \bigl(
            U_{(\xi,\varepsilon,r)}
            - W_{(\xi,\varepsilon,r)}
        \bigr)
        \, dv_{g}
    +
    \int_{M}
        \overline{\varphi}_{ik}
        \left(
            \frac{\partial}{\partial\xi_{jt}}
                U_{(\xi,\varepsilon,r)}
            -
            \frac{\partial}{\partial\xi_{jt}}
                \overline{w}_{j}
        \right)
        dv_{g}
    = 0 .
    \label{eq068}
\end{align}

In particular, for \(j \neq i\) we obtain
\[
    \int_{M}
        \overline{\varphi}_{ik}
        \frac{\partial}{\partial\varepsilon_j}
            U_{(\xi,\varepsilon,r)}
        \, dv_{g}
    =
    \int_{M}
        \overline{\varphi}_{ik}
        \frac{\partial}{\partial\xi_{jt}}
            U_{(\xi,\varepsilon,r)}
        \, dv_{g}
    = 0,
\]
since in this case  
\(\overline{\varphi}_{(\xi_i,\varepsilon_i,r_i,k)}\)  
does not depend on \(\varepsilon_j\) or \(\xi_{jt}\),  
and \(\overline{\varphi}_{ik}\)  
and \(\overline{w}_{j}\) have disjoint supports.

Using~\eqref{eq084}, \eqref{eq067}, and~\eqref{eq068}, we obtain
\begin{align}
    \frac{1}{2}
    \frac{\partial}{\partial \varepsilon_{j}}
        \mathcal{F}_{g}(\xi,\varepsilon)
    &=
    \sum_{i=1}^{\ell}
    \sum_{k=0}^{n}
        a_{ik}
        \int_{M}
            \overline{\varphi}_{ik}\,
            \frac{\partial}{\partial \varepsilon_{j}}
                U_{(\xi,\varepsilon,r)}
            \, dv_{g}
    =
    \sum_{k=0}^{n}
        a_{jk}
        \int_{M}
            \overline{\varphi}_{jk}\,
            \frac{\partial}{\partial \varepsilon_{j}}
                U_{(\xi,\varepsilon,r)}
            \, dv_{g}
            \nonumber\\
    &=
    \sum_{k=0}^{n}
        \varepsilon_{j}^{-1}\, a_{jk}\, \beta_{jkj0}
    \;-\;
    \sum_{k=0}^{n}
        a_{jk}
        \int_{M}
            \frac{\partial}{\partial\varepsilon_{j}}
                \overline{\varphi}_{jk}\,
            \bigl(
                U_{(\xi,\varepsilon,r)}
                - W_{(\xi,\varepsilon,r)}
            \bigr)
            \, dv_{g},
    \label{eq070}
\end{align}
and, for each \(t = 1,\ldots,n\),
\begin{align}
    \frac{1}{2}
    \frac{\partial}{\partial \xi_{jt}}
        \mathcal{F}_{g}(\xi,\varepsilon)
    &=
    \sum_{k=0}^{n}
        a_{jk}
        \int_{M}
            \overline{\varphi}_{jk}\,
            \frac{\partial}{\partial \xi_{jt}}
                U_{(\xi,\varepsilon,r)}
            \, dv_{g}
            \nonumber\\
    &=
    \sum_{k=0}^{n}
        \varepsilon_{j}^{-1}\, a_{jk}\, \beta_{jkjt}
    \;-\;
    \sum_{k=0}^{n}
        a_{jk}
        \int_{M}
            \frac{\partial}{\partial \xi_{jt}}
                \overline{\varphi}_{jk}\,
            \bigl(
                U_{(\xi,\varepsilon,r)}
                - W_{(\xi,\varepsilon,r)}
            \bigr)
            \, dv_{g},
    \label{eq071}
\end{align}
where \(\beta_{jkjt}\) is defined in~\eqref{eq085}. It is not difficult to see that
$$|\beta_{j0j0}| \geq \varepsilon_j\left|\int_{B_{r_j}(\xi_j)} \overline{\varphi}_{j0}\,\frac{\partial}{\partial \varepsilon_j} \overline{w}_j\, dx\right|\geq \frac{n-4}{2}\int_{B_1(0)}\left(\frac{2}{1+|x|}\right)^n\left(\frac{1-|x|^2}{1+|x|^2}\right)^2dx$$
and, for all $k=1,\ldots,n$,
$$|\beta_{jkjk}| \geq \varepsilon_j\left|\int_{B_{r_j}(\xi_j)} \overline{\varphi}_{jk}\,
\frac{\partial}{\partial \xi_{jk}} \overline{w}_j\, dx\right|\geq \frac{n-4}{2}\int_{B_1(0)}\left(\frac{2}{1+|x|}\right)^n\left(\frac{2x_k}{1+|x|^2}\right)^2dx.$$
Thus, $T_j:=\min\big\{|\beta_{j0j0}|,|\beta_{jkjk}|: k=1,\ldots, n\big\} \ge c(n)>0$, for all $j\in\{1,\ldots,\ell\}$. Therefore, if $(\overline\xi,\overline\varepsilon)$ is a critical point of $\mathcal F_g$, by \eqref{eq072}, \eqref{eq070} and \eqref{eq071}, we have
$$|a_{j0}(\overline\xi,\overline\varepsilon)|\leq c\overline\varepsilon_j\left|\sum_{k=0}^na_{jk}\int_M\frac{\partial}{\partial\varepsilon_{j}}\overline\varphi_{jk}(U_{(\overline\xi,\overline\varepsilon,r)} -W_{(\overline\xi,\overline\varepsilon,r)})\right|$$
and
$$|a_{jt}(\overline\xi,\overline\varepsilon)|\leq c\overline\varepsilon_j\left|\sum_{k=0}^na_{jk}\int_M\frac{\partial}{\partial\xi_{jt}}\overline\varphi_{jk}(U_{(\overline\xi,\overline\varepsilon,r)} -W_{(\overline\xi,\overline\varepsilon,r)})\right|,$$
for all $t=1,\ldots,n$.
By a direct computation, one checks that  
\(\partial_{\varepsilon_j}\overline{\varphi}_{jk}\) and  
\(\partial_{\xi_{jt}}\overline{\varphi}_{jk}\)  
are bounded, up to a dimensional constant, by  
\(\overline{\varepsilon}_{j}^{-1}w_{(\overline{\xi}_j,\overline{\varepsilon}_j)}^{\frac{n+4}{n-4}}\).  
Using~\eqref{eq005}, \eqref{eq073}, Hölder's and Sobolev's inequalities,  
together with Proposition~\ref{propo002}, we obtain, for each  
\(j \in \{1,\ldots,\ell\}\),  
\[
    \sum_{t=0}^{n}
        |a_{jt}(\overline{\xi},\overline{\varepsilon})|
    \;\le\;
    c(n)
    \sum_{k=0}^{n}
        |a_{jk}(\overline{\xi},\overline{\varepsilon})|\,
        \bigl\|
            w_{(\overline{\xi}_j,\overline{\varepsilon}_j)}^{\frac{n+4}{n-4}}
        \bigr\|_{L^{\frac{2n}{n+4}}(\mathbb{R}^{n})}\,
        \| U_{(\overline{\xi},\overline{\varepsilon},r)}
           - W_{(\overline{\xi},\overline{\varepsilon},r)}
        \|_{L^{\frac{2n}{n-4}}}
    \;\le\;
    c\,\alpha
    \sum_{k=0}^{n}
        |a_{jk}(\overline{\xi},\overline{\varepsilon})|.
\]
For \(\alpha > 0\) sufficiently small, this yields  
\(a_{jk}(\overline{\xi},\overline{\varepsilon}) = 0\)  
for all \(j = 1,\ldots,\ell\) and \(k = 0,\ldots,n\).  
Therefore \(U_{(\overline{\xi},\overline{\varepsilon},r)}\)  
is a weak solution of the equation
$P_gv=d(n)|v|^{\frac{8}{n-4}}v$. By \cite[Proposition 2.3]{robert}, $U_{(\overline{\xi},\overline{\varepsilon},r)}\in C^4(M)$ and is a solution of this equation in the usual sense.

\medskip

Next, we show that  
\(U_{(\overline{\xi},\overline{\varepsilon},r)} \ge 0\).  
Let 
\(U^- = \max\{0, -U_{(\overline{\xi},\overline{\varepsilon},r)}\}\geq 0\)  denote the negative part of $U_{(\overline{\xi},\overline{\varepsilon},r)}$. Since \(W_{(\xi,\varepsilon,r)} \ge 0\),  on the set  
\(\{ U_{(\overline{\xi},\overline{\varepsilon},r)} < 0 \}\) it holds $|U_{(\overline{\xi},\overline{\varepsilon},r)}|\le | W_{(\xi,\varepsilon,r)}-U_{(\overline{\xi},\overline{\varepsilon},r)} |$. Using~\eqref{eq073}, Proposition~\ref{propo002} and Sobolev's inequality, yields
\begin{equation}\label{eq112}
\|U^-\|_{L^{\frac{2n}{n-4}}(M,g)}
\leq C\alpha .    
\end{equation}

Under our assumptions, the operator \(P_g\) is positive and its Green's function is positive. Consequently, $P_g:W^{2,2}(M,g)\to L^{\frac{2n}{n+4}}(M,g)$ admits a bounded inverse, with the norm bounded independently of $\alpha>0$ being sufficiently small; see for instance \cite{MR3518237}. Define $V:=P_g^{-1}\!\left(d(n)(U^-)^{\frac{n+4}{n-4}}\right)\geq 0$. By the Sobolev's inequality,
\begin{equation}\label{eq110}
\|V\|_{L^{\frac{2n}{n-4}}(M,g)}
    \leq C\|V\|_{W^{2,2}(M,g)}
    \leq C\|(U^-)^{\frac{n+4}{n-4}}\|_{L^{\frac{2n}{n+4}}(M,g)} 
    = C\|U^-\|_{L^{\frac{2n}{n-4}}(M,g)}^{\frac{n+4}{n-4}}.
\end{equation}

Now, since $P_gU_{(\overline{\xi},\overline{\varepsilon},r)}=d(n)|U_{(\overline{\xi},\overline{\varepsilon},r)}|^{\frac{8}{n-4}}U_{(\overline{\xi},\overline{\varepsilon},r)}$ we get
\begin{align*}
    \int_M |U_{(\overline \xi,\overline \varepsilon,r)}|^{\frac{8}{n-4}} & U_{(\overline \xi,\overline \varepsilon,r)}V\,dv_g
    =
    d(n)^{-1}\int_M VP_gU_{(\overline \xi,\overline \varepsilon,r)}\,dv_g=    d(n)^{-1}\int_M U_{(\overline \xi,\overline \varepsilon,r)}P_gV\,dv_g \\
    &=
    \int_M (U^-)^{\frac{n+4}{n-4}}U_{(\overline \xi,\overline \varepsilon,r)}\,dv_g =    -\int_{\{U_{(\overline \xi,\overline \varepsilon,r)}<0\}} (U^-)^{\frac{2n}{n-4}}\,dv_g =
    -\|U^-\|_{L^{\frac{2n}{n-4}}(M,g)}^{\frac{2n}{n-4}} .
\end{align*}
Since \(V\geq 0\) and \(U^-\geq -U_{(\overline \xi,\overline \varepsilon,r)}\), it follows that
\begin{align}
    \|U^-\|_{L^{\frac{2n}{n-4}}(M,g)}^{\frac{2n}{n-4}}
    &=
    -\int_M |U_{(\overline \xi,\overline \varepsilon,r)}|^{\frac{8}{n-4}}U_{(\overline \xi,\overline \varepsilon,r)}V\,dv_g \nonumber\leq
    \int_M (U^-)^{\frac{n+4}{n-4}}V\,dv_g \nonumber\\
    &\leq
    \|U^-\|_{L^{\frac{2n}{n-4}}(M,g)}^{\frac{n+4}{n-4}}
    \|V\|_{L^{\frac{2n}{n-4}}(M,g)} .
\label{eq111}
\end{align}

Assuming that \(U^-\not\equiv 0\), combining \eqref{eq110} and \eqref{eq111}, we deduce $1\leq C\|U^-\|_{L^{\frac{2n}{n-4}}(M,g)}^{\frac{8}{n-4}}$. This yields a contradiction with \eqref{eq112}, provided \(\alpha>0\) is sufficiently small. Therefore $U_{(\overline \xi,\overline \varepsilon,r)}\geq 0$.
\end{proof}

To guaranty the nonnegativity of $U_{(\overline{\xi},\overline{\varepsilon},r)}$, we used the positivity of the invariants $Y(M,g) > 0$ and $Y_{4}(M,g) > 0$. In the final section, in Proposition \ref{propo013}, we use these hypotheses to apply the maximum principle of \cite[Theorem~A]{MR3420504} and conclude that the function is actually positive.

\section{Some preliminary estimates in \texorpdfstring{$\mathbb{R}^n$}{Lg}}\label{sec004}

In this section we follow the ideas of~\cite[Section~2]{MR2425176}. Set
\[
    \mathcal{E}
    :=\left\{
        w\in L^{\frac{2n}{n-4}}(\mathbb{R}^n)\cap W^{2,2}_{\mathrm{loc}}(\mathbb{R}^n)
        :\ \int_{\mathbb{R}^n} (\Delta w)^2\, dx < \infty
    \right\}.
\]
By Sobolev's inequality, there exists a constant \(K>0\), depending only on \(n\), such that
\begin{equation}\label{eq077}
    \left( \int_{\mathbb{R}^n} |w|^{\frac{2n}{n-4}}\, dx \right)^{\!\frac{n-4}{n}}
    \leq
    K \int_{\mathbb{R}^n} (\Delta w)^2\, dx,
\end{equation}
for all \(w\in\mathcal{E}\). We endow \(\mathcal{E}\) with the norm $\|w\|_{\mathcal{E}}^2 := \displaystyle\int_{\mathbb{R}^n} (\Delta w)^2\, dx$. With this norm, \((\mathcal{E},\|\cdot\|_{\mathcal{E}})\) is a complete space. Given \((\xi,\varepsilon)\in\mathbb{R}^n\times(0,\infty)\), define the subspace
\begin{equation}\label{eq081}
    \mathcal{E}_{(\xi,\varepsilon)}
    :=
    \left\{
        w\in\mathcal{E}
        :\ 
        \int_{\mathbb{R}^n} \varphi_{(\xi,\varepsilon,k)}\, w\, dx = 0
        \ \text{for all } k=0,1,\ldots,n
    \right\},
\end{equation}
where the function \(\varphi_{(\xi,\varepsilon,k)}\) is given in \eqref{eq098}. By \eqref{eq005}, it holds $w_{(\xi,\varepsilon)}\in\mathcal{E}_{(\xi,\varepsilon)}$.

On \(\mathbb{R}^n\), consider a Riemannian metric of the form \(g(x)=\exp(h(x))\),
where \(h(x)\) is a trace-free symmetric two-tensor satisfying~\eqref{eq014} and
such that \(h(x)=0\) for \(|x|\geq R>0\). Using an argument similar to that of
Proposition~\ref{propo002}, we obtain the existence of a constant \(c(n)>0\)
such that
\begin{equation}\label{eq086}
    \left\|
        P_g w_{(\xi,\varepsilon)}
        - d(n)\, w_{(\xi,\varepsilon)}^{\frac{n+4}{n-4}}
    \right\|_{L^{\frac{2n}{n+4}}(\mathbb{R}^n)}
    \leq c(n)\,\alpha.
\end{equation}

Finally, using \eqref{eq076} and arguing as in \cite[Proposition~2]{MR2425176}
(see also \cite[Appendix~D]{MR1040954}), we obtain the following proposition.

\begin{proposition}\label{propo004}
   There exist positive constants $a_1(n)$ and $b_1(n)$, depending only on $n$, such that
for every $w \in \mathcal{E}_{(\xi,\varepsilon)}$ one has
\begin{equation}\label{eq078}
    \int_{\mathbb{R}^n}
    \left((\Delta w)^2-\frac{n+4}{n-4}d(n)\, w_{(\xi,\varepsilon)}^{\frac{8}{n-4}}\, w
    \right)
    \geq 
    a_1(n)\, \|w\|_{\mathcal{E}}^{2}
    \;-\;
    b_1(n)
    \left(
        \int_{\mathbb{R}^n}
        w_{(\xi,\varepsilon)}^{\frac{n+4}{n-4}}\, w
    \right)^{\!2}.
\end{equation}
\end{proposition}

Note that the second term on the right-hand side appears because the condition  \(w \in \mathcal{E}_{(\xi,\varepsilon)}\) guarantees only that, after transporting  \(w\) to the sphere via stereographic projection, the resulting function is  orthogonal merely to the coordinate functions.

\begin{corollary}\label{propo005}
For all sufficiently small $\alpha>0$, depending only on $n$, there exist positive constants $c_1(n)$ and $d_1(n)$ such that, for every $w \in \mathcal{E}_{(\xi,\varepsilon)}$, one has
\begin{align*}
    &\left\langle P_g w,\, w \right\rangle_{L^2(\mathbb{R}^n)}
    - \frac{n+4}{n-4}d(n)
      \int_{\mathbb{R}^n}
      w_{(\xi,\varepsilon)}^{\frac{8}{n-4}}\, w^2\, dx \\
    &\qquad
    + d_1(n)\,
      \left(
          \int_{\mathbb{R}^n}
          \left(
              P_g w_{(\xi,\varepsilon)}
              - \frac{n+4}{n-4}d(n)\,
                w_{(\xi,\varepsilon)}^{\frac{n+4}{n-4}}
          \right)
          w\, dx
      \right)^{\!2}
    \;\;\geq\;\;
    c_1(n)\, \|w\|_{\mathcal{E}}^{2}.
\end{align*}
\end{corollary}
\begin{proof}
First, observe that under the assumption $w \in \mathcal{E}_{(\xi,\varepsilon)}$, 
and using Proposition~\ref{propo004} together with the Sobolev inequality 
\eqref{eq077}, as well as the fact that the support of $h$ is contained in the 
unit ball, we infer that, for all $\alpha>0$ sufficiently small (depending only 
on $n$), inequality~\eqref{eq078} continues to hold, possibly with different 
constants, when the left-hand side is replaced by the corresponding expression 
involving the curvature terms.

Using \eqref{eq077}, \eqref{eq086}, and H\"older's inequality, we obtain
\[
\left|
    \int_{\mathbb{R}^n}
    \left(
        P_g w_{(\xi,\varepsilon)}
        - \frac{n+4}{n-4}d(n)\,
          w_{(\xi,\varepsilon)}^{\frac{n+4}{n-4}}
    \right) w
\right|
\;\geq\;
\frac{8}{n-4}d(n)
\left|
    \int_{\mathbb{R}^n}
    w_{(\xi,\varepsilon)}^{\frac{n+4}{n-4}}\, w
\right|
\;-\;
C \alpha\, \|w\|_{\mathcal{E}}.
\]
For $\alpha>0$ sufficiently small, an application of Young's inequality yields
\begin{align*}
    \left(
        \int_{\mathbb{R}^n}
        \left(
            P_g w_{(\xi,\varepsilon)}
            - \frac{n+4}{n-4}d(n)\,
              w_{(\xi,\varepsilon)}^{\frac{n+4}{n-4}}
        \right) w
    \right)^{\!2}
    &\geq
    \frac{64}{(n-4)^2}d(n)^2
    \left(
        \int_{\mathbb{R}^n}
        w_{(\xi,\varepsilon)}^{\frac{n+4}{n-4}}\, w
    \right)^{\!2}
    \;-\;
    C \alpha^2 \|w\|_{\mathcal{E}}^{2}.
\end{align*}
By \eqref{eq078}, the desired estimate follows.
\end{proof}

\begin{proposition}\label{propo006}
    Given \((\xi,\varepsilon) \in \mathbb{R}^{n} \times (0,\infty)\) and any function  
\(f \in L^{\frac{2n}{n+4}}(\mathbb{R}^{n})\), for all sufficiently small parameters  
\(\alpha > 0\), depending only on \(n\), there exists a unique function  
\(w \in \mathcal{E}_{(\xi,\varepsilon)}\) such that
\begin{equation}\label{eq079}
    \left\langle P_{g} w ,\, \varphi \right\rangle_{L^{2}(\mathbb{R}^{n})}
    \;-\;
    \frac{n+4}{n-4}d(n)
    \int_{\mathbb{R}^{n}}
        w_{(\xi,\varepsilon)}^{\frac{8}{n-4}}\,
        w\, \varphi \, dx
    =
    \int_{\mathbb{R}^{n}}
        f\, \varphi\, dx,
\end{equation}
for every \(\varphi \in \mathcal{E}_{(\xi,\varepsilon)}\).  
Moreover, there exists a constant \(C>0\), depending only on \(n\), such that
\[
    \| w \|_{\mathcal{E}}
    \;\leq\;
    C\, \| f \|_{L^{\frac{2n}{n+4}}(\mathbb{R}^{n})}.
\]
\end{proposition}
\begin{proof}
Suppose that $w\in\mathcal{E}_{(\xi,\varepsilon)}$ satisfies \eqref{eq079} for 
all $\varphi\in\mathcal{E}_{(\xi,\varepsilon)}$.  
Recalling that $w_{(\xi,\varepsilon)}\in\mathcal{E}_{(\xi,\varepsilon)}$, we obtain
\[
    \langle P_g w, w\rangle_{L^2(\mathbb{R}^n)}
    - \frac{n+4}{n-4}d(n)
      \int_{\mathbb{R}^n}
      w_{(\xi,\varepsilon)}^{\frac{8}{n-4}} w^2\, dx
    =
    \int_{\mathbb{R}^n} f\, w\, dx,
\]
and similarly,
\[
    \langle P_g w_{(\xi,\varepsilon)}, w\rangle_{L^2(\mathbb{R}^n)}
    - \frac{n+4}{n-4}d(n)
      \int_{\mathbb{R}^n}
      w_{(\xi,\varepsilon)}^{\frac{n+4}{n-4}} w\, dx
    =
    \int_{\mathbb{R}^n} f\, w_{(\xi,\varepsilon)}\, dx.
\]

Using \eqref{eq005}, \eqref{eq077}, Proposition~\ref{propo005}, and the fact that 
$w_{(\xi,\varepsilon)}\in\mathcal{E}_{(\xi,\varepsilon)}$, we obtain
\begin{align*}
    c_1(n)\,\|w\|_{\mathcal{E}}^2
    &\leq
    \langle w, P_g w\rangle_{L^2(\mathbb{R}^n)}
    - \frac{n+4}{n-4}d(n)
      \int_{\mathbb{R}^n}
      w_{(\xi,\varepsilon)}^{\frac{8}{n-4}} w^2\, dx
\\
    &\quad
    + d_1(n)
      \left(
        \int_{\mathbb{R}^n}
        \left(
            P_g w_{(\xi,\varepsilon)}
            - \frac{n+4}{n-4}d(n)\,
              w_{(\xi,\varepsilon)}^{\frac{n+4}{n-4}}
        \right) w\, dx
      \right)^{\!2}
\\
    &\leq
    \int_{\mathbb{R}^n} f w\, dx
    + d_1(n)
      \left(
        \int_{\mathbb{R}^n} f\, w_{(\xi,\varepsilon)}
      \right)^{\!2}
\\
    &\leq
    \|f\|_{L^{\frac{2n}{n+4}}(\mathbb{R}^n)}
    \,\|w\|_{L^{\frac{2n}{n-4}}(\mathbb{R}^n)}
    +
    \|f\|_{L^{\frac{2n}{n+4}}(\mathbb{R}^n)}^{2}
    \,
    \|w_{(\xi,\varepsilon)}\|_{L^{\frac{2n}{n-4}}(\mathbb{R}^n)}^{2}
\\
    &\leq
    C\, \|f\|_{L^{\frac{2n}{n+4}}(\mathbb{R}^n)}\, \|w\|_{\mathcal{E}}
    + C\, \|f\|_{L^{\frac{2n}{n+4}}(\mathbb{R}^n)}^{2}
\leq
    \lambda\, \|w\|_{\mathcal{E}}^{2}
    + C \left( 1 + \frac{1}{\lambda} \right)
      \|f\|_{L^{\frac{2n}{n+4}}(\mathbb{R}^n)}^{2},
\end{align*}
where $\lambda>0$ is arbitrary.
Choosing $\lambda>0$ sufficiently small yields the estimate $\|w\|_{\mathcal{E}}
    \leq C\, \|f\|_{L^{\frac{2n}{n+4}}(\mathbb{R}^n)}$, which in turn implies the uniqueness of the solution.

To prove existence, consider the functional 
\(F : \mathcal{E}_{(\xi,\varepsilon)} \to \mathbb{R}\) defined by
\[
    F(w)
    =
    \langle w, P_g w\rangle_{L^2(\mathbb{R}^n)}
    - \frac{n+4}{n-4}d(n)
      \int_{\mathbb{R}^n}
      w_{(\xi,\varepsilon)}^{\frac{8}{n-4}} w^2\, dx
    - 2 \int_{\mathbb{R}^n} f w\, dx
    + d_1(n)\, A(w)^{2},
\]
where $    A(w)=\displaystyle\int_{\mathbb{R}^n}\left(P_g w_{(\xi,\varepsilon)}- \frac{n+4}{n-4}d(n)w_{(\xi,\varepsilon)}^{\frac{n+4}{n-4}}\right) w$. Since \(F\) is coercive and lower semicontinuous, it admits a minimizer 
\(w_0 \in \mathcal{E}_{(\xi,\varepsilon)}\).
By applying the Euler-Lagrange equation associated with \(F\), we conclude that $ w_0 + d_1(n)\, A(w_0)\, w_{(\xi,\varepsilon)}    \in \mathcal{E}_{(\xi,\varepsilon)}$ satisfies \eqref{eq079}, together with the corresponding norm estimate.
\end{proof}

\section{Estimate for the \texorpdfstring{$\ell$}{Lg}-bubble}\label{sec006}

A very useful strategy introduced in \cite{MR2425176} is to work with a reduced version of the energy functional in order to simplify the estimates. This approach is particularly effective in the $Q$-curvature setting. In this section, our goal is to construct such a reduced energy functional, building upon the analysis in \cite{MR3016505}, in order to obtain several useful estimates in our framework. The main goal is to construct an auxiliary functional that is sufficiently close to $\mathcal{F}_g$, allowing the transfer of information between them.

\subsection{The set up for the metric}\label{sec009}

In what follows, we fix a multilinear form  
\(W : \mathbb{R}^{n} \times \mathbb{R}^{n} \times \mathbb{R}^{n} \times \mathbb{R}^{n} \to \mathbb{R}\).  
We assume that the components \(W_{ijkl}\) satisfy all the algebraic symmetries of the Weyl tensor.  
Moreover, we assume that at least one component of \(W\) is non-zero, so that
\[
    \sum_{i,j,k,l=1}^{n}
        \bigl( W_{ijkl} + W_{ilkj} \bigr)^{2}
        > 0 .
\]

For brevity, we define $H_{ik}(x):=\displaystyle\sum_{p,q=1}^{n}W_{ipkq}\, x_{p} x_{q}$ for $x\in\mathbb R^n$ and
\begin{equation}\label{eq095}
    \overline{H}_{ik}(x)
    :=
    f(|x|^{2})\, H_{ik}(x),
\end{equation}
where the auxiliary fourth-order polynomial
$ f(x) = \tau - 1200x + 2411x^{2} - 135x^{3} + x^{4}$ was originally introduced in \cite{MR3016505} to handle dimensions $n \ge 25$. The constant $\tau$ is defined in \cite[Lemma 11.5]{MR3016505}. Note that \(H_{ik}\) is trace-free, that \(\sum_{i=1}^{n} x_{i} H_{ik}(x) = 0\), and that \(\sum_{i=1}^{n} \partial_{i} H_{ik}(x) = 0\) for all \(x \in \mathbb{R}^{n}\).

Fix a sufficiently small constant \(\alpha\in(0,1)\) such that the assumptions of Theorem~\ref{teo007} and Proposition~\ref{propo006} are satisfied. For each \(t \in \{1, \ldots, \ell\}\), we choose parameters $\lambda_{t}$, $\mu_{t}$ and $\rho_{t}$ such that  
\begin{equation}\label{eq087}
\mu_{t} \le 1, 
    \qquad
    2\lambda_{t} \le \rho_{t} \le R/2 \le 1,
    \qquad
    (3/2 - \alpha)\lambda_{t} < \alpha \rho_{t},
    \qquad
    2\rho_{t} + 3\lambda_{t} < R/2.    
\end{equation}
Let \(y_{1},\ldots,y_{\ell} \in B_{R/2}(0)\) be distinct points satisfying
\begin{equation}\label{eq088}
\|y_{i} - y_{j}\|
    \;\ge\;
    3(\lambda_{i}+\lambda_{j}) + 2(\rho_{i}+\rho_{j}),
    \qquad i \neq j.    
\end{equation}

We consider a Riemannian metric \(g\) of the form~\eqref{eq013},  
where \(h(x)\) is a trace-free symmetric two-tensor on \(\mathbb{R}^{n}\)  
satisfying~\eqref{eq014}, and such that
\begin{equation*}\label{eq100}
h_{ik}(x)
    =
    \mu_{t}\,\lambda_{t}^{8}\,
    f\bigl(\lambda_{t}^{-2} |x - y_{t}|^{2}\bigr)\;
    H_{ik}(x - y_{t}),\qquad \text{if } |x-y_{t}| \le \rho_{t},    
\end{equation*}
and \(h(x)=0\) for all \(|x| \ge R\). It is easy to see that $\sum_{i=1}^nx_ih_{ik}(x)=0$ and $\sum_{i=1}^n\partial_ih_{ik}(x)=0$ for $|x-y_t|\leq\rho_t$. For \(y=(y_{1},\ldots,y_{\ell})\), we define
\begin{equation}\label{eq102}
\Omega(\lambda, y)
    :=
    \Bigl\{
        (\xi,\varepsilon) \in \mathbb{R}^{n\ell} \times \mathbb{R}_{+}^{\ell}
        :\;
        |\xi_{t} - y_{t}| < \lambda_{t},\;
        \frac{\lambda_{t}}{2} < \varepsilon_{t} < \frac{3}{2}\lambda_{t},\;
        \frac{1}{2} < \frac{\varepsilon_{t}}{\varepsilon_{i}} < 2
    \Bigr\}.    
\end{equation}

By considering \(r_{t} = \rho_{t} + \lambda_{t}\), we have \(B_{\rho_{t}}(y_{t}) \subset B_{r_{t}}(\xi_{t})\) and $\Omega(\lambda,y) \subset \mathcal{D}_{(\alpha,r)}$. Indeed, let \((\xi,\varepsilon) \in \Omega(\lambda,y)\).  
Since \(\varepsilon_{t} < 3\lambda_{t}/2\) and \((3/2 - \alpha)\lambda_{t} < \alpha\rho_{t}\), we have $\varepsilon_{t}/r_{t} < \alpha$. Moreover, if \(i \neq j\), then $    |\xi_{i}-\xi_{j}|\geq |y_{i} - y_{j}| - |\xi_{i} - y_{i}| - |\xi_{j} - y_{j}|\geq 2(r_{i}+r_{j})$. Finally, $|\xi_{t}|    \le |\xi_{t}-y_{t}| + |y_{t}|    < \lambda_{t} + R/2    < R - 2r_{t}$, because \(2\rho_{t} + 3\lambda_{t} < R/2\).

Using the fact that the support of  
\(w_{(\xi_{t},\varepsilon_{t},r_{t})}\)  
is contained in \(B_{2r_{t}}(\xi_{t})\),  
and applying an argument analogous to the proof of  
Proposition~\ref{propo002},  
we obtain the following result for the $\ell$-bubbles.
\begin{lemma}\label{lem008}
    For every $(\xi,\varepsilon)\in\Omega(\lambda,y)$, for some $c(n,\ell)>0$, we have
    $$\left\|P_gW_{(\xi,\varepsilon,r)}-d(n)W_{(\xi,\varepsilon,r)}^{\frac{n+4}{n-4}}\right\|_{L^{\frac{2n}{n+4}}(M\backslash\cup_{t=1}^\ell B_{\rho_t}(y_t))}\leq c(n,\ell)\sum_{t=1}^\ell\left(\frac{\lambda_t}{\rho_t}\right)^{\frac{n-4}{2}}.$$
\end{lemma}

\subsection{Energy Expansion}
Fix \((\xi, \varepsilon) \in \mathcal{D}_{(\alpha, r)}\).
To define the reduced energy functional, we apply Proposition~\ref{propo006}
with \(h \equiv 0\). Consequently, for each \(t \in \{1, \ldots, \ell\}\),
there exists a unique function  $z_{(\xi_t, \varepsilon_t)} \in \mathcal{E}_{(\xi_t, \varepsilon_t)}$ (see~\eqref{eq081}) satisfying
\begin{equation}\label{eq057}
    \int_{\mathbb R^n}
    \left( 
        \Delta z_{(\xi_t,\varepsilon_t)} \, \Delta \varphi
        - \frac{n+4}{n-4}d(n)\, w_{(\xi_t,\varepsilon_t)}^{\frac{8}{n-4}}
          z_{(\xi_t,\varepsilon_t)} \, \varphi 
    \right)
    = 
    -\int_{\mathbb R^n}
     \Gamma_{(y_t,\xi_t,\varepsilon_t)} \, \varphi,
\end{equation}
for all \(\varphi \in \mathcal{E}_{(\xi_t,\varepsilon_t)}\). Here
\begin{align}
    \Gamma_{(y_t,\xi_t,\varepsilon_t)}(x)
    &= 
    \mu_t \lambda_t^{8} \Big(
       2\, G_{kj}(x-y_t)
           \partial_k \partial_j \partial_s^2
           w_{(\xi_t,\varepsilon_t)}(x)
       + 2\, \partial_s G_{kj}(x-y_t)
           \partial_k \partial_j \partial_s
           w_{(\xi_t,\varepsilon_t)}(x) \nonumber \\
    &\quad 
       + \frac{n}{n-2}
           \partial_s^2 G_{kj}(x-y_t)
           \partial_k \partial_j
           w_{(\xi_t,\varepsilon_t)}(x)
       + \frac{2}{n-2}
           \partial_j \partial_s^2 G_{kj}(x-y_t)
           \partial_k w_{(\xi_t,\varepsilon_t)}(x)
       \Big),\label{eq099}
\end{align}
where the repeated indices \(k\), \(j\) and $s$ indicate summation from \(1\) to \(n\), and $ G_{ik}(x)
    = \lambda_t^2 \overline H_{ik}(\lambda_t^{-1} x)
    = f(\lambda_t^{-2} |x|^2) \, H_{ik}(x)$.  Additionally, since $(\lambda_t+|x-y_t|)^2/(\varepsilon_t^2+|x-\xi_t|^2)\leq C$ for all \(x \in \mathbb R^n\), we obtain
\begin{equation}\label{eq089}
    |\Gamma_{(y_t,\xi_t,\varepsilon_t)}(x)|
    \leq C\, \mu_t \lambda_t^{\frac{n-4}{2}}
          \bigl(\lambda_t + |x - y_t|\bigr)^{10 - n}.
\end{equation}

By preliminary estimates obtained in Section \ref{sec004}, for each $t$, we can find real numbers 
\(b_k(\xi_t,\varepsilon_t)\), \(k = 0,1,\ldots,n\), such that
\begin{equation}\label{eq093}
\int_{\mathbb{R}^n} 
\left(
    \Delta z_{(\xi_t,\varepsilon_t)} \, \Delta \varphi
    - c_n\, w_{(\xi_t,\varepsilon_t)}^{\frac{8}{n-4}}
      z_{(\xi_t,\varepsilon_t)} \, \varphi
\right)
=
-\int_{\mathbb{R}^n} \Gamma_{(y_t,\xi_t,\varepsilon_t)} \, \varphi
+ 
\sum_{k=0}^n b_k(\xi_t,\varepsilon_t)
    \int_{\mathbb{R}^n} \varphi_{(\xi_t,\varepsilon_t,k)} \, \varphi,
\end{equation}
for all \(\varphi \in \mathcal{E}\).  
By standard elliptic regularity, each function \(z_{(\xi_t,\varepsilon_t)}\) is smooth.  
Furthermore, using an argument analogous to Proposition 5.1 in \cite{MR3016505}  
(see also Section 8 therein), we obtain
\begin{equation}\label{eq027}
    |\partial^{i} z_{(\xi_t,\varepsilon_t)}(x)|
    \leq C\, \mu_t\, \lambda_t^{\frac{n-4}{2}}
    \bigl(\lambda_t + |x - y_t|\bigr)^{14 - n - i},
    \qquad \text{for all } i = 1,2,3,4.
\end{equation}

Using the results of the Section 4 in \cite{MR3016505} we prove the following estimates.
\begin{lemma}\label{lem009}
Consider a metric \(g = \exp(h)\) on \(\mathbb{R}^n\), where \(h\) is a trace-free symmetric two-tensor satisfying $\sum_{i=1}^nx_ih_{ik}(x)=0$ and $\sum_{i=1}^n\partial_ih_{ik}(x)=0$. Then
    \begin{align}
        (\Delta_g^2-\Delta^2)w_{(\xi_t,\varepsilon_t)}   = & ~  -\partial_s^2h_{ij}\partial_i\partial_jw_{(\xi_t,\varepsilon_t)}-2\partial_sh_{ij}\partial_s\partial_i\partial_jw_{(\xi_t,\varepsilon_t)}-2h_{ij}\partial_s^2\partial_i\partial_jw_{(\xi_t,\varepsilon_t)}\nonumber\\
    & ~ +O(|h||\partial^2h|)|\partial^2w_{(\xi_t,\varepsilon_t)}|+O(|h||\partial h|)|\partial^3w_{(\xi_t,\varepsilon_t)}|+O(|h|^2)|\partial^4w_{(\xi_t,\varepsilon_t)}|,\nonumber\\
     Q_g  = & ~ \frac{1}{4(n-1)}\left((\partial_i\partial_lh_{mk})^2+\partial_lh_{mk}\partial_i^2\partial_lh_{mk}\right)-\frac{1}{2(n-2)^2}\partial_m^2h_{ij}\partial_s^2h_{ij}\nonumber\\
        & ~+O(|\partial h|^2|\partial^2 h|+|h||\partial^2h|^2+|h||\partial h||\partial^3h|+|h|^2|\partial^4h|),\nonumber\\
        \operatorname{div}_g(\operatorname{Ric}_g(\nabla_gw_{(\xi_t,\varepsilon_t)}))  = & ~ -\frac{1}{2}\partial_i\partial_m^2h_{il}\partial_lw_{(\xi_t,\varepsilon_t)}-\frac{1}{2}\partial_m^2h_{il}\partial_i\partial_lw_{(\xi_t,\varepsilon_t)}+|\partial^2w_{(\xi_t,\varepsilon_t)}|O(|h||\partial^2h|+|\partial h|^2)\nonumber\\
        & ~ +|\partial w_{(\xi_t,\varepsilon_t)}|O(|\partial h||\partial^2h|+|\partial h|^3+|h||\partial^3h|)\nonumber
    \end{align}
    and
    $$\operatorname{div}_g(R_g\nabla_gw_{(\xi_t,\varepsilon_t)}) = |\partial w_{(\xi_t,\varepsilon_t)}|O(|\partial h||\partial^2h|+|h|^2|\partial^3h|+|\partial h|^3)+|\partial^2w_{(\xi_t,\varepsilon_t)}|O(|\partial h|^2+|h|^2|\partial h|).\\$$
\end{lemma}

These expansions will be used in the ball $|x-y_t|\leq\rho_t$. Now we have the following estimates.

\begin{proposition}\label{propo007}
    For every $(\xi,\varepsilon)\in\Omega(\lambda,y)$, there exist a positive constant $c=c(n,\ell)$ such that
    \begin{equation}\label{eq090}
        \left\|P_gW_{(\xi,\varepsilon,r)}-d(n)W_{(\xi,\varepsilon,r)}^{\frac{n+4}{n-4}}\right\|_{L^{\frac{2n}{n+4}}(M)}\leq c\sum_{t=1}^\ell\left(\left(\frac{\lambda_t}{\rho_t}\right)^{\frac{n-4}{2}}+\mu_t\lambda_t^{10}\right)
    \end{equation}
    and
\begin{equation}\label{eq091}
\left\|P_gW_{(\xi,\varepsilon,r)}-d(n)W_{(\xi,\varepsilon,r)}^{\frac{n+4}{n-4}}+\sum_{t=1}^\ell\eta_{(r_t,\xi_t)}\Gamma_{(y_t,\xi_t,\varepsilon_t)}\right\|_{L^{\frac{2n}{n+4}}(M)}\leq c\sum_{t=1}^\ell\left(\left(\frac{\lambda_t}{\rho_t}\right)^{\frac{n-4}{2}}+\mu^2\lambda^{10\frac{n-3}{n-4}}\right).
\end{equation}
\end{proposition}
\begin{proof} 
Recall that \(B_{\rho_t}(y_t) \subset B_{r_t}(\xi_t)\), that \(\varepsilon_t < 2 \lambda_t \leq \rho_t\), and that \(r_t = \rho_t + \lambda_t\). Moreover, the support of the function \(W_{(\xi,\varepsilon,r)}\) is contained in \(\bigcup_{t=1}^{\ell} B_{2r_t}(\xi_t)\).
Using \eqref{eq089}, we find
\begin{align*}
    \|\Gamma_{(y_t,\xi_t,\varepsilon_t)}\|_{L^{\frac{2n}{n+4}}(\mathbb R^n\backslash B_{\rho_t}(y_t))}  & \leq C\mu_t\lambda_t^{\frac{n-4}{2}}\left(\int_{\mathbb R^n\backslash B_{\rho_t}(y_t))}(\lambda_t+|x-y_t|)^{(10-n)\frac{2n}{n+4}}dx\right)^{\frac{2n}{n+4}}\\
    &
     \leq C\mu_t\rho_t^{10}\left(\frac{\lambda_t}{\rho_t}\right)^{\frac{n-4}{2}}.
\end{align*}
Using Lemma \ref{lem008}, we obtain
\begin{align*}
        \left\|P_gW_{(\xi,\varepsilon,r)}-d(n)W_{(\xi,\varepsilon,r)}^{\frac{n+4}{n-4}}+\sum_{t=1}^\ell\eta_{(r_t,\xi_t)}\Gamma_{(y_t,\xi_t,\varepsilon_t)}\right\|_{L^{\frac{2n}{n+4}}(M\backslash\cup_{t=1}^\ell B_{\rho_t}(y_t))}\leq c(n,\ell)\sum_{t=1}^\ell\left(\frac{\lambda_t}{\rho_t}\right)^{\frac{n-4}{2}}.
    \end{align*}

Now, let us estimate the integral inside each ball \( B_{\rho_t}(y_t) \), for \( t = 1, \ldots, \ell \), 
where the metric takes the form \( g(x) = \exp(h(x)) \). Since, for each \( t \in \{1, \ldots, \ell\} \), the function \( w_{(\xi_t, \varepsilon_t)} \) is a solution of \eqref{eq012}, we have
\begin{align*}
    A_t & :=P_gw_{(\xi_t,\varepsilon_t)}-d(n)w_{(\xi_t,\varepsilon_t)}^{\frac{n+4}{n-4}}\\
    & = (\Delta_g^2-\Delta^2)w_{(\xi_t,\varepsilon_t)}+\operatorname{div}_g(a(n)\operatorname{Ric}_g(\nabla w_{(\xi_t,\varepsilon_t)})-b(n)R_gdw_{(\xi_t,\varepsilon_t)})+Q_gw_{(\xi_t,\varepsilon_t)}.
\end{align*}

The Lemma \ref{lem009} implies that
\begin{align*}
    |Q_g w_{(\xi_t,\varepsilon_t)}| & \leq C\mu^2\lambda_t^{\frac{n-4}{2}}(\lambda_t+|x-y_t|)^{20-n},\\
        |(\Delta_g^2-\Delta^2)w_{(\xi_t,\varepsilon_t)}| & \leq C\mu\lambda_t^{\frac{n-4}{2}}(\lambda_t+|x-y_t|)^{10-n},\\
        \operatorname{div}_g(\operatorname{Ric}_g(\nabla_gw_{(\xi_t,\varepsilon_t)})) & \leq C\mu_t\lambda_t^{\frac{n-4}{2}}(\lambda_t+|x-y_t|)^{10-n},\\
        \operatorname{div}_g(R_g\nabla_gw_{(\xi_t,\varepsilon_t)}) & \leq C\mu^2\lambda_t^{\frac{n-4}{2}}(\lambda_t+|x-y_t|)^{20-n}.
\end{align*}
Therefore, for all $x\in B_{\rho_t}(\xi_t)$ it holds
$|A_t(x)|\leq C\mu_t\lambda_t^{\frac{n-4}{2}}(\lambda_t+|x-y_t|)^{10-n}$. Using the existence of a constant \( c_n > 0 \) such that, for all \( q < -n \), the following estimate holds, $\int_{\mathbb R^n}(\lambda+|x|)^qdx\leq c_n\lambda^{q+n}$, we obtain
\begin{align*}
    \|A_t\|_{L^{\frac{2n}{n+4}}(B_{\rho_t}(\xi_t))} & \leq C\mu_t\lambda_t^{\frac{n-4}{2}}\left(\int_{\mathbb R^n}(\lambda_t+|x|)^{(10-n)\frac{2n}{n+4}}dx\right)^{\frac{n+4}{2n}} \leq C\mu_t\lambda_t^{10}.
\end{align*}
By Lemma \ref{lem008} we obtain the estimate \eqref{eq090}.

For the estimate \eqref{eq091}, we observe that 
\(\Gamma_{(y_t,\xi_t,\varepsilon_t)}\) cancels the terms in the expansions of 
\(\operatorname{div}_g(\operatorname{Ric}_g(\nabla_g w_{(\xi_t,\varepsilon_t)}))\) 
and 
\((\Delta_g^2 - \Delta^2) w_{(\xi_t,\varepsilon_t)}\) 
that are linear in \(h\); see \eqref{eq099} and Lemma~\ref{lem009}.  
By Lemma~\ref{lem009}, we therefore obtain
\[
    \bigl|\, A_t + \Gamma_{(y_t,\xi_t,\varepsilon_t)} \,\bigr|
    \leq C \mu^2 \lambda_t^{\frac{n-4}{2}}
          \bigl( \lambda_t + |x - y_t| \bigr)^{20 - n}
    \leq 
    C \mu^2 \lambda_t^{\frac{n-4}{2}}
          \bigl( \lambda_t + |x - y_t| \bigr)^{10\frac{n-3}{n-4} - n},
\]
for all \( |x - y_t| \leq \rho_t \).  
As before, this yields $\bigl\| A_t + \Gamma_{(y_t,\xi_t,\varepsilon_t)} 
    \bigr\|_{L^{\frac{2n}{n+4}}(B_{\rho_t}(\xi_t))}
    \leq 
    C \mu^2 \lambda^{\,10\frac{n-3}{n-4}}$.
\end{proof}

A direct consequence of \eqref{eq073} and \eqref{eq090} is the following result.

\begin{corollary}\label{cor001}
   For $(\xi,\varepsilon)\in\Omega(\lambda,y)$, there exists a positive function $c=c(n,\ell)$ such that the function $U_{(\xi,\varepsilon,r)}$ given by Theorem \ref{teo006} satisfies the estimate
    $$\|U_{(\xi,\varepsilon,r)}-W_{(\xi,\varepsilon,r)}\|_{L^{\frac{2n}{n-4}}(M)}\leq c(n,\ell)\sum_{t=1}^\ell\left(\left(\frac{\lambda_t}{\rho_t}\right)^{\frac{n-4}{2}}+\mu_t\lambda_t^{10}\right).$$
\end{corollary}

We now require a more refined estimate for the difference 
\( U_{(\xi,\varepsilon,r)} - W_{(\xi,\varepsilon,r)} \).  
Applying Theorem~\ref{teo005} with \(h \equiv 0\), we conclude that there exists a
unique function $Z_{(\xi,\varepsilon)} \in\mathcal{F}_{(\xi,\varepsilon,\alpha,r)}^{\perp}(M,g_s)$ such that
\begin{equation}\label{eq083}
    \langle P_{g_s} Z_{(\xi,\varepsilon)}, \varphi \rangle_{L^2}
    - \frac{n+4}{n-4} \, d(n) 
      \int_{M} W_{(\xi,\varepsilon,r)}^{\frac{8}{n-4}}
               Z_{(\xi,\varepsilon)} \, \varphi \, dv_{g_s}
    = - \sum_{t=1}^{\ell} 
        \int_{M} \eta_{(r_t,\xi_t)} 
        \Gamma_{(y_t,\xi_t,\varepsilon_t)} \, \varphi \, dv_{g_s},
\end{equation}
for all \(\varphi \in \mathcal{F}_{(\xi,\varepsilon,\alpha,r)}^{\perp}(M,g_s)\). Moreover, $\|Z_{(\xi,\varepsilon)} \|_{W^{2,2}(M,g)}    \;\le\;    c \,     \bigl\| \eta_{(r_t,\xi_t)}            \Gamma_{(y_t,\xi_t,\varepsilon_t)}     \bigr\|_{L^{\frac{2n}{n+4}}(M,g)}$.
Using that the support of 
\(\eta_{(r_t,\xi_t)} \Gamma_{(y_t,\xi_t,\varepsilon_t)}\) is contained in
\(\bigcup_{t=1}^{\ell} B_{2r_t}(\xi_t)\), and arguing again as in \cite{MR3016505}, we obtain
\begin{equation}\label{eq069}
    |\partial^{i} Z_{(\xi,\varepsilon)}(x)|
    \;\le\;
    \sum_{t=1}^{\ell}
    \mu_t \, \lambda_t^{\frac{n-4}{2}}
    \bigl( \lambda_t + |x - y_t| \bigr)^{14 - n - i},
    \qquad \text{for all } i = 1,2,3,4.
\end{equation}

\begin{lemma}
    Fix $(\xi,\varepsilon)\in\Omega(\lambda,y)$. There exists $\mathcal Z_{(\xi,\varepsilon)}\in W^{2,2}(M,g)$ such that
    \begin{equation}\label{eq015}
        \overline Z_{(\xi,\varepsilon)}:=Z_{(\xi,\varepsilon)}-\sum_{t=1}^\ell\eta_{(r_t,\xi_t)}z_{(\xi_t,\varepsilon_t)}+\mathcal Z_{(\xi,\varepsilon)}\in \mathcal F_{(\xi,\varepsilon,\alpha,r)}^\perp(M,g),
    \end{equation}
    \begin{equation}\label{eq032}
        \|\mathcal Z_{(\xi,\varepsilon)}\|_{L^{\frac{2n}{n-4}}(M,g)}\leq c\sum_{t=1}^\ell \left(\frac{\varepsilon_t}{\rho_t}\right)^n
    \end{equation}
    and
\begin{equation}\label{eq055}
    \left\|P_g\mathcal Z_{(\xi,\varepsilon)}-d(n)W_{(\xi,\varepsilon,r)}^{\frac{8}{n-4}}\mathcal Z_{(\xi,\varepsilon)}\right\|_{L^{\frac{2n}{n+4}}(M,g)}\leq c\sum_{t=1}^\ell \left(\frac{\varepsilon_t}{\rho_t}\right)^n.
\end{equation}
\end{lemma}
\begin{proof}
Given coefficients \( c_{tk} \in \mathbb{R} \), define
\[
\mathcal Z_{(\xi,\varepsilon)}
=\sum_{t=1}^{\ell}\sum_{k=0}^{n}
c_{tk}\,
\eta_{(r_t,\xi_t)}
w_{(\xi_t,\varepsilon_t)}^{-\frac{8}{n-4}}
\varphi_{(\xi_t,\varepsilon_t,k)}.
\]
For this choice, since \( Z_{(\xi,\varepsilon)} \in \mathcal{F}_{(\xi,\varepsilon,\alpha,r)}^{\perp}(M,g) \),
the condition \eqref{eq015} is equivalent to
\[
\sum_{t=1}^{\ell}
\int_M
\eta_{(r_t,\xi_t)}
z_{(\xi_t,\varepsilon_t)}
\overline{\varphi}_{(\xi_s,\varepsilon_s,r_s,l)}
=
\sum_{t=1}^{\ell}\sum_{k=0}^{n}
c_{tk}
\int_M
\eta_{(r_t,\xi_t)}
w_{(\xi_t,\varepsilon_t)}^{-\frac{8}{n-4}}
\varphi_{(\xi_t,\varepsilon_t,k)}
\overline{\varphi}_{(\xi_s,\varepsilon_s,r_s,l)},
\]
for all
\(
\overline{\varphi}_{(\xi_s,\varepsilon_s,r_s,l)} \in
\mathcal{F}_{(\xi,\varepsilon,\alpha,r)}.
\)
With computations similar to those in the proof of Lemma~\ref{lem004},
we can show that
$$\int_M\eta_{(r_t,\xi_t)}w_{(\xi_t,\varepsilon_t)}^{-\frac{8}{n-4}}\varphi_{(\xi_t,\varepsilon_t,k)}\overline\varphi_{(\xi_s,\varepsilon_s,r_s,l)}=0$$
for $(t,k)\not=(s,l)$, and 
$$\int_M\eta_{(r_t,\xi_t)}w_{(\xi_t,\varepsilon_t)}^{-\frac{8}{n-4}}\varphi_{(\xi_t,\varepsilon_t,k)}\overline\varphi_{(\xi_t,\varepsilon_t,r_t,k)}\geq c(n)>0.$$
This implies the existence of real numbers $c_{tk}$ such that $\mathcal Z_{(\xi,\varepsilon)}$ satifies \eqref{eq015}. Also, using that $z_{(\xi_t,\varepsilon_t)}\in\mathcal E_{(\xi_t,\varepsilon_t)}$, \eqref{eq027} and a similar computation as in Lemma \ref{lem003} give us
$$\left|\int_M
\eta_{(r_t,\xi_t)}
z_{(\xi_t,\varepsilon_t)}
\overline{\varphi}_{(\xi_s,\varepsilon_s,r_s,l)}\right|\leq c\left(\frac{\varepsilon_t}{\rho_t}\right)^n,$$
which implies \eqref{eq032}.

For \eqref{eq055}, it is easy to see that on 
\(M \setminus \bigcup_{t=1}^\ell \bigl( B_{2r_t}(\xi_t) \setminus B_{r_t}(\xi_t) \bigr)\) 
we have $P_g \, \mathcal Z_{(\xi,\varepsilon)}    - d(n)\, W_{(\xi,\varepsilon,r)}^{\frac{8}{\,n-4\,}}      \mathcal Z_{(\xi,\varepsilon)} = 0$. Finally, using \eqref{eq027} together with a computation analogous to that in 
Proposition~\ref{propo002}, we obtain an estimate for the 
\(L^{\frac{2n}{n+4}}\)-norm on the annular region 
\(B_{2r_t}(\xi_t)\setminus B_{r_t}(\xi_t)\), which yields \eqref{eq055}.
\end{proof}

\begin{proposition}
    For $(\xi,\varepsilon)\in\Omega(\lambda,y)$ it holds
    $$\left\|\overline Z_{(\xi,\varepsilon)}\right\|_{L^{\frac{2n}{n-4}}(M)}\leq c(n,\ell)\sum_{t=1}^\ell\left(\mu_t^2\lambda_t^{10\frac{n-3}{n-4}}+\left(\frac{\lambda_t}{\rho_t}\right)^n+\mu_t\left(\frac{\lambda_t}{\rho_t}\right)^{\frac{n-4}{2}}\right)$$
\end{proposition}
\begin{proof}
Using \eqref{eq057}, \eqref{eq093}, \eqref{eq083} and \eqref{eq015}, we obtain
\[
    \langle P_g \overline{Z}_{(\xi,\varepsilon)}, \varphi \rangle_{L^2(M,g)}
    - c_n \int_{M} 
        W_{(\xi,\varepsilon,r)}^{\frac{8}{n-4}} 
        \overline{Z}_{(\xi,\varepsilon)} \, \varphi \, dv_g
    = \int_{M} f \, \varphi \, dv_g,
\]
for all 
\(\varphi \in \mathcal{F}_{(\xi,\varepsilon,\alpha,r)}^{\perp}(M,g)\),
where
    \begin{align*}
        f & = (P_g-P_{g_s})Z_{(\xi,\varepsilon)}-\sum_{t=1}^\ell\left(P_g\left(\eta_{(r_t,\xi_t)}z_{(\xi,\varepsilon)}\right)-\eta_{(r_t,\xi_t)}\Delta^2z_{(\xi_t,\varepsilon_t)}\right) +P_g\mathcal Z_{(\xi,\varepsilon)}-c_nW_{(\xi,\varepsilon,r)}^{\frac{8}{n-4}} \mathcal Z_{(\xi,\varepsilon)}.
    \end{align*}
Since \(\overline Z_{(\xi,\varepsilon)} \in 
\mathcal{F}_{(\xi,\varepsilon,\alpha,r)}^\perp(M,g)\), 
Theorem~\ref{teo005} implies that $\|\overline Z_{(\xi,\varepsilon)}\|_{W^{2,2}(M,g)}    \le c\, \|f\|_{L^{\frac{2n}{n+4}}(M,g)}$.

By \eqref{eq013}, the metrics \(g\) and \(g_s\) coincide on 
\(M \setminus B_s(p)\). 
Inside \(B_s(p)\), we have \(g = \exp(h)\), while \(g_s\) is the Euclidean 
metric (see \eqref{eq035}).  
Consequently, the support of the operator \(P_g - P_{g_s}\) is contained 
in \(B_s(p)\), and we may write
\[
\begin{aligned}
(P_g - P_{g_s}) Z_{(\xi,\varepsilon)}
&= (\Delta_g^2 - \Delta^2) Z_{(\xi,\varepsilon)} \\
&\quad
   + \operatorname{div}\!\left( 
        a(n)\, \operatorname{Ric}_g(\nabla_g Z_{(\xi,\varepsilon)})
        - b(n)\, R_g\, dZ_{(\xi,\varepsilon)}
     \right)
   + c(n)\, Q_g\, Z_{(\xi,\varepsilon)}.
\end{aligned}
\]

Using \eqref{eq069} and Lemma~\ref{lem009}, for all \(|x| \le s\), we obtain
\begin{equation}\label{eq094}
    |(P_g - P_{g_s}) Z_{(\xi,\varepsilon)}|
    \;\le\;
    \sum_{t=1}^{\ell}
        \mu_t^{2} \lambda_t^{\frac{n-4}{2}}
        \bigl( \lambda_t + |x - y_t| \bigr)^{20 - n}
    \;\le\;
    \sum_{t=1}^{\ell}
        \mu_t^{2} \lambda_t^{\frac{n-4}{2}}
        \bigl( \lambda_t + |x - y_t| \bigr)^{10\frac{n-3}{n-4} - n}.
\end{equation}

Next, note that $P_g\bigl( \eta_{(r_t,\xi_t)} z_{(\xi,\varepsilon)} \bigr)    - \eta_{(r_t,\xi_t)} \Delta^2 z_{(\xi_t,\varepsilon_t)}    = 0$ in $M \setminus B_{2r_t}(\xi_t)$, while in \(B_{2r_t}(\xi_t)\) we have
\[
\begin{aligned}
\bigl|
P_g\bigl( \eta_{(r_t,\xi_t)} z_{(\xi,\varepsilon)} \bigr)
- \eta_{(r_t,\xi_t)} \Delta^2 z_{(\xi_t,\varepsilon_t)}
\bigr|
&\le
C \mu_t^{2} \lambda_t^{\frac{n-4}{2}}
  (\lambda_t + |x - y_t|)^{20 - n}
+ \mu_t \lambda_t^{\frac{n-4}{2}} \, r_t^{\,11-n},
\end{aligned}
\]
where the second term on the right-hand side appears only in the annular region  
\( B_{2r_t}(\xi_t) \setminus B_{r_t}(\xi_t)\). This implies that
\[
\begin{aligned}
&\Bigl\|
   (P_g - P_{g_s}) Z_{(\xi,\varepsilon)}
   - \sum_{t=1}^{\ell}
     \bigl( 
        P_g(\eta_{(r_t,\xi_t)} z_{(\xi,\varepsilon)})
        - \eta_{(r_t,\xi_t)} \Delta^2 z_{(\xi_t,\varepsilon_t)}
     \bigr)
 \Bigr\|_{L^{\frac{2n}{n+4}}(M)}
\\
&\qquad\le
C \sum_{t=1}^{\ell}
   \left(
      \mu_t^{2} \lambda_t^{\,10\frac{n-3}{n+4}}
      + \mu_t \left( \frac{\lambda_t}{\rho_t} \right)^{\frac{n-4}{2}}
   \right).
\end{aligned}
\]

Finally, by \eqref{eq055} we obtain the desired result.
\end{proof}

\begin{proposition}\label{propo008}
    If $(\xi,\varepsilon)\in\Omega(\lambda,y)$, then
    $$\|U_{(\xi,\varepsilon,r)}-W_{(\xi,\varepsilon,r)}-Z_{(\xi,\varepsilon,r)}\|_{L^{\frac{2n}{n-4}}(M)}\leq\sum_{t=1}^\ell\left(\left(\frac{\lambda_t}{\rho_t}\right)^{\frac{n-4}{2}}+\mu_t^{\frac{n+4}{n-4}}\lambda_t^{10\frac{n+4}{n-4}}\right)$$
\end{proposition}
\begin{proof}
Given \((\xi,\varepsilon)\in\Omega(\lambda,y)\), let $G_{(\xi,\varepsilon)}:L^{\frac{2n}{\,n+4\,}}(M,g)\to
\mathcal{F}_{(\xi,\varepsilon,\alpha,r)}^{\perp}(M,g)$ be the solution operator constructed in Theorem~\ref{teo005}.  
Set $B_1 := (P_g - P_{g_s})Z_{(\xi,\varepsilon,r)}$ and $B_2 := \sum_{t=1}^{\ell}\eta_{(r_t,\xi_t)}\,\Gamma_{(y_t,\xi_t,\varepsilon_t)}$. Then, by \eqref{eq083}, we obtain
\[
\left\langle P_g Z_{(\xi,\varepsilon)},\,\varphi \right\rangle_{L^2}
-\frac{n+4}{n-4}\,d(n)\int_M 
w_{(\xi,\varepsilon)}^{\frac{8}{n-4}} Z_{(\xi,\varepsilon)} \varphi\, dv_g
= \int_M (B_1 - B_2)\,\varphi\, dv_g,
\]
for every \(\varphi\in\mathcal{F}_{(\xi,\varepsilon,\alpha,r)}^{\perp}(M,g)\). Since \(Z_{(\xi,\varepsilon)} \in 
\mathcal{F}_{(\xi,\varepsilon,\alpha,r)}^{\perp}(M,g)\), 
it follows that $Z_{(\xi,\varepsilon)} = 
G_{(\xi,\varepsilon)}(B_1 - B_2)$. Furthermore, by Theorem~\ref{teo006} we obtain $U_{(\xi,\varepsilon,r)} - W_{(\xi,\varepsilon,r)}
= G_{(\xi,\varepsilon)}\bigl(-B_3 + d(n) B_4 \bigr)$, where
\begin{align*}
    B_3 & = P_gW_{(\xi,\varepsilon,r)}-d(n)W_{(\xi,\varepsilon,r)}^{\frac{n+4}{n-4}},\\
    B_4& = |U_{(\xi,\varepsilon,r)}|^{\frac{8}{n-4}}U_{(\xi,\varepsilon,r)}-W_{(\xi,\varepsilon,r)}^{\frac{n+4}{n-4}}-\frac{n+4}{n-4}W_{(\xi,\varepsilon,r)}^{\frac{8}{n-4}}(U_{(\xi,\varepsilon,r)}-W_{(\xi,\varepsilon,r)}).
\end{align*}
Thus, we conclude that $U_{(\xi,\varepsilon,r)}-W_{(\xi,\varepsilon,r)}-Z_{(\xi,\varepsilon,r)}=G_{(\xi,\varepsilon)}(B_1-B_2-B_3+d(n)B_4)$. The Theorem \ref{teo005} and Sobolev inequality imply that
$$\|U_{(\xi,\varepsilon,r)}-W_{(\xi,\varepsilon,r)}-Z_{(\xi,\varepsilon,r)}\|_{L^{\frac{2n}{n-4}}(M)}\leq C\|B_1-B_2-B_3+d(n)B_4\|_{L^{\frac{2n}{n+4}}(M)}.$$
We have the pointwise estimate $|B_4| \le C\, 
\bigl| U_{(\xi,\varepsilon,r)} - W_{(\xi,\varepsilon,r)} \bigr|^{\frac{n+4}{\,n-4\,}}$. Combining this with Corollary~\ref{cor001}, we obtain
\[
\|B_4\|_{L^{\frac{2n}{\,n+4\,}}(M)}
\le 
C\,\bigl\| U_{(\xi,\varepsilon,r)} - W_{(\xi,\varepsilon,r)} 
\bigr\|_{L^{\frac{2n}{\,n-4\,}}(M)}^{\frac{n+4}{\,n-4\,}}
\le 
C \sum_{t=1}^{\ell}
\left[
\left(\frac{\lambda_t}{\rho_t}\right)^{\frac{n+4}{2}}
+
\mu_t^{\frac{n+4}{\,n-4\,}}
\lambda_t^{\,10\frac{n-3}{\,n-4\,}}
\right].
\]
Together with \eqref{eq091} and \eqref{eq094}, this yields the desired result.
\end{proof}

\begin{proposition}\label{propo009}
    For all $(\xi,\varepsilon)\in\Omega(\lambda,y)$, we have
    \begin{align*}
        &\left| \langle P_gU_{(\xi,\varepsilon,r)},U_{(\xi,\varepsilon,r)}\rangle_{L^2}-\langle P_gW_{(\xi,\varepsilon,r)},W_{(\xi,\varepsilon,r)}\rangle_{L^2} -d(n)\int_M\left(|U_{(\xi,\varepsilon,r)}|^{\frac{2n}{n-4}}-W_{(\xi,\varepsilon,r)}^{\frac{2n}{n-4}}\right)\right.\\
        & \left.+d(n)\int_M\left(|U_{(\xi,\varepsilon,r)}|^{\frac{8}{n-4}}-W_{(\xi,\varepsilon,r)}^{\frac{8}{n-4}}\right)U_{(\xi,\varepsilon,r)}W_{(\xi,\varepsilon,r)} -\sum_{t=1}^\ell\int_{\mathbb R^n}\Gamma_{(y_t,\xi_t,\varepsilon_t)}z_{(\xi_t,\varepsilon_t)}\right|\leq \\
        & \le C\sum_{t=1}^\ell\left(\left(\frac{\lambda_t}{\rho_t}\right)^{n-4}+\mu_t\lambda_t^{10}\left(\frac{\lambda_t}{\rho_t}\right)^{\frac{n-4}{2}}+\mu_t^{\frac{2n}{n-4}}\lambda_t^{10\frac{2n}{n-4}}\right)
    \end{align*}
\end{proposition}
\begin{proof}
Recall the definition of \(\Gamma_{(y_t,\xi_t,\varepsilon_t)}\) in \eqref{eq099}.  
Using Theorem~\ref{teo006} with  
\(\varphi = U_{(\xi,\varepsilon,r)} - W_{(\xi,\varepsilon,r)} 
\in \mathcal{F}_{(\xi,\varepsilon,\alpha,r)}^{\perp}(M,g)\), 
we obtain
\begin{align*}
&\bigl\langle P_g U_{(\xi,\varepsilon,r)}, U_{(\xi,\varepsilon,r)} \bigr\rangle_{L^2}
-
\bigl\langle P_g W_{(\xi,\varepsilon,r)}, W_{(\xi,\varepsilon,r)} \bigr\rangle_{L^2}
-
d(n)\!\int_M
\bigl( |U_{(\xi,\varepsilon,r)}|^{\frac{2n}{n-4}}
      - W_{(\xi,\varepsilon,r)}^{\frac{2n}{n-4}}
\bigr)\, dv_g \\[0.3em]
&\quad
+ d(n)\!\int_M
\bigl( |U_{(\xi,\varepsilon,r)}|^{\frac{8}{n-4}}
      - W_{(\xi,\varepsilon,r)}^{\frac{8}{n-4}}
\bigr)
U_{(\xi,\varepsilon,r)} W_{(\xi,\varepsilon,r)}\, dv_g
\\[0.3em]
&=
\int_M 
\bigl( P_g W_{(\xi,\varepsilon,r)} 
      - d(n)W_{(\xi,\varepsilon,r)}^{\frac{n+4}{n-4}} \bigr)
\bigl( U_{(\xi,\varepsilon,r)} - W_{(\xi,\varepsilon,r)} \bigr)\, dv_g .
\end{align*}

Hence,
\begin{align*}
&\Biggl|
\bigl\langle P_g U_{(\xi,\varepsilon,r)}, U_{(\xi,\varepsilon,r)} \bigr\rangle_{L^2}
-
\bigl\langle P_g W_{(\xi,\varepsilon,r)}, W_{(\xi,\varepsilon,r)} \bigr\rangle_{L^2}
-
d(n)\!\int_M
\bigl( |U_{(\xi,\varepsilon,r)}|^{\frac{2n}{n-4}}
     - W_{(\xi,\varepsilon,r)}^{\frac{2n}{n-4}}
\bigr)
\\
&\qquad\quad
+ d(n)\!\int_M
\bigl( |U_{(\xi,\varepsilon,r)}|^{\frac{8}{n-4}}
      - W_{(\xi,\varepsilon,r)}^{\frac{8}{n-4}}
\bigr)
U_{(\xi,\varepsilon,r)} W_{(\xi,\varepsilon,r)}\, dv_g
-
\sum_{t=1}^{\ell}
\int_{\mathbb{R}^n}
\Gamma_{(y_t,\xi_t,\varepsilon_t)}\, z_{(\xi_t, \varepsilon_t)}dx
\Biggr|
\\[0.4em]
&\le 
C\,\Bigl\|
P_g W_{(\xi,\varepsilon,r)}
- d(n)W_{(\xi,\varepsilon,r)}^{\frac{n+4}{n-4}}
+ \sum_{t=1}^{\ell}\eta_{(r_t,\xi_t)} \Gamma_{(y_t,\xi_t,\varepsilon_t)}
\Bigr\|_{L^{\frac{2n}{n+4}}(M)}
\,
\bigl\| U_{(\xi,\varepsilon,r)} - W_{(\xi,\varepsilon,r)} 
\bigr\|_{L^{\frac{2n}{n-4}}(M)}
\\[0.5em]
&\quad
+ \sum_{t=1}^{\ell}
\left|
\int_M
\eta_{(r_t,\xi_t)}\Gamma_{(y_t,\xi_t,\varepsilon_t)}
\bigl(
U_{(\xi,\varepsilon,r)} - W_{(\xi,\varepsilon,r)}
- z_{(\xi_t,\varepsilon_t)}
\bigr)\, dv_g
\right|
\\[0.4em]
&\quad
+ \sum_{t=1}^{\ell}
\left|
\int_{\mathbb{R}^n}
\bigl(\eta_{(r_t,\xi_t)} - 1\bigr)\Gamma_{(y_t,\xi_t,\varepsilon_t)}
\, z_{(\xi_t,\varepsilon_t)}\, dx
\right|.
\end{align*}

By \eqref{eq089}, \eqref{eq027}, \eqref{eq069}, and Proposition~\ref{propo008},
\begin{align*}
&\left|
\int_M
\eta_{(r_t,\xi_t)}\Gamma_{(y_t,\xi_t,\varepsilon_t)}
\bigl(
U_{(\xi,\varepsilon,r)} - W_{(\xi,\varepsilon,r)}
- z_{(\xi_t,\varepsilon_t)}
\bigr)\, dv_g
\right|
\\
&\le
\left|
\int_M
\eta_{(r_t,\xi_t)}\Gamma_{(y_t,\xi_t,\varepsilon_t)}
\bigl(
U_{(\xi,\varepsilon,r)} - W_{(\xi,\varepsilon,r)}
- Z_{(\xi,\varepsilon)}
\bigr)\, dv_g
\right|
+
\left|
\int_M
\eta_{(r_t,\xi_t)}\Gamma_{(y_t,\xi_t,\varepsilon_t)}
\bigl(
Z_{(\xi,\varepsilon)}
- z_{(\xi_t,\varepsilon_t)}
\bigr)\, dv_g
\right|
\\[0.4em]
&\le
C\,\|\Gamma_{(y_t,\xi_t,\varepsilon_t)}\|_{L^{\frac{2n}{n+4}}(\mathbb{R}^n)}
\,
\bigl\| U_{(\xi,\varepsilon,r)} - W_{(\xi,\varepsilon,r)} - Z_{(\xi,\varepsilon)} 
\bigr\|_{L^{\frac{2n}{n-4}}(M)}\\
& \;\;\;\; + \mu_t^2\lambda_t^{n-4}
\int_{B_{2r_t}(\xi_t)}
(\lambda_t + |x-y_t|)^{\,24-2n}\, dx\\
&\le
\sum_{t=1}^{\ell}
\left(
\mu_t\,\lambda_t^{10}\Bigl(\frac{\lambda_t}{\rho_t}\Bigr)^{\frac{n-4}{2}}
+
\mu_t^{\frac{2n}{n-4}}\lambda_t^{10\frac{2n}{n-4}}
\right)
\end{align*}
and
\begin{align*}
    \left|
\int_{\mathbb{R}^n}
\bigl(\eta_{(r_t,\xi_t)} - 1\bigr)\Gamma_{(y_t,\xi_t,\varepsilon_t)}\,
z_{(\xi_t,\varepsilon_t)}
\right|
\le
C\,\mu_t^2\lambda_t^{n-4}
\int_{\mathbb{R}^n\setminus B_{r_t}(\xi_t)}
(\lambda_t + |x-y_t|)^{24 - 2n}
\le
C\,\mu_t^2\Bigl(\tfrac{\lambda_t}{\rho_t}\Bigr)^{n-4}.
\end{align*}
The result now follows from \eqref{eq091} and Corollary~\ref{cor001}.
\end{proof}

\begin{proposition}\label{propo010}
We have
\begin{align*}
\int_M\left||U_{(\xi,\varepsilon,r)}|^{\frac{2n}{n-4}}-W_{(\xi,\varepsilon,r)}^{\frac{2n}{n-4}}-\frac{n}{4}\left(|U_{(\xi,\varepsilon,r)}|^{\frac{8}{n-4}}-W_{(\xi,\varepsilon,r)}^{\frac{8}{n-4}}\right) U_{(\xi,\varepsilon,r)} W_{(\xi,\varepsilon,r)}\right|\\
\leq c(n,\ell)\sum_{t=1}^{\ell}\left(\left(\frac{\lambda_t}{\rho_t}\right)^{n}+\mu_t^{\frac{2n}{n-4}}\lambda_t^{\frac{20n}{n-4}}\right). 
\end{align*}
\end{proposition}
\begin{proof}
We have the pointwise estimate
\[
\left|
\,|U_{(\xi,\varepsilon,r)}|^{\frac{2n}{n-4}}
-
W_{(\xi,\varepsilon,r)}^{\frac{2n}{n-4}}
-
\frac{n}{4}\Bigl(
|U_{(\xi,\varepsilon,r)}|^{\frac{8}{n-4}}
-
W_{(\xi,\varepsilon,r)}^{\frac{8}{n-4}}
\Bigr)
U_{(\xi,\varepsilon,r)} W_{(\xi,\varepsilon,r)}
\right|
\le
C\,\bigl|U_{(\xi,\varepsilon,r)} - W_{(\xi,\varepsilon,r)}\bigr|^{\frac{2n}{n-4}},
\]
where \(C>0\) depends only on \(n\).  
Combining this with Corollary~\ref{cor001}, we obtain
\begin{align*}
\int_M\left|\,|U_{(\xi,\varepsilon,r)}|^{\frac{2n}{n-4}}-W_{(\xi,\varepsilon,r)}^{\frac{2n}{n-4}}-\frac{n}{4}\Bigl(|U_{(\xi,\varepsilon,r)}|^{\frac{8}{n-4}}-W_{(\xi,\varepsilon,r)}^{\frac{8}{n-4}}\Bigr)U_{(\xi,\varepsilon,r)} W_{(\xi,\varepsilon,r)}\right|\, dv_g\\[0.4em]\le\bigl\|U_{(\xi,\varepsilon,r)} - W_{(\xi,\varepsilon,r)}\bigr\|_{L^{\frac{2n}{n-4}}(M)}^{\frac{2n}{n-4}}\le c(n,\ell)\sum_{t=1}^{\ell}\left(\left(\frac{\lambda_t}{\rho_t}\right)^{n}+\mu_t^{\frac{2n}{n-4}}\lambda_t^{\frac{20n}{n-4}}\right).
\end{align*}
\end{proof}

\begin{proposition}\label{propo011}
We have
    \begin{align*}
        \left|\langle P_gW_{(\xi,\varepsilon,r)},W_{(\xi,\varepsilon,r)}\rangle_{L^2}-d(n)\int_MW_{(\xi,\varepsilon,r)}^{\frac{2n}{n-4}}-\sum_{t=1}^\ell\int_{B_{\rho_t}(y_t)}\Delta_t\right|\\
        \leq C \sum_{t=1}^\ell \left(\mu_t^3\lambda_t^{20}+\left(\frac{\lambda_t}{\rho_t}\right)^{n-4}\right),
    \end{align*}
    where
        \begin{align*}
        \Delta_t & = \frac{a(n)}{4}\partial_jh_{ms}\partial_ih_{sm}\partial_iw_{(\xi_t,\varepsilon_t)}\partial_jw_{(\xi_t,\varepsilon_t)}+\frac{n-4}{8(n-1)}\left((\partial_i\partial_lh_{mk})^2+\partial_lh_{mk}\partial_i^2\partial_lh_{mk}\right)w_{(\xi_t,\varepsilon_t)}^2\\
        & +\frac{a(n)}{2}\left(h_{ms}\partial_sh_{ij}-h_{si}\partial_sh_{mj}+h_{sj}\partial_ih_{ms}-h_{ms}\partial_ih_{sj}\right)\partial_m\left(\partial_iw_{(\xi_t,\varepsilon_t)}\partial_jw_{(\xi_t,\varepsilon_t)}\right)\\
        & -h_{il}h_{jl}\partial_i\partial_k^2w_{(\xi_t,\varepsilon_t)}\partial_jw_{(\xi_t,\varepsilon_t)}+\left(h_{ij}\partial_i\partial_jw_{(\xi_t,\varepsilon_t)}\right)^2-\frac{b(n)}{4}(\partial_lh_{mk})^2(\partial_iw_{(\xi_t,\varepsilon_t)})^2\\
        & -\frac{a(n)}{2}h_{is}\partial_m^2h_{js}\partial_iw_{(\xi_t,\varepsilon_t)}\partial_jw_{(\xi_t,\varepsilon_t)}-\frac{n-4}{4(n-2)^2}\partial_m^2h_{ij}\partial_s^2h_{ij}w_{(\xi_t,\varepsilon_t)}^2.
    \end{align*}
\end{proposition}
\begin{proof}
  Recall from \eqref{eq050} that $\operatorname{supp} W_{(\xi,\varepsilon,r)}
\subset \bigcup_{t=1}^{\ell} B_{2r_t}(\xi_t)
\subset B_{2R}(p)\subset B_{s}(p)$, where on this region the metric satisfies \(g=\exp(h)\).  
Moreover, we have \(B_{\rho_t}(y_t)\subset B_{r_t}(\xi_t)\).

Since \(h \equiv 0\) for all \(|x|\geq R\), using a computation analogous to 
\cite[Lemma 4.11]{MR3016505} (see also \cite[Lemma 4.11]{wei2011noncompactnessprescribedqcurvatureproblem}), we obtain
\begin{align*}
\int_M (\Delta_g W_{(\xi_t,\varepsilon_t,r)})^2\, dv_g 
&= \sum_{t=1}^\ell \int_{B_{\rho_t}(y_t)} (\Delta w_{(\xi_t,\varepsilon_t)})^2
 - \sum_{t=1}^\ell \int_{B_{\rho_t}(y_t)} 
   h_{il} h_{jl}\, (\partial_i\partial_k^2 w_{(\xi_t,\varepsilon_t)})\, \partial_j w_{(\xi_t,\varepsilon_t)}
\\
&\quad 
+ \sum_{t=1}^\ell \int_{B_{\rho_t}(y_t)}
   \big(h_{ij}\partial_i\partial_j w_{(\xi_t,\varepsilon_t)}\big)^{\!2}
   + O(\mu_t^3\lambda_t^{20})
   + O\!\left(\left(\frac{\lambda_t}{\rho_t}\right)^{n-4}\right).
\end{align*}

As in Lemmas 4.13 and 4.14 of \cite{MR3016505}, we also obtain
\[
\int_{M} R_g\, |\nabla_g W_{(\xi_t,\varepsilon_t,r)}|\, dv_g
= 
-\frac{1}{4}\sum_{t=1}^\ell 
\int_{B_{\rho_t}(y_t)}
 (\partial_l h_{mk})^2\, (\partial_i w_{(\xi_t,\varepsilon_t)})^2
+ O(\mu_t^3\lambda_t^{20})
+ O\!\left(\alpha^2\left(\frac{\lambda_t}{\rho_t}\right)^{n-4}\right),
\]
and
\begin{align*}
&\int_{M} \operatorname{Ric}_g(\nabla_g W_{(\xi,\varepsilon,r)}, \nabla_g W_{(\xi,\varepsilon,r)}) 
= 
-\frac{1}{4}\sum_{t=1}^\ell 
  \int_{B_{\rho_t}(y_t)} 
    \partial_j h_{ms}\,\partial_i h_{sm}\,
    \partial_i w_{(\xi_t,\varepsilon_t)}\, 
    \partial_j w_{(\xi_t,\varepsilon_t)}
\\
&\quad 
-\frac{1}{2}\sum_{t=1}^\ell 
\int_{B_{\rho_t}(y_t)}
 \big(h_{ms}\partial_s h_{ij}
     - h_{si}\partial_s h_{mj}
     + h_{sj}\partial_i h_{ms}
     - h_{ms}\partial_i h_{sj}\big)
\partial_m\big(\partial_i w_{(\xi_t,\varepsilon_t)}\,\partial_j w_{(\xi_t,\varepsilon_t)}\big)
\\
&\quad
+\frac{1}{2}\sum_{t=1}^\ell 
\int_{B_{\rho_t}(y_t)}
   h_{is}\,\partial_m^2 h_{js}\,
   \partial_i w_{(\xi_t,\varepsilon_t)}\,
   \partial_j w_{(\xi_t,\varepsilon_t)}
+ O(\mu_t^3\lambda_t^{20})
+ O\!\left(\alpha\left(\frac{\lambda_t}{\rho_t}\right)^{n-4}\right).
\end{align*}

By Lemma \ref{lem009} we further obtain
\begin{align*}
\int_M & Q_g\, W_{(\xi,\varepsilon,r)}^2\, dv_g 
= \frac{1}{4(n-1)}
\sum_{t=1}^\ell \int_{B_{\rho_t}(y_t)}
   \Big( (\partial_i\partial_l h_{mk})^2
       + \partial_l h_{mk}\,\partial_i^2\partial_l h_{mk} \Big)
   w_{(\xi,\varepsilon)}^2
\\
&\quad
- \frac{1}{2(n-2)^2}
\sum_{t=1}^\ell
\int_{B_{\rho_t}(y_t)}
   \partial_m^2 h_{ij}\, \partial_s^2 h_{ij}\,
   w_{(\xi_t,\varepsilon_t)}^2
+ O\!\left(\mu_t^2\left(\frac{\lambda_t}{\rho_t}\right)^{n-4}\right)
+ O(\mu_t^3\lambda_t^{20}).
\end{align*}

Finally, since \(w_{(\xi_t,\varepsilon_t)}\) solves \eqref{eq012} and \eqref{eq066} holds, we conclude that
\[
\sum_{t=1}^\ell 
\int_{B_{\rho_t}(y_t)} (\Delta w_{(\xi_t,\varepsilon_t)})^2
- d(n)\int_M W_{(\xi,\varepsilon,r)}^{\frac{2n}{n-4}}
=
O\!\left(\left(\frac{\lambda_t}{\rho_t}\right)^{n-4}\right).
\]

This completes the proof.
\end{proof}

Finally, we prove the main proposition of this subsection.

\begin{proposition}\label{propo012}
    If $(\xi,\varepsilon)\in\Omega(\lambda,y)$, then
    \begin{align*}
          \left|\mathcal F_g(\xi,\varepsilon)-\sum_{t=1}^\ell\int_{B_{\rho_t}(y_t)}(\Delta_t+\Gamma_{(y_t,\xi_t,\varepsilon_t)}z_{(\xi_t,\varepsilon_t)})\right|\\
          \leq C\sum_{t=1}^\ell\left(\left(\frac{\lambda_t}{\rho_t}\right)^{n-4}+\mu_t\lambda_t^{10}\left(\frac{\lambda_t}{\rho_t}\right)^{\frac{n-4}{2}}+\mu_t^{\frac{2n}{n-4}}\lambda_t^{10\frac{2n}{n-4}}\right),
    \end{align*}
    where $\Delta_t$ is defined in Proposition \ref{propo011} and $\Gamma_{(y_t,\xi_t,\varepsilon_t)}$ in \eqref{eq099}.
\end{proposition}
\begin{proof}
    Recall the definition of \(\mathcal{F}_g\) given in \eqref{eq082}. Using \eqref{eq005} we obtain
    
    \begin{align*}
        &\left|\mathcal F_g(\xi,\varepsilon)-\sum_{t=1}^\ell\int_{B_{\rho_t}(y_t)}(\Delta_{t}+\Gamma_{(y_t,\xi_t,\varepsilon_t)}z_{(\xi_t,\varepsilon_t)})\right|  \\
        &\leq\left| \langle P_gU_{(\xi,\varepsilon,r)},U_{(\xi,\varepsilon,r)}\rangle_{L^2}-\langle P_gW_{(\xi,\varepsilon,r)},W_{(\xi,\varepsilon,r)}\rangle_{L^2}-d(n)\int_M\left(|U_{(\xi,\varepsilon,r)}|^{\frac{2n}{n-4}}-W_{(\xi,\varepsilon,r)}^{\frac{2n}{n-4}}\right) \right.\\
        & +d(n)\int_M\left(|U_{(\xi,\varepsilon,r)}|^{\frac{8}{n-4}}-W_{(\xi,\varepsilon,r)}^{\frac{8}{n-4}}\right)U_{(\xi,\varepsilon,r)}W_{(\xi,\varepsilon,r)} \left.-\sum_{t=1}^\ell\int_{B_{\rho_t}(y_t)}\Gamma_{(y_t,\xi_t,\varepsilon_t)}z_{(\xi_t,\varepsilon_t)}\right|\\
        & +\left|\langle P_gW_{(\xi,\varepsilon,r)},W_{(\xi,\varepsilon,r)}\rangle_{L^2}-d(n)\int_MW_{(\xi,\varepsilon,r)}^{\frac{2n}{n-4}}-\sum_{t=1}^\ell\int_{B_{\rho_t}(y_t)}\Delta_{t}\right|\\
        & +\frac{4}{n}d(n)\int_M\left||U_{(\xi,\varepsilon,r)}|^{\frac{2n}{n-4}}-W_{(\xi,\varepsilon,r)}^{\frac{2n}{n-4}}-\frac{n}{4}\left(|U_{(\xi,\varepsilon,r)}|^{\frac{8}{n-4}}-W_{(\xi,\varepsilon,r)}^{\frac{8}{n-4}}\right) U_{(\xi,\varepsilon,r)} W_{(\xi,\varepsilon,r)}\right|\\
        &+\frac{4}{n}d(n)\left(\int_MW_{(\xi,\varepsilon,r)}^{\frac{2n}{n-4}}-\sum_{t=1}^\ell\int_{\mathbb R^n}w_{(\xi_t,\varepsilon_t)}^{\frac{2n}{n-4}}\right).
    \end{align*}
    One readily verifies that
    $$\left|\int_MW_{(\xi,\varepsilon,r)}^{\frac{2n}{n-4}}-\sum_{t=1}^\ell\int_{\mathbb R^n}w_{(\xi_t,\varepsilon_t)}^{\frac{2n}{n-4}}\right|\leq C\sum_{t=1}^\ell\left(\frac{\lambda_t}{\rho_t}\right)^n.$$
    Therefore, the result follows by Propositions \ref{propo009}, \ref{propo010} and \ref{propo011}.
\end{proof}

\subsection{The reduced energy functional}

Consider the reduced energy functional $F: \mathbb{R}^N \times(0, \infty) \rightarrow \mathbb{R}$ defined in \cite[Section 9]{MR3016505}, and given as follows: for a given pair $\left(\xi, \lambda\right) \in \mathbb{R}^n \times(0, \infty)$,
the reduced energy is defined as
$$
\begin{aligned}
F\left(\xi, \lambda\right)= & -\int_{\mathbb{R}^n}\overline{H}_{il} \overline{H}_{jl}\partial_i \partial_k^2 w_{(\xi,\lambda)}\partial_j w_{(\xi,\lambda)}+\int_{\mathbb{R}^n}\left(\overline{H}_{i j}\partial_i \partial_j w_{(\xi,\lambda)}\right)^2  -\frac{b(n)}{4} \int_{\mathbb{R}^n} \left(\partial_{l} \overline{H}_{m k}\right)^2\left(\partial_i w_{(\xi,\lambda)}\right)^2 \\
& +\frac{a(n)}{2} \int_{\mathbb{R}^n} \left(\overline{H}_{m s}\partial_s \overline{H}_{i j}-\overline{H}_{s i}\partial_s \overline{H}_{m j}+\overline{H}_{s j}\partial_i \overline{H}_{m s}-\overline{H}_{m s}\partial_i \overline{H}_{s j}\right)\partial_m\left(\partial_i w_{(\xi,\lambda)} \partial_j w_{(\xi,\lambda)}\right) \\
& +\frac{a(n)}{4} \int_{\mathbb{R}^n} \partial_j \overline{H}_{m s}\partial_i \overline{H}_{s m}\partial_i w_{(\xi,\lambda)}\partial_j w_{(\xi,\lambda)}-\frac{n-4}{4(n-2)^2} \int_{\mathbb{R}^n} \partial_m^2\overline{H}_{i j}\partial_s^2 \overline{H}_{i j}w_{(\xi,\lambda)}^2 \\
& +\frac{n-4}{8(n-1)} \int_{\mathbb{R}^n} \left(\left(\partial_i \partial_l \overline{H}_{m k}\right)^2+\partial_{l} \overline{H}_{m k}\partial_i^2 \partial_l \overline{H}_{m k}\right) w_{(\xi,\lambda)}^2\\
& +\frac{b(n)}{2} \int_{\mathbb{R}^n} \overline{H}_{i s}\partial_m^2 \overline{H}_{j s}\partial_i w_{(\xi,\lambda)}\partial_j w_{(\xi,\lambda)}+\int_{\mathbb{R}^n} \overline\Gamma_{(\xi,\varepsilon)}\overline z_{(\xi,\varepsilon)},
\end{aligned}
$$
where the tensor $\overline H$ is defined in \eqref{eq095},
\begin{align}
    \overline \Gamma_{(\xi,\varepsilon)} & =2\overline H_{kj}\partial_k\partial_j\partial_s^2w_{(\xi,\varepsilon)}+2\partial_s\overline H_{kj}\partial_k\partial_j\partial_sw_{(\xi,\varepsilon)}+\frac{n}{n-2}\partial_s^2\overline H_{kj}\partial_k\partial_jw_{(\xi,\varepsilon)}\nonumber +\frac{2}{n-2}\partial_j\partial_s^2\overline H_{kj}\partial_kw_{(\xi,\varepsilon)}.\nonumber\label{eq092}
\end{align}
and  $\overline z_{(\xi,\varepsilon)}\in \mathcal{E}_{(\xi, \varepsilon)} $ satisfies the relation
$$
    \int_{\mathbb R^n}\left( \Delta \overline z_{(\xi,\varepsilon)}\Delta\varphi-\frac{n+4}{n-4}d(n)w_{(\xi,\varepsilon)}^{\frac{8}{n-4}} \overline z_{(\xi,\varepsilon)}\varphi\right)=-\int_{\mathbb R^n}\overline \Gamma_{(\xi,\varepsilon)}\varphi,
$$
for all $\varphi\in\mathcal E_{(\xi,\varepsilon)}$, see Section \ref{sec004}.

Considering the definition in \eqref{eq099}, we observe that $\Gamma_{(y,\lambda\xi+y,\lambda\varepsilon)}(\lambda x+y)=\mu\lambda^{\frac{16-n}{2}}\overline \Gamma_{(\xi,\varepsilon)}(x)$. Combining this identity with the uniqueness of the solution to the corresponding linearized problem, we obtain
\begin{equation}\label{eq103}
z_{(\lambda \xi +y,\, \lambda \varepsilon)}(\lambda x + y)= \mu \, \lambda^{\frac{24-n}{2}}\, \overline z_{(\xi,\varepsilon)}(x),    
\end{equation}
for all $(y,\xi,\varepsilon)\in \mathbb R^n\times\mathbb R^n\times(0,\infty)$.

\section{Proof of the main theorem}\label{sec008}

In this section we prove the main result of this work, Theorem~\ref{teo008}. Recall that \((M,g_0)\) is a closed Riemannian manifold of dimension 
\(n \ge 25\) satisfying \(Y(M,g_0) > 0\) and \(Y_4(M,g_0) > 0\).

\begin{proposition}\label{propo013}
Fix $\ell\in\mathbb N$ and $p\in M$. For each $t\in\{1,\ldots,\ell\}$, Choose parameters $\lambda_t$, $\mu_t$ and $\rho_t$, and points $y_1,\ldots,y_\ell\in B_{R/2}(p)$ satisfying \eqref{eq087} and \eqref{eq088}. Let $g$ be the perturbed metric $g$ defined as
$$g(x)=\begin{cases}
        \exp(h(x)), & x \in B_s(p), \\[0.3em]
        g_s(x),     & x \in M \setminus B_s(p),
    \end{cases}$$
    where $g_s$ is defined in \eqref{eq035}, with $s>0$ satisfying Theorem \ref{teo003}, $h(x)$ is a trace-free symmetric two-tensor on $\mathbb R^n$ such that
    $$|h(x)| + |\partial h(x)| + |\partial^2 h(x)|
    + |\partial^3 h(x)| + |\partial^4 h(x)|
    \;\leq\; \alpha < 1,$$
    for all $x\in \mathbb R^n$, $h(x)=0$ for $|x|\geq R$ and $h_{ik}(x)=\mu_{t}\,\lambda_{t}^{8}\,f\bigl(\lambda_{t}^{-2}| x-y_{t}|^{2}\bigr)\;H_{ik}(x - y_{t})$, if $|x-y_{t}| \le \rho_{t}$. If the parameter \(\alpha\) and $\mu_t^{-2}\rho_t^{4-n}\lambda_t^{n-24}$ are chosen sufficiently small, 
then there exists a positive function \(u\) on \(M\) such that
\begin{enumerate}
    \item[(a)] $P_{g} u 
    = d(n)
      u^{\frac{n+4}{n-4}}.$
    \item[(b)] the volume and the energy \eqref{eq101} satisfies the estimates
    $$\int_{M} u^{\frac{2n}{\,n-4\,}} 
    > C(n) \ell\qquad\mbox{ and }
    \qquad\mathcal{E}_g(u) \;\ge\; C(n)\, \ell^{\frac{4}{n}}.$$
    \item[(c)] for all $t = 1,\ldots,\ell$, it holds $\displaystyle\sup_{|x - y_t| \le \lambda_t} u(x)
    \;\ge\;
    C(n)\, \lambda_t^{\frac{4-n}{2}}$.
\end{enumerate}
\end{proposition}
\begin{proof}
Define
$$\Lambda=\left\{(\xi,\varepsilon)\in\mathbb R^n\times\mathbb R:|\xi|\leq 1\mbox{ and }\frac{1}{2}<\varepsilon<\frac{3}{2}\right\}.$$

By \cite[Proposition~11.6]{MR3016505}, the reduced energy functional \(F(\xi,\lambda)\) has a strict local minimum at \((0,1)\), with \(F(0,1)<0\), see \cite[Propositions~11.2 and~11.5]{MR3016505}. Consequently, there exists an open set \(\mathcal{P} \subset \Lambda\) such that \((0,1) \in \mathcal{P}\) and  
\begin{equation*}
    F(0,1) < \inf_{\partial\mathcal P}F(\xi,\lambda) <0.
\end{equation*}
Given $\lambda=(\lambda_1,\ldots,\lambda_\ell)\in(0,\infty)^n$ and $y=(y_1,\ldots,y_\ell)\in\mathbb R^{n\ell}$,  define
$$\mathcal P(\lambda,y)=\left\{(\xi,\varepsilon)\in(\lambda_1\mathcal P+(y_1,0))\times\cdots\times(\lambda_\ell\mathcal P+(y_\ell,0)):\frac{1}{2}<\frac{\varepsilon_i}{\varepsilon_j}<2\right\}\subset\Omega(\lambda,y).$$
Recall $\Omega(\lambda,y)$ in \eqref{eq102}. Consequently, it follows from Proposition~\ref{propo012}, together with \eqref{eq103}, that
\begin{equation*}
 \left|\mathcal{F}_g(\xi, \varepsilon)-\sum_{t=1}^{\ell} \mu_t^2\lambda_t^{20} F\left(\frac{\xi_t-y_t}{\lambda_t}, \frac{\varepsilon_t}{\lambda_t}\right)\right| \leq C\sum_{t=1}^\ell\left(\left(\frac{\lambda_t}{\rho_t}\right)^{n-4}+\mu_t\lambda_t^{10}\left(\frac{\lambda_t}{\rho_t}\right)^{\frac{n-4}{2}}+\mu_t^{\frac{2n}{n-4}}\lambda_t^{10\frac{2n}{n-4}}\right),
\end{equation*}
for all $(\xi,\varepsilon)\in\mathcal P(\lambda,y)$.
Define $J=\max\{\mu_t^2\lambda_t^{20}:t=1,\ldots,\ell\}$ and $J_t=\mu_t^2\lambda_t^{20}J^{-1}\in(0,1]$. Note that $J_i=1$, for some $i\in\{1,\ldots,\ell\}$. This implies that
\begin{align*}
    \left|J^{-1}\mathcal{F}_g(\lambda\xi+y, \lambda\varepsilon)-\sum_{t=1}^{\ell} J_t F\left(\xi_t, \varepsilon_t\right)\right| \leq C\sum_{t=1}^\ell\left(\mu_t^{-2}\rho_t^{4-n}\lambda_t^{n-24}+\mu_t^{-1}\rho_t^{\frac{4-n}{2}}\lambda_t^{\frac{n-24}{2}}+\mu_t^{\frac{8}{n-4}}\lambda_t^{\frac{80}{n-4}}\right),
\end{align*}
for all \((\lambda \xi + y, \lambda \varepsilon) \in \mathcal{P}(\lambda, y)\).  
Here $\lambda \xi = (\lambda_1 \xi_1, \ldots, \lambda_\ell \xi_\ell)$ and  $\lambda \varepsilon = (\lambda_1 \varepsilon_1, \ldots, \lambda_\ell \varepsilon_\ell)$. By \eqref{eq087}, we may choose \(\mu_t^{-2}\rho_t^{4-n}\lambda_t^{n-4}\) sufficiently small for every \(t \in \{1,\ldots,\ell\}\).  
With this choice we obtain
$$\mathcal F_g(y,\lambda)<\inf_{(\xi,\varepsilon)\in\partial\mathcal P(\lambda,y)}\mathcal F_g(\xi,\varepsilon)<0.$$
 This implies the existence of $(\xi_0,\varepsilon_0)\in\mathcal P$ such that 
$$\mathcal F_g(\lambda\xi_0+y,\lambda\varepsilon_0)=\inf_{(\xi,\varepsilon)\in\mathcal P(\lambda,y)}\mathcal F_g(\xi,\varepsilon)<0.$$
By Theorem \ref{teo007}, the function 
\(
U_{(\lambda\xi_0+y,\lambda\varepsilon_0,r)}
\)
is a nonnegative $C^4$ solution of the fourth-order equation on $(M,g)$,
\[
P_g U_{(\lambda\xi_0+y,\lambda\varepsilon_0,r)} = d(n)\, U_{(\lambda\xi_0+y,\lambda\varepsilon_0,r)}^{\frac{n+4}{\,n-4\,}}.
\]
By elliptic regularity theory (see, for instance, \cite{MR4228252,MR2667016}), the function $U_{(y+\lambda\xi_0,\lambda\varepsilon_0,r)}$ is in fact smooth.

By Proposition~\ref{propo002}, Theorem~\ref{teo006}, and the Sobolev inequality, we obtain
$$\|W_{(y+\lambda\xi_0,\lambda\varepsilon_0,r)}-U_{(y+\lambda\xi_0,\lambda\varepsilon_0,r)}\|_{L^{\frac{2n}{n-4}}(M,g)}\leq C\alpha.$$
Using \eqref{eq005} and \eqref{eq050}, we obtain $\| W_{(y + \lambda \xi_0,\, \lambda \varepsilon_0,\, r)}     \|_{L^{\frac{2n}{\,n-4\,}}(B_{r_t}(y_t + \lambda_t \xi_{0t}),g)}    \ge c(n)> 0,$ for all \(\alpha > 0\) sufficiently small and $t\in\{1,\ldots,\ell\}$.
Therefore, for \(\alpha > 0\) small enough, we obtain
\begin{align*}
    \|U_{(y+\lambda\xi_0,\lambda\varepsilon_0,r)}\|_{L^{\frac{2n}{n-4}}(M,g)} & \geq \|W_{(y+\lambda\xi_0,\lambda\varepsilon_0,r)}\|_{L^{\frac{2n}{n-4}}(M,g)}-\|W_{(y+\lambda\xi_0,\lambda\varepsilon_0,r)}-U_{(y+\lambda\xi_0,\lambda\varepsilon_0,r)}\|_{L^{\frac{2n}{n-4}}(M,g)}\\
    & \geq \|W_{(y+\lambda\xi_0,\lambda\varepsilon_0,r)}\|_{L^{\frac{2n}{n-4}}(M,g)}-C\alpha \geq\frac{1}{2}\|W_{(y+\lambda\xi_0,\lambda\varepsilon_0,r)}\|_{L^{\frac{2n}{n-4}}(M,g)}.
\end{align*}
This implies that 
\[
\int_M U_{(y+\lambda\xi_0,\lambda\varepsilon_0,r)}^{\frac{2n}{n-4}}\, dv_g
\;\geq\;
c(n)\int_M W_{(y+\lambda\xi_0,\lambda\varepsilon_0,r)}^{\frac{2n}{n-4}}\, dv_g
\;\geq\;
C(n)\,\ell^{}.
\]
Moreover, for each \(t\in\{1,\ldots,\ell\}\), we have
\begin{align*}
    \bigl(\operatorname{Vol}(B_{\lambda_t}(y_t))\bigr)^{\frac{n-4}{2n}}
    \sup_{B_{\lambda_t}(y_t)} U_{(y+\lambda\xi_0,\lambda\varepsilon_0,r)}
    &\;\geq\;
    \|U_{(y+\lambda\xi_0,\lambda\varepsilon_0,r)}\|_{L^{\frac{2n}{n-4}}(B_{\lambda_t}(y_t))} \\
    &\;\geq\;
    \|W_{(y+\lambda\xi_0,\lambda\varepsilon_0,r)}\|_{L^{\frac{2n}{n-4}}(B_{\lambda_t}(y_t))}
    - C(n)\,\alpha \;\geq\;
    C(n),
\end{align*}
for \(\alpha > 0\) sufficiently small.  
This implies, in particular, that 
\(U_{(y + \lambda \xi_0,\, \lambda \varepsilon_0,\, r)}\) is not identically zero.  
By \cite{MR3509928} and the maximum principle 
\cite[Theorem~A]{MR3420504}, we conclude that $U_{(y + \lambda \xi_0,\, \lambda \varepsilon_0,\, r)} > 0$. This finishes the proof.
\end{proof}

Finally, we can prove the main result of this work.

\begin{theorem}[Theorem \ref{teo008}]
Let \((M,g_0)\) be a closed Riemannian manifold of dimension \(n \ge 25\) with 
\(Y(M,g_0) > 0\) and \(Y_4(M,g_0) > 0\). Fix \(\varepsilon > 0\).  
For each pair \(k,\ell \in \mathbb{N}\), there exist a smooth Riemannian metric 
\(g\) on \(M\) and a smooth positive function \(U_{k,\ell}\) such that:
\begin{enumerate}
    \item[(a)] \(g\) is not conformally flat.

    \item[(b)] \(\|g - g_0\|_{C^1(M,g_0)} < \varepsilon\).

    \item[(c)] The $Q$-curvature of the metric $g_{k,\ell}=U_{k,\ell}^{\frac{4}{n-4}}g$ is constant equal to $n(n^2-4)/8$.

    \item[(d)] There exists a positive constant $c(n)$, depending only on $n$, such that the volume of the metric $g_{k,\ell}$ and its energy satisfy the estimates
\[
   \operatorname{Vol}\bigl(M, g_{k,\ell}\bigr) \ge c(n)\,\ell, \qquad \text{and} \qquad 
   \mathcal{E}\bigl(g_{k,\ell}\bigr) \ge c(n)\,\ell^{\frac{4}{n}}.
\]
\end{enumerate}
\end{theorem}
\begin{proof}
Consider the metric \(g\) defined in \eqref{eq013}, with \(s>0\) as in Theorem~\ref{teo005} and \(\alpha>0\) as in Theorem~\ref{teo007} and Proposition~\ref{propo013}. Let \(R \in (0, s/4)\). Define a trace-free symmetric two-tensor on \(\mathbb{R}^n\) by
\[
h_{ik}(x)= \sum_{N = N_0}^{\infty}\eta_{\left(\frac{1}{4N^2},\, \overline y_{N}\right)}(x)\,2^{-\frac{25}{3}N}\,f\!\left( 2^{2N} |x - \overline y_N| \right)\,H_{ik}\!\left(|x - \overline y_N|\right),
\]
where \(\overline y_N = \left(\frac{1}{N}, 0, \ldots, 0\right) \in \mathbb{R}^n\). Recall the definition of \(\eta_{(\frac{1}{4N^2}, \overline y_N)}\) from Section~\ref{sec005}. Since each point \(x \in \mathbb{R}^n\) belongs to only finitely many balls \(B_{1/(2 N^2)}(\overline y_N)\), it follows that the tensor \(h\) is smooth. If \(N_0\) is sufficiently large, then \(h\) satisfies \eqref{eq014} and \(h(x)=0\) for all \(|x| \ge R\).  Moreover, for every \(N \ge N_0\),
\[
    h_{ik}(x)
    =
    2^{-\frac{25}{3}N}\,
    f\!\left( 2^{2N} |x - \overline y_N| \right)
    H_{ik}(|x - \overline y_N|),
    \qquad 
    \text{for all } |x - \overline y_N| \le \frac{1}{4N^{2}}.
\]

For the choice $\lambda_N = 2^{-N}$, $\mu_N = 2^{-N/3}$ and $\rho_N = 1/4N^{2}$ the condition \eqref{eq087} is satisfied, provided \(N_0\) is sufficiently large.  
In addition,
\[
    \mu_N^{-2} \rho_N^{\,4-n} \lambda_N^{\,n-24}
    = 2^{\left(\frac{74}{3}-n\right)N}(4N^{2})^{\,n-4}
    \longrightarrow 0,
    \qquad\text{as } N \to \infty.
\]

Arguing by induction, we obtain a subsequence \((\overline y_{N_k})\) such that
\[
    \| \overline y_{N_i} - \overline y_{N_j} \|
    \ge 
    \frac{3}{2^{N_i}}
    + \frac{3}{2^{N_j}}
    + \frac{1}{2N_i^{2}}
    + \frac{1}{2N_j^{2}},
    \qquad \text{whenever } i \ne j.
\]

Finally, for every \(t,q \in \mathbb{N}\), set  $y_t := \overline y_{N_{t+q}}$, $\lambda_t := \lambda_{N_{t+q}}$, $\mu_t := \mu_{N_{t+q}}$ and $\rho_t := \rho_{N_{t+q}}$. Then this family of parameters satisfies conditions \eqref{eq087} and \eqref{eq088}. Moreover, for \(q\) sufficiently large, the quantity \(\mu_t^{-2} \rho_t^{\,4-n} \lambda_t^{\,n-24}\) is arbitrarily small.  Hence, Proposition~\ref{propo013} applies to the points \(\{y_1,\ldots,y_\ell\}\). This concludes the proof.
\end{proof}

\vspace{1cm}
\noindent\textbf{{Data Availability:}} Data sharing not applicable to this article as no datasets were generated or analyzed during the current manuscript.

\vspace{1cm}
\noindent\textbf{{Conflict of interest:}} The authors declare that there is no conflict of interest regarding the publication of this article.

\bibliography{References}
\bibliographystyle{abbrv}
    
\end{document}